\documentclass{amsart}

\usepackage{amssymb}
\usepackage[toc]{appendix}
\usepackage{array}
\usepackage{caption}
\usepackage{dcolumn}
\newcolumntype{d}[1]{D{.}{\cdot}{#1} }
\usepackage{hyperref}
\usepackage{tikz-cd, tikz}
\usetikzlibrary{decorations.pathmorphing}
\usepackage{float}
\usepackage{thmtools}
\usepackage{thm-restate}
\usepackage[utf8]{inputenc}
\allowdisplaybreaks

\newtheorem{theorem}{Theorem}[subsection]
\newtheorem*{theorem*}{Theorem}
\newtheorem{lemma}[theorem]{Lemma}
\newtheorem*{lemma*}{Lemma}
\newtheorem{proposition}[theorem]{Proposition}
\newtheorem*{proposition*}{Proposition}
\newtheorem{corollary}[theorem]{Corollary}
\theoremstyle{definition}
\newtheorem{definition}[theorem]{Definition}
\theoremstyle{remark}
\newtheorem{remark}[theorem]{Remark}
\newtheorem{example}[theorem]{Example}

\newcommand{\Z}{\mathbf{Z}}
\newcommand{\Q}{\mathbf{Q}}
\newcommand{\R}{\mathbf{R}}
\newcommand{\F}{\mathbf{F}}
\newcommand{\A}{\mathbf{A}}
\newcommand{\Proj}{\mathbf{P}}
\newcommand{\Aut}{\textnormal{Aut}}
\newcommand{\im}{\textnormal{im}}
\newcommand{\Gal}{\textnormal{Gal}}
\newcommand{\Hom}{\textnormal{Hom}}
\newcommand{\Cl}{\textnormal{Cl}}
\newcommand{\Norm}{\textnormal{Norm}}
\newcommand{\res}{\textnormal{res}}
\newcommand{\Sym}{\textnormal{Sym}}
\newcommand{\GL}{\operatorname{GL}}
\newcommand{\cyclo}{\chi}
\newcommand{\ur}{\textnormal{ur}}
\newcommand{\bigrep}[2]{\Sym^{#1}V \otimes \F_p(#2)}

\newcommand{\rad}{\textnormal{rad}}

\begin{document}

\title[Class groups of Kummer extensions via cup products]{Class groups of Kummer extensions via cup products in Galois cohomology}

\author{Karl Schaefer}
\address{University of Chicago Department of Mathematics, Chicago, IL}
\email{karl@math.uchicago.edu}

\author{Eric Stubley}
\address{University of Chicago Department of Mathematics, Chicago, IL}
\email{stubley@uchicago.edu}
\thanks{The second author wishes to acknowledge the support of the Natural Sciences and Engineering Research Council of Canada (NSERC).}


\date{}

\begin{abstract}
We use Galois cohomology to study the $p$-rank of the class group of $\mathbf{Q}(N^{1/p})$, where $N \equiv 1 \bmod{p}$ is prime.
We prove a partial converse to a theorem of Calegari--Emerton, and provide a new explanation of the known counterexamples to the full converse of their result.
In the case $p = 5$, we prove a complete characterization of the $5$-rank of the class group of $\mathbf{Q}(N^{1/5})$ in terms of whether or not $\prod_{k=1}^{(N-1)/2} k^{k}$ and $\frac{\sqrt{5} - 1}{2}$ are $5$th powers mod $N$.
\end{abstract}

\maketitle

\setcounter{tocdepth}{2}
\tableofcontents

\section{Introduction}\label{sec_introduction}
Let $N$ and $p \geq 3$ be prime numbers with $p | (N-1)$. Fix an algebraic closure $\overline{\Q}$ of $\Q$ and a choice of $N^{1/p} \in \overline{\Q}$, and let $K$ denote the field $\Q(N^{1/p})$.
The goal of this article is to study the class group $\Cl_K$ of $K$, and in particular its $p$-rank $r_K = \dim_{\F_p}(\Cl_K \otimes \F_p)$.
We access $r_{K}$ by class field theory, as $r_{K}$ is equal to the maximal $r$ such that $K$ admits an unramified-everywhere $(\Z/p\Z)^{r}$-extension.
By genus theory, $r_{K} \geq 1$, since the degree-$p$ subfield of $K(\zeta_{N})/K$ is unramified everywhere; this is a corollary of Lemma \ref{zeta_N_plus_localized}.
Our starting point is the following theorem of Calegari--Emerton.

\begin{theorem*}[Calegari--Emerton, Theorem 1.3, (ii) of \cite{calegari_emerton_ramification}]\label{intro_calegari_emerton}
Suppose that $p \geq 5$, and let $C = \prod_{k=1}^{(N-1)/2} k^{k}$.
If $C$ is a $p$th power in $\F_{N}^{\times}$, then $r_{K} \geq 2$.
\end{theorem*}

This theorem is proven using deformation theory of Galois representations.
Previous work of Merel \cite{merel} showed that whether or not the number $C$ is a $p$th power determines whether the $\Z_{p}$-rank of a certain Hecke algebra is at least $2$.
Calegari--Emerton identify this Hecke algebra with a deformation space of Galois representations, and construct an unramified $\F_{p}$-extension of $K$ in the case that the deformation space has $\Z_{p}$-rank at least 2.
More recently, this theorem was given another proof by Wake--Wang-Erickson (see Proposition 11.1.1 of \cite{wake_wang-erickson_mazur_eisenstein_ideal}, restated in this article as Proposition \ref{wake_wang_erickson_theorem}) using cup products in Galois cohomology.

Calegari--Emerton also raise the question of whether or not the converse to this theorem holds.
Numerical computations suggested that it was true when $p = 5$, but not in general. 
Indeed, Lecouturier noticed in \cite{lecouturier} that the converse fails in the case $p = 7$, $N = 337$.

\subsection{Results}\label{sec_results}

For odd $i$ satisfying $1 \leq i \leq p-4$, let 
\begin{equation*}
M_{i} = \prod_{k=1}^{N-1} \prod_{a=1}^{k-1} k^{a^{i}},
\end{equation*}
as first defined by Lecouturier in \cite{lecouturier}, and let 
$r_{\Q(\zeta_p)}$ be the $p$-rank of $\Cl_{\Q(\zeta_p)}$. Let $\cyclo$ be the mod-$p$ cyclotomic character and say that $(p, -i)$ is a regular pair if the $\cyclo^{-i}$-eigenspace of $\Cl_{\Q(\zeta_{p})}$ is trivial.

Lecouturier proves that
\begin{equation*}
r_{K} \leq r_{\Q(\zeta_{p})} + p - 2 - \mu,
\end{equation*}
where $\mu$ is the number of odd $i$ such that $1 \leq i \leq p-4$, $(p, -i)$ is a regular pair, and $M_{i}$ is not a $p$th power in $\F_{N}^{\times}$.

Using a new method, we improve the previous bound on $r_K$:
\begin{theorem}\label{intro_bound_theorem}
\begin{equation*}
r_{K} \leq r_{\Q(\zeta_{p})} + p - 2 - 2\mu.
\end{equation*}
\end{theorem}

This follows from the stronger inequality of Theorem \ref{intro_sigma_bound_theorem} combined with Theorems \ref{intro_odd_invariant_theorem} and \ref{intro_cohomology_theorem}.
An immediate corollary of Theorem \ref{intro_bound_theorem} in the case of regular $p$ is the following partial converse to the theorem of Calegari--Emerton:

\begin{theorem}\label{intro_converse_theorem}
Suppose that $p$ is regular, and that $r_{K} \geq 2$.
Then at least one of the $M_{i}$ is a $p$th power in $\F_{N}^{\times}$.
\end{theorem}
\begin{proof}
If $r_{K} \geq 2$, then the inequality of Theorem \ref{intro_bound_theorem} shows that $2 \leq p - 2 - 2\mu$.
As there are $\frac{p-3}{2}$ many $M_{i}$, it must be the case that $\mu < \frac{p-3}{2}$, i.e. at least one of the $M_{i}$ \emph{is} a $p$th power in $\F_{N}^{\times}$.
\end{proof}

The quantity $M_{1}$ is a $p$th power in $\F_{N}^{\times}$ if and only if $C = \prod_{k=1}^{(N-1)/2} k^{k}$ is (see Section \ref{sec_invariant_relationship} for this comparison).

When $p = 5$, Theorem \ref{intro_converse_theorem} is the full converse to the theorem of Calegari--Emerton, as the only $M_{i}$ is $M_{1}$. Furthermore, we give in Section \ref{5_subsection} an effective method for completely determining $r_K$ in this case:

\begin{theorem}\label{intro_5_theorem}
Let $p = 5$. 
Then, $1 \leq r_K \leq 3$ according to the following conditions:
\begin{enumerate}
\item
$r_{K} \geq 2$ if and only if $M_{1}$ is a $5$th power in $\F_N^\times$.
\item
$r_K = 3$ if and only if both $M_{1}$ and $\frac{\sqrt{5}-1}{2}$ are $5$th powers in $\F_N^\times$.
\end{enumerate}
\end{theorem}

The converse to Theorem \ref{intro_converse_theorem} is not true in general: in the case $p = 11$, $N = 353$ one has that both $r_K = 1$ and $M_3$ is an $11$th power in $\F_{353}^{\times}$. However, the converse to Theorem \ref{intro_converse_theorem} is true in the case $p = 7$, which we prove in Section \ref{7_subsection}:

\begin{theorem}\label{intro_7_theorem}
Let $p = 7$.
Then $r_{K} \geq 2$ if and only if one of $M_{1}$ or $M_{3}$ is a $7$th power in $\F_{N}^{\times}$.
\end{theorem}

This also explains the counterexample $p = 7$, $N = 337$ to the naive converse of the theorem of Calegari--Emerton: in that case, $r_K = 2$ and $M_1$ is not a $7$th power in $\F_{337}^{\times}$, but $M_3$ is.

\subsection{Strategy}\label{sec_strategy}
Put $S = \{p, N, \infty\}$ and let $G_{\Q,S}$ be the Galois group over $\Q$ of the maximal extension of $\Q$ unramified outside of $S$.

The methods used in this article are inspired by the strategy that Wake--Wang-Erickson use to prove the theorem of Calegari--Emerton.
They show that $M_{1}$ being a $p$th power in $\F_{N}^{\times}$ is equivalent to the vanishing of a certain cup product in Galois cohomology.
The vanishing of this cup product implies the existence of a reducible representation $G_{\Q, S} \to \GL_{3}(\F_{p})$, from which an unramified $\F_{p}$-extension of $K$ is constructed. 

Let $\F_p(i)$ denote the module $\F_p$ on which $G_{\Q,S}$ acts by $\cyclo^i$. Choose an isomorphism $\mu_p \to \F_p(1)$ and let $b: G_{\Q,S} \to \F_{p}(1)$ be the cocycle defined by $b(\sigma) = \sigma(N^{1/p})/N^{1/p}$. Let $V \cong \F_p^2$ be the vector space on which $G_{\Q,S}$ acts by the representation
\begin{align*}
G_{\Q, S}   & \to \GL_{2}(\F_{p}) \\*
\sigma      & \mapsto 
\begin{pmatrix}
\cyclo(\sigma) & b(\sigma) \\ 0 & 1
\end{pmatrix}.
\end{align*}

In an abuse of notation, we will also use $b$ to refer to the class of this cocycle in $H^1(G_{\Q,S}, \F_p(1))$, which is just the Kummer class of $N$.  Starting with an unramified $\F_{p}$-extension of $K$, we use the classification of indecomposable $\F_p$-representations of $\Gal(K(\zeta_{p})/\Q) \cong \Z/p\Z \rtimes (\Z/p\Z)^\times$ to show the existence of an upper-triangular Galois representation $G_{\Q, S} \to \GL_{m+2}(\F_{p})$ of the form
\begin{equation*}
\left(\begin{array}{c|c}
 & \\ \bigrep{m}{-m} & * \\ & \\ \hline  & 1 \\
\end{array}\right) = 
\left(\begin{array}{ccccc|c}
    1 & \cyclo^{-1}b & \cyclo^{-2}\frac{b^{2}}{2} & \cdots & \cyclo^{-m}\frac{b^{m}}{m!} & \ast \\
    & \cyclo^{-1} & \cyclo^{-2}b & \cdots & \cyclo^{-m}\frac{b^{m-1}}{(m-1)!} & \ast \\
    & & \cyclo^{-2} & & \vdots & \vdots \\
    & & & \ddots & \cyclo^{-m} b & \ast \\
    & & & & \cyclo^{-m} & \ast \\
    \hline
    & & & & & 1 \\
\end{array}\right).
\end{equation*}

Note that this symmetric power is written using a slightly non-standard basis, see Remark \ref{nonstandard_basis_for_sym} for an explanation as to why we use this basis.

The representations arising in this fashion give rise to classes in the $G_{\Q,S}$-cohomology of the high-dimensional Galois representations $\bigrep{j}{i}$.
We study the local properties of these cohomology classes and show that they satisfy a Selmer condition $\Sigma$, first considered by Wake--Wang-Erickson for the Galois module $\F_{p}(-1)$ (see Section \ref{sec_sigma} for the definition of $\Sigma$ in general).
This Selmer condition $\Sigma$ is chosen to detect exactly those classes whose cup product with $b$ is equal to $0$.
This leads to the following bound on $r_{K}$ in terms of the dimensions of the cohomology groups:

\begin{theorem}\label{intro_sigma_bound_theorem}
Let $h^{1}_{\Sigma}(\F_{p}(-i))$ denote the $\F_{p}$-dimension of  $H^{1}_{\Sigma}(\F_{p}(-i))$.
We have
\begin{equation*}
1 + h^{1}_{\Sigma}(\F_{p}(-1)) \leq r_{K} \leq 1 + \sum_{i=1}^{p-3} h^{1}_{\Sigma}(\F_{p}(-i)).
\end{equation*}
\end{theorem}

Section \ref{sec_big_general} is dedicated to the proof of this theorem.
Note that this theorem has as a corollary the statement that if $r_{K} \geq 2$, then at least one of the $H^{1}_{\Sigma}(\F_{p}(-i))$ is nonzero.
By a computation using Gauss sums, we relate the dimensions $h^{1}_{\Sigma}(\F_{p}(-i))$ to the quantities $M_{i}$ introduced earlier.

\begin{theorem}\label{intro_odd_invariant_theorem}
Assume that $i$ is odd and $(p,-i)$ is a regular pair.
Then we have $h^{1}_{\Sigma}(\F_{p}(-i)) = 1$ if and only if $M_{i}$ is a $p$th power in $\F_{N}^{\times}$, and $h^1_{\Sigma}(\F_p(-i)) = 0$ otherwise.
\end{theorem}

The proof of this theorem can be found in Sections \ref{sec_odd_invariant} and \ref{sec_invariant_relationship}.

While not needed to establish Theorem \ref{intro_bound_theorem}, in order to prove Theorem \ref{intro_5_theorem} we need to find a computable criterion for determining when $h^1_\Sigma(\F_p(-i)) = 1$ for even $i$. This is done for $(p,1+i)$ a regular pair in Section \ref{sec_even_invariant}.

Finally, to establish Theorem \ref{intro_bound_theorem}, we need the following theorem, which comes from duality theorems in Galois cohomology.

\begin{theorem}\label{intro_cohomology_theorem}
For any $1 \leq i \leq p-3$
\begin{equation*}
    h^1_\Sigma(\F_p(-i)) \leq 1 + r_{\Q(\zeta_p)}^{\cyclo^{-i}},
\end{equation*}
where $r_{\Q(\zeta_p)}^{\cyclo^{-i}}$ is the $p$-rank of the $\cyclo^{-i}$-eigenspace of $\Cl_{\Q(\zeta_p)}$.

Furthermore, if $p$ is odd and $h^1_\Sigma(\F_p(-i)) = 0$, then $h^1_\Sigma(\F_p(-(p-2-i))) = h^1_\Sigma(\F_p(1+i)) = 0$ as well.
\end{theorem}

This theorem is a combination of Theorem \ref{H1_N_bound} and Corollary \ref{even_H1Sigma_implies_odd_H1Sigma}.

The outline of this article is as follows. 
In Section \ref{sec_cohomology}, we recall some facts about Selmer groups, define the Selmer condition $\Sigma$, and prove several lemmas about the relationship between the condition $\Sigma$ and the vanishing of cup products. 
In Section \ref{sec_big_general}, we relate the $p$-part of $\Cl_K$ to Selmer groups of higher-dimensional representations of $G_{\Q,S}$ and prove Theorem \ref{intro_sigma_bound_theorem}.
In Section \ref{sec_lifting_selmer_classes}, we prove results about when classes in $H^{1}_{\Sigma}(\F_{p}(-i))$ can be lifted to classes in the $\Sigma$-Selmer group of the higher-dimensional representations arising in Section \ref{sec_big_general}.
In Section \ref{sec_effective_criteria}, we demonstrate relationships between Selmer groups of characters and the quantities $M_i$ for odd $i$. For even $i$, the Selmer group is shown to be related to both $M_{1-i}$ and another quantity arising from the units of the cyclotomic field $\Q(\zeta_{p})$. 
Finally, in Section \ref{sec_specific_primes}, we analyze the cases $p = 5$ and $p = 7$ in more detail. 
Appendix \ref{sec_data} contains computer calculations of $r_K$ and the dimensions $h^1_\Sigma(\F_p(-i))$ for $p=5, N \leq 20{,}000{,}000$ and $p = 7, N \leq 100{,}000{,}000$.

One might ask if the techniques of this article can be applied to composite $N$.
The authors are currently considering this generalization.

\subsection{Acknowledgements}\label{sec_acknowledgements}

The authors would like to thank Frank Calegari and Matthew Emerton for many helpful discussions on this topic. We would in particular like to thank Frank Calegari for drawing our attention to this problem and for suggesting the techniques that led to the theorems in Section \ref{pfsetup}. The authors would also like to thank Emmanuel Lecouturier, Romyar Sharifi, Preston Wake, and Carl Wang-Erickson for their encouragement and interest in our results, and for feedback on an early draft of this article. We also thank the anonymous referee for their suggestions which greatly improved the clarity of the paper.

\section{Cohomology Computations}\label{sec_cohomology}

Throughout this article we will work with Selmer groups in the cohomology of various mod-$p$ representations of $G_{\Q} = \Gal(\overline{\Q}/\Q)$.
In fact all representations we consider will be unramified outside of $S = \{p, N, \infty\}$, so will be representations of $G_{\Q, S}$, the Galois group over $\Q$ of the maximal extension of $\Q$ unramified outside of $S$.

\subsection{Notation}\label{sec_notation}
We first establish some notation and conventions used throughout the article as well as recall some facts about group cohomology. Let $A$ be an $\F_p$-vector space with an action of $G_{\Q}$ via $\rho:G_{\Q} \to \GL_n(\F_p)$.
\begin{itemize}
\item Let $\F_{p}$ and $\F_{p}(1)$ be the $1$-dimensional $\F_{p}$-vector spaces on which $G_{\Q}$ acts trivially and by the mod-$p$ cyclotomic character $\cyclo$, respectively. 
Let $A(i) = A \otimes_{\F_{p}} \F_{p}(1)^{\otimes i}$.
Throughout, fix a primitive $p$th root of unity $\zeta_{p}$, which determines an isomorphism $\mu_{p} \cong \F_{p}(1)$.

\item Let $b: G_{\Q,S} \to \F_p(1)$ be the cocycle defined by $\sigma \mapsto \sigma(N^{1/p})/N^{1/p}$. By Kummer Theory,
\begin{equation*}
H^1(G_{\Q,S}, \F_p(1)) = \frac{\Z[1/pN]^\times}{\Z[1/pN]^{\times p}}.
\end{equation*}
The class of $b$ in $H^1(G_{\Q,S}, \F_p(1))$, which we also denote by $b$, is the class of $N$ under this isomorphism.

\item We denote by $A^\vee$ and $A^\ast$ the $G_\Q$-modules
\begin{equation*}
A^{\vee} = \Hom(A, \F_{p}) \quad\text{ and }\quad
A^{\ast} = A^{\vee}(1) = \Hom(A, \F_{p}(1)).
\end{equation*}

\item Given a class $a \in H^1(G_\Q, A)$ represented by a cocyle $a: G_\Q \to A \cong \F_p^n$, we can write
\begin{equation*}
    a(\sigma) = \begin{bmatrix}a_0(\sigma) \\ \vdots \\ a_{n-1}(\sigma)\end{bmatrix}
\end{equation*}
for $\sigma \in G_\Q$. This defines a new $(n+1)$-dimensional $G_{\Q}$-representation which is an extension of $\F_p$ by $A$ via the map
\begin{equation*}
\sigma \mapsto
\left(\begin{array}{ccc|c}
& & & a_0(\sigma) \\
& \rho(\sigma) & & \vdots \\
& & & a_{n-1}(\sigma) \\
\hline
& 0 & & 1
\end{array}\right) \in \GL_{n+1}(\F_p)
\end{equation*}
whose kernel cuts out a Galois extension of $\Q$. Conversely, given a $G_{\Q}$-representation which is an extension of $\F_p$ by $A$ of the above form, we get a cohomology class which we denote by
\begin{equation*}
a = \begin{bmatrix}a_0 \\ \vdots \\ a_{n-1}\end{bmatrix} \in H^1(G_\Q, A).
\end{equation*}

\item Given any characters $\chi, \chi': G_\Q \to \F_p^\times$, let $\F_p(\chi)$ and $\F_p(\chi')$ be the lines on which $G_\Q$ acts by $\chi$ and $\chi'$, respectively. Classes $a \in H^1(G_\Q, \F_p(\chi))$ and $a' \in H^1(G_\Q, \F_p(\chi'))$ correspond to $2$-dimensional $G_\Q$-representations of the forms
\begin{equation*}
\begin{pmatrix}\chi & a \\ 0 & 1\end{pmatrix} \quad\text{ and }\quad
\begin{pmatrix}\chi' & a' \\ 0 & 1\end{pmatrix},
\end{equation*}
respectively. These patch together to form a $3$-dimensional representation
\begin{equation*}
\begin{pmatrix}\chi\chi' & \chi'a & c \\ 0 & \chi' & a' \\ 0 & 0 & 1\end{pmatrix}
\end{equation*}
if and only if $a \cup a' = 0$ as cohomology classes, in which case the coboundary of $-c$ is the cochain $a \cup a'$.
\end{itemize}

For a $G_{\Q}$-module $A$, recall that a Selmer condition is a collection $\mathcal{L} = \{L_{v}\}$ of subspaces $L_{v} \subseteq H^{1}(G_{\Q_{v}}, A)$ where $v$ runs over all places of $\Q$, such that $L_{v}$ is the unramified subspace
\begin{equation*}
H^{1}_{\ur}(G_{\Q_{v}}, A) := H^{1}(G_{\F_{v}}, A^{I_{v}})
\end{equation*}
for almost all places $v$, where $I_v \subseteq G_{\Q_v} = \Gal(\overline{\Q_v}/\Q_v)$ is the inertia subgroup and $G_{\F_v} = G_{\Q_v}/I_v$ is the absolute Galois group of the residue field at $v$.
The Selmer group associated to a set of conditions $\mathcal{L}$ is then
\begin{equation*}
H^{1}_{\mathcal{L}}(G_{\Q}, A) = \ker{\left(H^{1}(G_{\Q}, A) \to \prod_{v} \frac{H^{1}(G_{\Q_{v}}, A)}{L_{v}}\right)}.
\end{equation*}

We will use the following conventions in describing Selmer groups.
\begin{itemize}
\item To simplify notation, we will denote a Selmer group $H^{1}_{\mathcal{L}}(G_{\Q}, A)$ by $H^{1}_{\mathcal{L}}(A)$.

\item As every module $A$ we will consider will be an $\F_{p}$-vector space, we will use the following notation for dimensions:
\begin{equation*}
h^{1}_{\mathcal{L}}(A) = \dim_{\F_{p}}(H^{1}_{\mathcal{L}}(A)).
\end{equation*}

\item All Selmer conditions we use have the unramified condition at places outside of $S$.
In particular, since $p$ is assumed to be odd, we will always have $H^{1}(G_{\R}, A) = 0$, removing the need to specify a local condition at the infinite place.

\item Given a subset $T \subset S = \{p, N, \infty \}$, we will use the notation $H^{1}_{T}(A)$ to denote the Selmer group with the unramified condition at all places outside of $T$, and any behavior allowed at the places of $T$.
\end{itemize}

\begin{remark}
If $A$ is a module for $G_{\Q, S}$ then the Selmer group $H^{1}_{S}(A)$ is equal to the $G_{\Q, S}$-cohomology $H^{1}(G_{\Q, S}, A)$.
Every $G_{\Q}$-module we consider will in fact be a $G_{\Q, S}$-module.
\end{remark}

Given a Selmer condition $\mathcal{L} = \{L_{v}\}$ for $A$, $\mathcal{L}^{\ast} := \{L_{v}^{\perp}\}$ is a Selmer condition for $A^{\ast}$, where the orthogonal complements are taken with respect to the Tate pairing on local cohomology groups.
Note that when $v$ does not divide $\#A$ and the action of $G_{\Q_{v}}$ on $A$ is unramified, we have that $H^{1}_{\ur}(G_{\Q_{v}}, A)^{\perp} = H^{1}_{\ur}(G_{\Q_{v}}, A^{\ast})$ (see Theorem 2.6 of \cite{milne_arithmetic_duality_theorems}).
A main tool that we will use is the following formula for sizes of Selmer groups, due to Greenberg and Wiles.

\begin{theorem}\label{greenberg_wiles}
Let $A$ be a finite $G_{\Q}$-module, and let $\mathcal{L} = \{L_{v}\}$ be a Selmer condition for $A$. Then $H^{1}_{\mathcal{L}}(A)$ and $H^{1}_{\mathcal{L}^{\ast}}(A^{\ast})$ are finite and
\begin{equation*}
\frac{\# H^{1}_{\mathcal{L}}(A)}{\# H^{1}_{\mathcal{L}^{\ast}}(A^{\ast})} = \frac{\# H^{0}(G_{\Q}, A)}{\# H^{0}(G_{\Q}, A^{\ast})} \prod_{v} \frac{\# L_{v}}{\# H^{0}(G_{\Q_{v}}, A)}
\end{equation*}
where the product is over all places $v$ of $\Q$.
\end{theorem}
See \cite{washington_galois_cohomology} for a proof of this theorem.
For all $v$ that don't divide $\# A$ and for which $L_v$ is the subgroup of unramified classes, one has $\# L_v = \# H^0(G_{\Q_v}, A)$. Since every Selmer condition that we will use will have the unramified condition at places outside $S$ and since all of our modules will be $\F_p$-vector spaces, the only terms of the above product which will ever contribute in our applications are the $H^{0}$ term and the local terms at $N$, $p$, and $\infty$.

We will often want to compare sizes of Selmer groups when we change the Selmer conditions.
The following lemma gives a way to do such a comparison.

\begin{lemma}\label{change_of_selmer_group}
Suppose that $\mathcal{L} = \{L_{v}\}$ and $\mathcal{L}' = \{L_{v}'\}$ are two Selmer conditions for $A$ where $\mathcal{L} \subset \mathcal{L}'$ in the sense that $L_{v} \subseteq L_{v}'$ for all $v$.
Then we have
\begin{equation*}
\# H^{1}_{\mathcal{L}'}(A) \leq \# H^{1}_{\mathcal{L}}(A) \prod_{v} \frac{\# L_{v}'}{\# L_{v}}
\end{equation*}
where the product is over all places $v$ of $\Q$.
\end{lemma}
\begin{proof}
By the definitions of the Selmer groups in question there is an exact sequence
\begin{equation*}
0 \to H^{1}_{\mathcal{L}}(A) \to H^{1}_{\mathcal{L}'}(A) \to \bigoplus_{v} \frac{L_{v}'}{L_{v}}.
\end{equation*}
The lemma follows by considering the sizes of the terms in this sequence.
\end{proof}

\subsection{The Selmer Condition \texorpdfstring{$\Sigma$}{Sigma}}\label{sec_sigma}

We define here the Selmer condition $\Sigma = \{L_{v}\}$.
The local conditions of $\Sigma$ are defined by
\begin{itemize}
\item $L_{p} = 0$.
\item $L_{N} = \ker\left(\res: H^{1}(G_{\Q_{N}}, A) \to H^{1}(G_{K_{N}}, A)\right)$
where $K_N = \Q_N(N^{1/p})$ is the completion of the field $K$ at the unique prime above $N$.
\item $L_{v}$ is the unramified condition at places outside $S$.
\end{itemize}
As usual, we define the dual Selmer condition $\Sigma^{\ast} = \{L_{v}^{\perp}\}$, where $L_{v}^{\perp}$ is the annihilator of $L_{v}$ under the local cup product pairing.
Applied to a $G_{\Q, S}$-module $A$, it is clear that $L_{p}^{\perp} = H^{1}(G_{\Q_{p}}, A)$, and $L_{v}^{\perp} = H^{1}_{\ur}(G_{\Q_{p}}, A)$ for places $v$ outside $S$.
See Proposition \ref{condition_at_N_is_self_dual} for the determination of the condition $L_{N}^{\perp}$.

We will only consider these Selmer conditions $\Sigma$ and $\Sigma^{\ast}$ for modules which are isomorphic as $G_{\Q_{N}}$-modules to $\Sym^{n}(V)$ for some $n \geq 0$, where $V$ is the $2$-dimensional $\F_{p}$-vector space on which $G_{\Q, S}$ acts in some basis by
\begin{equation*}
\begin{pmatrix}
\cyclo & b \\ 0 & 1
\end{pmatrix}.
\end{equation*}
Note that when viewed as a $G_{\Q_{N}}$-module the cyclotomic character $\cyclo$ is trivial, as $\mu_{p} \subset \Q_{N}^{\times}$.

We establish in the following lemma and propositions the statements about $\Sym^{n}V$ and its cohomology as a $G_{\Q_{N}}$-module which will be relevant for applying Theorem \ref{greenberg_wiles}.

\begin{lemma}\label{big_rep_dual}
For $n \leq p-1$, $(\Sym^{n}V)^{\vee} \cong \bigrep{n}{-n}$.
\end{lemma}
\begin{proof}
Note that the action of $G_{\Q}$ on $\Sym^{n}V$ factors through $G = \Gal(K(\zeta_{p})/\Q)$.
The range of $n$ considered are in fact those symmetric powers of $V$ which are indecomposable as $\F_{p}$-representations of $G$ (see Theorem \ref{classification_of_reps}).
The only indecomposable representation of $G$ of dimension $n$ are the twists by $\cyclo$ of $\Sym^{n}V$; since the dual of an indecomposable representation will certainly also be indecomposable and of the same dimension, we must have that $(\Sym^{n}V)^{\vee} \cong \bigrep{n}{m}$ for some $m$.
We consider the evaluation pairing
\begin{equation*}
    \Sym^{n}V \otimes (\Sym^{n}V)^{\vee} \to \F_{p}
\end{equation*}
restricted to the $1$-dimensional subrepresentation $\F_{p}(n)$ of $\Sym^{n}V$.
Since the above pairing is a perfect $G_{\Q}$-module pairing, the annihilator of $\F_{p}(n)$ must be an $n$-dimensional subrepresentation of $(\Sym^{n}V)^{\vee}$. Since $(\Sym^{n}V)^{\vee} \cong \bigrep{n}{m}$ has a unique $n$-dimensional subrepresentation, this means that the pairing descends to a perfect pairing between $\F_{p}(n)$ and the (unique) $1$-dimensional quotient of $(\Sym^{n}V)^{\vee}$.
As this $1$-dimensional quotient is $\F_{p}(m)$, we conclude that $m = -n$, as a perfect $G_{\Q}$-module pairing
\begin{equation*}
    \F_{p}(n) \otimes \F_{p}(m) \to \F_{p}
\end{equation*}
exists if and only if $m = -n$.
\end{proof}

\begin{proposition}\label{basis_of_local_H1}
For all $n$ in the range $0 \leq n \leq p-2$, $H^{1}(G_{\Q_{N}}, \Sym^{n}V)$ is $2$-dimensional, being spanned by
\begin{equation*}
\mathbf{a} = \begin{bmatrix} a \\ 0 \\ \vdots \\ 0 \\ \end{bmatrix}
\quad \text{and} \quad
\mathbf{b} = \begin{bmatrix} \frac{b^{n+1}}{(n+1)!} \\ \frac{b^{n}}{n!} \\ \vdots \\ b \\ \end{bmatrix},
\end{equation*}
where $a$ is a class spanning the $1$-dimensional unramified subspace of $H^{1}(G_{\Q_{N}}, \F_{p})$, and $\mathbf{b}$ is the class corresponding to $\Sym^{n+1}V$, which is an extension of $\F_p$ by $\Sym^{n}V$ as $G_{\Q_{N}}$-modules.

Further, the subgroup $L_{N} \subset H^{1}(G_{\Q_{N}}, \Sym^{n}V)$ is $1$-dimensional and spanned by $\mathbf{b}$.
\end{proposition}
\begin{proof}
It follows from the Local Euler Characteristic Formula (Theorem 2.8 of \cite{milne_arithmetic_duality_theorems}) that $H^{1}(G_{\Q_{N}}, \Sym^{n}V)$ is $2$-dimensional.
Consider the short exact sequence
\begin{equation*}
0 \to \F_{p} \to \Sym^{n}V \to \Sym^{n-1}V \to 0
\end{equation*}
of $G_{\Q_{N}}$-modules.
The first terms of the associated long exact sequence in $G_{\Q_{N}}$-cohomology give us
\begin{equation*}
0 \to \F_{p} \to \F_{p} \to \F_{p} \overset{b \cup -}{\longrightarrow} H^{1}(G_{\Q_{N}}, \F_{p}) \to H^{1}(G_{\Q_{N}}, \Sym^{n}V).
\end{equation*}
Since $a$ and $b$ form a basis for $H^{1}(G_{\Q_{N}}, \F_{p})$, and the image of $\F_{p} \overset{b \cup -}{\longrightarrow} H^{1}(G_{\Q_{N}}, \F_{p})$ is the span of $b$, we conclude that the image of
\begin{equation*}
    H^{1}(G_{\Q_{N}}, \F_{p}) \to H^{1}(G_{\Q_{N}}, \Sym^{n}V)
\end{equation*}
is spanned by the image of $a$; this is the class $\mathbf{a}$ defined above.
To see that the class $\mathbf{b}$ is nonzero, consider the map
\begin{equation*}
H^{1}(G_{\Q_{N}}, \Sym^{n}V) \to H^{1}(G_{\Q_{N}}, \F_{p})
\end{equation*}
coming from the long exact sequence in $G_{\Q_{N}}$-cohomology associated to the short exact sequence
\begin{equation*}
0 \to \Sym^{n-1}V \to \Sym^{n}V \to \F_{p} \to 0.
\end{equation*}
The image of $\mathbf{b}$ under this map is the class $b \in H^{1}(G_{\Q_{N}}, \F_{p})$, which is nonzero, hence we conclude that $\mathbf{b}$ itself is nonzero.

Finally we see that $\mathbf{a}$ and $\mathbf{b}$ are linearly independent in $H^{1}(G_{\Q_{N}}, \Sym^{n}V)$ (and therefore constitute a basis), as $\mathbf{b}$ is trivial when restricted to $G_{K_{N}}$ (even as a cocycle) and $\mathbf{a}$ is not.
This also establishes that
\begin{equation*}
L_{N} = \ker(H^{1}(G_{\Q_{N}}, \Sym^{n}V) \to H^{1}(G_{K_{N}}, \Sym^{n}V))
\end{equation*}
is $1$-dimensional and is spanned by $\mathbf{b}$.
\end{proof}

\begin{proposition}\label{condition_at_N_is_self_dual}
Suppose that $A \cong \Sym^{n}V$ as a $G_{\Q_{N}}$-representation for some $n \leq p-2$.
Under the local Tate pairing
\begin{equation*}
H^{1}(G_{\Q_{N}}, A) \otimes H^{1}(G_{\Q_{N}}, A^{\ast}) \to H^{2}(G_{\Q_{N}}, \F_{p}(1))
\end{equation*}
the annihilator of $L_{N} \subseteq H^{1}(G_{\Q_{N}}, A)$ is
\begin{equation*}
L_{N}^{\perp} = \ker\left(\res: H^{1}(G_{\Q_{N}}, A^{\ast}) \to H^{1}(G_{K_{N}}, A^{\ast})\right).
\end{equation*}
That is, the dual condition $L_{N}^{\perp}$ is again the condition $L_{N}$ (applied to the module $A^{\ast}$).
\end{proposition}
\begin{proof}
We first note that if suffices to prove this proposition only for $\Sym^{n}V$, as if $f: \Sym^{n}V \to A$ is an isomorphism of $G_{\Q_{N}}$-modules, we have that $f$ induces an isomorphism between first cohomology groups which restricts to an isomorphism between the $L_{N}$ subgroup on each side.
The same also holds for the $L_{N}$ subgroups in the first cohomology of $(\Sym^{n}V)^{\ast}$ and $A^{\ast}$.

Choose an isomorphism of $G_{\Q_{N}}$-modules $\phi: \Sym^{n}V \to (\Sym^{n}V)^{\ast}$ (this is possible as globally $\Sym^{n}V$ is self-dual up to a twist by some power of $\cyclo$ by Lemma \ref{big_rep_dual}, and $\cyclo$ is trivial as a character of $G_{\Q_{N}}$).
We have as before that the isomorphism on first cohomology induced by $\phi$ restricts to an isomorphism of $L_{N}$ subgroups; we write henceforth $\phi(L_{N})$ for the $L_{N}$ condition subgroup of $H^{1}(G_{\Q_{N}}, (\Sym^{n}V)^{\ast})$.
We know that the local Tate pairing in question is a perfect pairing, hence the annihilator $L_{N}^{\perp}$ of $L_{N}$ must also be $1$-dimensional.
Therefore it suffices to prove that $\phi(L_{N})$ is contained in $L_{N}^{\perp}$, i.e. $L_{N} \cup \phi(L_{N}) = 0$.

The cup product map
\begin{equation*}
H^{1}(G_{\Q_{N}}, \Sym^{n}V) \otimes H^{1}(G_{\Q_{N}}, \Sym^{n}V) \to H^{2}(G_{\Q_{N}}, (\Sym^{n}V)^{\otimes 2})
\end{equation*}
is alternating, as it is in an odd degree of cohomology.
In particular under this cup product map $L_{N} \cup L_{N} = 0$.
Applying the isomorphism $\phi$ to the second coordinate gives that under the cup product
\begin{equation*}
H^{1}(G_{\Q_{N}}, \Sym^{n}V) \otimes H^{1}(G_{\Q_{N}}, (\Sym^{n}V)^{\ast}) \to H^{2}(G_{\Q_{N}}, \Sym^{n}V \otimes (\Sym^{n}V)^{\ast})
\end{equation*}
we have that $L_{N} \cup \phi(L_{N}) = 0$.
The local Tate pairing is the composition of the above cup product map with the map
\begin{equation*}
H^{2}(G_{\Q_{N}}, \Sym^{n}V \otimes (\Sym^{n}V)^{\ast}) \to H^{2}(G_{\Q_{N}}, \F_{p}(1))
\end{equation*}
induced by the evaluation pairing $\Sym^{n} V \otimes (\Sym^{n}V)^{\ast} \to \F_{p}(1)$, so we conclude that $\phi(L_{N}) = L_{N}^{\perp}$.
\end{proof}

We finish this section with a lemma regarding the Selmer group $H^1_{\Sigma^{\ast}}(\F_p)$.

\begin{lemma}\label{zeta_N_plus_localized}
The completion of $\Q(\zeta_{N}^{(p)})$ at the prime above $N$ is $K_{N}$. That is, the class $c \in H^{1}_{S}(\F_{p})$ which represents $\Q(\zeta_{N}^{(p)})$ lies in the Selmer subgroup $H^1_{\Sigma^{\ast}}(\F_p)$.
\end{lemma}
\begin{proof}
The two extensions of $\Q_N$ in question are $\Q_N(\zeta_N^{(p)})$ and $\Q_N(N^{1/p})$, both of which are totally ramified $\F_p$-extensions of $\Q_N$. 

We can see their equality by computing the norm subgroup in $\Q_{N}^{\times}$ of both extensions and showing they are equal.
We know that the norm subgroups will contain $(\Q_{N}^{\times})^{p}$ as an index $p$ subgroup; since this is index $p^{2}$ in $\Q_{N}^{\times}$, it suffices to show that our two norm groups both contain the element $N$.
One one hand we have that
\begin{align*}
\Norm_{\Q_{N}}^{K_{N}}(N^{1/p})     & = \prod_{i = 0}^{p-1} \zeta_{p}^{i} N^{1/p} \\*
                                & = N
\end{align*}
but we also have
\begin{align*}
\Norm_{\Q_{N}}^{\Q_{N}(\zeta_{N}^{(p)})} (\Norm_{\Q_{N}(\zeta_{N}^{(p)})}^{\Q_{N}(\zeta_{N})}(1 - \zeta_{N}))   & = \Norm_{\Q_{N}}^{\Q_{N}(\zeta_{N})}(1 - \zeta_{N}) \\*
& = \prod_{j=1}^{N-1} (1 - \zeta_{N}^{j}) \\*
& = N.
\end{align*}
Therefore we conclude that $\Q_{N}(\zeta_{N}^{(p)}) = K_{N}$.
\end{proof}

\subsection{Selmer Groups in the Cohomology of the Cyclotomic Character}\label{sec_cohomology_groups_of_characters}

This section contains a collection of statements about the dimensions of various Selmer groups in the cohomology of $\F_p(i)$.

\begin{definition}
Let $p$ be an odd prime and $0 \leq i \leq p-2$. Let $r_{\Q(\zeta_{p})}^{\cyclo^{i}}$ denote the $p$-rank of the $\cyclo^{i}$-eigenspace of the class group of $\Q(\zeta_{p})$. We say that $(p, i)$ is a \emph{regular pair} if $r_{\Q(\zeta_{p})}^{\cyclo^{i}} = 0$.
\end{definition}

\begin{remark}\label{regular_pair_bernoulli_number}
It is always true that $(p, 0)$ and $(p, 1)$ are regular pairs. 
If $i$ is odd, the theorems of Herbrand and Ribet give the following characterization: $(p, i)$ is a regular pair if and only if the generalized Bernoulli number $B_{1, \cyclo^{-i}}$ (equivalently, the Bernoulli number $B_{p-i}$) is not divisible by $p$. See Section 6.3 of \cite{washington_book} for a more detailed discussion of these facts.
\end{remark}

\begin{theorem}\label{cohomology_theorem_unconditional}
Let $p$ be an odd prime.
The following statements are true.
\begin{enumerate}
\item\label{cohomology_theorem_kw} The group $H^{1}_{S}(\F_{p})$ is $2$-dimensional, spanned by the classes of the homomorphisms defining the degree $p$ subfields $\Q(\zeta_{N}^{(p)})$ and $\Q(\zeta_{p^{2}}^{(p)})$ of $\Q(\zeta_{N})$ and $\Q(\zeta_{p^{2}})$, respectively.

\item\label{cohomology_theorem_kummer} The group $H^{1}_{S}(\F_{p}(1))$ is $2$-dimensional, and spanned by the classes of $N$ and $p$ under the Kummer isomorphism
\begin{equation*}
H^{1}_{S}(\F_{p}(1)) = \frac{\Z[1/pN]^{\times}}{(\Z[1/pN]^{\times})^{p}}.
\end{equation*}

\item\label{cohomology_theorem_empty} For any $i$, we have that
\begin{equation*}
h^{1}_{\emptyset}(\F_{p}(i)) = r_{\Q(\zeta_{p})}^{\cyclo^{i}}.
\end{equation*}

\item\label{cohomology_theorem_reflection} For any odd $i \not\equiv 1 \bmod{p-1}$ we have that
\begin{equation*}
h^{1}_{\emptyset}(\F_{p}(1-i)) \leq h^{1}_{\emptyset}(\F_{p}(i)) \leq 1 + h^{1}_{\emptyset}(\F_{p}(1-i)).
\end{equation*}
This is equivalent to Theorem 10.9 of \cite{washington_book}.
\end{enumerate}
\end{theorem}
\begin{proof}
Parts \ref{cohomology_theorem_kw} and \ref{cohomology_theorem_kummer} follow from the Kronecker-Weber theorem and Kummer theory, respectively.

For part \ref{cohomology_theorem_empty}, note that the restriction map \begin{equation*}
    H^{1}(G_{\Q, S}, \F_{p}(i)) \to H^{1}(G_{\Q(\zeta_{p}), S}, \F_{p}(i))^{\Gal(\Q(\zeta_{p})/\Q)}
\end{equation*}
is an isomorphism by the inflation-restriction sequence.
This latter group can be interpreted as the $\F_{p}$-extensions of $\Q(\zeta_{p})$ which are unramified away from $S$ and whose Galois group is $\F_{p}(i)$ as a $\Gal(\Q(\zeta_{p})/\Q)$-module through the equality
\begin{align*}
H^{1}(G_{\Q(\zeta_{p}), S}, \F_{p}(i))     & = \Hom(G_{\Q(\zeta_{p}), S}, \F_{p}(i)).
\end{align*}
The subgroup $H^{1}_{\emptyset}(\F_{p}(i))$ is those classes which are unramified everywhere.
Global class field theory gives that $r_{\Q(\zeta_{p})}^{\cyclo^{i}}$ is the number of independent $\F_{p}$-extensions of $\Q(\zeta_{p})$ which are unramified everywhere and whose Galois group is $\F_{p}(i)$ as a $\Gal(\Q(\zeta_{p})/\Q)$-module.
Thus we conclude that the dimension $h^{1}_{\emptyset}(\F_{p}(i))$ is equal to $r_{\Q(\zeta_{p})}^{\cyclo^{i}}$, as both count the same set of extensions.

The inequalities in part \ref{cohomology_theorem_reflection} both follow from applying Theorem \ref{greenberg_wiles} and estimating dimensions in a change of Selmer conditions as in Lemma \ref{change_of_selmer_group}.
For instance, by Theorem \ref{greenberg_wiles} applied to $H^{1}_{\emptyset}(\F_{p}(i))$ we have
\begin{align*}
& \frac{\# H^{1}_{\emptyset}(\F_{p}(i))}{\# H^{1}_{\emptyset^{\ast}}(\F_{p}(1-i))} \\*
& = \frac{\# H^{0}(\F_{p}(i))}{\# H^{0}(\F_{p}(1-i))} \prod_{v} \frac{\# L_{v}}{\# H^{0}(G_{\Q_{v}}, \F_{p}(i))} \\
& = \frac{\# H^{0}(\F_{p}(i))}{\# H^{0}(\F_{p}(1-i))} \cdot  \frac{\# H^{1}_{\ur}(G_{\Q_{N}}, \F_{p}(i))}{\# H^{0}(G_{\Q_{N}}, \F_{p}(i))} \cdot \frac{\# H^{1}_{\ur}(G_{\Q_{p}}, \F_{p}(i))}{\# H^{0}(G_{\Q_{p}}, \F_{p}(i))} \cdot \frac{\# H^{1}(G_{\R}, \F_{p}(i))}{\# H^{0}(G_{\R}, \F_{p}(i))} \\
& = \frac{1}{1} \cdot \frac{p}{p} \cdot \frac{1}{1} \cdot \frac{1}{1} \\*
& = 1
\end{align*}
where we know all of the local terms using the Local Euler Characteristic Formula and the parity of $i$.
Stated in terms of dimensions, this relation is
\begin{equation*}
h^{1}_{\emptyset}(\F_{p}(i)) = h^{1}_{\emptyset^{\ast}}(\F_{p}(1-i)).
\end{equation*}
Since we have that the Selmer condition $\emptyset^{\ast}$ contains the Selmer condition $\emptyset$, we may apply Lemma \ref{change_of_selmer_group} to get
\begin{align*}
\# H^{1}_{\emptyset^{\ast}}(\F_{p}(1-i))    & \leq \#H^{1}_{\emptyset}(\F_{p}(1-i)) \frac{\# H^{1}(G_{\Q_{p}}, \F_{p}(1-i))}{\# H^{1}_{\ur}(G_{\Q_{p}}, \F_{p}(1-i))} \\*
                                            & = \#H^{1}_{\emptyset}(\F_{p}(1-i)) \cdot p
\end{align*}
where we have again used the Local Euler Characteristic Formula to determine the local terms.
Stated in terms of dimensions, this relation is
\begin{equation*}
h^{1}_{\emptyset^{\ast}}(\F_{p}(1-i)) \leq h^{1}_{\emptyset}(\F_{p}(1-i)) + 1.
\end{equation*}
Thus we conclude that
\begin{equation*}
h^{1}_{\emptyset}(\F_{p}(i)) \leq h^{1}_{\emptyset}(\F_{p}(1-i)) + 1.
\end{equation*}
The other inequality of part \ref{cohomology_theorem_reflection} follows from a similar argument, starting with $H^{1}_{\emptyset}(\F_{p}(1-i))$.
\end{proof}

\begin{corollary}\label{H1_N_bound}
For any $i$, $h^1_\Sigma(\F_p(i)) \leq 1 + r_{\Q(\zeta_p)}^{\cyclo^{i}}$.
\end{corollary}
\begin{proof}
This follows from the fact that $H^{1}_{\Sigma}(\F_{p}(i)) \subseteq H^{1}_{N}(\F_{p}(i))$ and Lemma \ref{change_of_selmer_group} applied to the Selmer conditions $\emptyset$ and $N$ along with part \ref{cohomology_theorem_empty} of the previous Theorem.
\end{proof}

\begin{theorem}\label{cohomology_theorem_conditional}
Let $p$ be an odd prime, let $i \not\equiv 1 \bmod{p-1}$ be odd, and assume that $(p, i)$ is a regular pair.
Then we have the following
\begin{enumerate}
\item\label{cohomology_theorem_empty_regular}
$h^{1}_{\emptyset}(\F_{p}(i)) = h^{1}_{\emptyset}(\F_{p}(1-i)) = 0$.
\item\label{cohomology_theorem_odd_regular}
$h^{1}_{S}(\F_{p}(i)) = 2$, $h^{1}_{p}(\F_{p}(i)) = 1$, and $h^{1}_{N}(\F_{p}(i)) = 1$.
\item\label{cohomology_theorem_even_regular}
$h^{1}_{S}(\F_{p}(1-i)) = 1$.
\item\label{cohomology_theorem_sigma_regular}
$h^{1}_{\Sigma}(\F_{p}(i))$ and $h^{1}_{\Sigma}(\F_{p}(1-i))$ are both at most $1$.
\end{enumerate}
\end{theorem}
\begin{proof}
The first statement follows from parts \ref{cohomology_theorem_empty} and \ref{cohomology_theorem_reflection} of Theorem \ref{cohomology_theorem_unconditional} under the assumption that $(p, i)$ is a regular pair.

Parts \ref{cohomology_theorem_odd_regular} and \ref{cohomology_theorem_even_regular} each follow from applying Theorem \ref{greenberg_wiles} and then estimating changes in Selmer conditions.
For instance, Theorem \ref{greenberg_wiles} for $H^{1}_{N}(\F_{p}(i))$ yields
\begin{equation*}
h^{1}_{N}(\F_{p}(i)) = 1 + h^{1}_{N^{\ast}}(\F_{p}(1-i)).
\end{equation*}
We have that the Selmer condition $N^{\ast}$ means classes which are split at $N$ and have any behavior at $p$, hence $H^{1}_{N^{\ast}}(\F_{p}(1-i)) \subseteq H^{1}_{p}(\F_{p}(1-i))$.
Applying Theorem \ref{greenberg_wiles} to $H^{1}_{p}(\F_{p}(1-i))$ yields
\begin{equation*}
h^{1}_{p}(\F_{p}(1-i)) = h^{1}_{p^{\ast}}(\F_{p}(i)).
\end{equation*}
Since the Selmer condition $p^{\ast}$ is ``unramified at $N$ and split at $p$'', we have 
\begin{equation*}
H^{1}_{p^{\ast}}(\F_{p}(i)) \subseteq H^{1}_{\emptyset}(\F_{p}(i)).
\end{equation*}
The statement $h^{1}_{N}(\F_{p}(i)) = 1$ thus follows from the chain of inequalities
\begin{align*}
h^{1}_{N}(\F_{p}(i))    & = 1 + h^{1}_{N^{\ast}}(\F_{p}(1-i)) \\*
                        & \leq 1 + h^{1}_{p}(\F_{p}(1-i)) \\
                        & = 1 + h^{1}_{p^{\ast}}(\F_{p}(i)) \\
                        & \leq 1 + h^{1}_{\emptyset}(\F_{p}(i)) \\*
                        & = 1 + 0.
\end{align*}

Part \ref{cohomology_theorem_sigma_regular} of the theorem now follows from the inclusions $H^{1}_{\Sigma}(\F_{p}(i)) \subseteq H^{1}_{N}(\F_{p}(i))$ and $H^{1}_{\Sigma}(\F_{p}(1-i)) \subseteq H^{1}_{S}(\F_{p}(1-i))$; in both cases we know that the dimension of the larger group is 1.
\end{proof}

\begin{theorem}\label{cohomology_theorem_sigma}
Let $p$ be an odd prime. Then for odd $3 \leq i \leq p-2$ we have
\begin{align*}
h^{1}_{\Sigma}(\F_{p}(i))   & = h^{1}_{\Sigma^{\ast}}(\F_{p}(1-i)) \\*
h^{1}_{\Sigma^{\ast}}(\F_{p}(i))    & = h^{1}_{\Sigma}(\F_{p}(1-i)) + 1 \\*
h^{1}_{\Sigma^{\ast}}(\F_{p}(i))    & \leq 1 + h^{1}_{\Sigma}(\F_{p}(i)).
\end{align*}
\end{theorem}
\begin{proof}
The first two statements are proved by applying Theorem \ref{greenberg_wiles} to $H^{1}_{\Sigma}(\F_{p}(i))$ and $H^{1}_{\Sigma^{\ast}}(\F_{p}(i))$. 
The final statement follows from Lemma \ref{change_of_selmer_group} applied to $\Sigma$ and $\Sigma^\ast$.
\end{proof}

\begin{corollary}\label{even_H1Sigma_implies_odd_H1Sigma}
Let $p$ be an odd prime.
Then for even $i \not\equiv 0 \bmod{p-1}$,
\begin{equation*}
h^{1}_{\Sigma}(\F_{p}(i)) \neq 0 \implies h^{1}_{\Sigma}(\F_{p}(1-i)) \neq 0.
\end{equation*}
\end{corollary}
\begin{proof}
If $h^{1}_{\Sigma}(\F_{p}(i)) \geq 1$, then by Theorem \ref{cohomology_theorem_sigma} we have $h^{1}_{\Sigma^{\ast}}(\F_{p}(1-i)) \geq 2$.
Comparing via
\begin{equation*}
h^{1}_{\Sigma^{\ast}}(\F_{p}(1-i)) \leq 1 + h^{1}_{\Sigma}(\F_{p}(1-i))
\end{equation*}
gives that $h^{1}_{\Sigma}(\F_{p}(1-i)) \geq 1$.
\end{proof}

\begin{remark}\label{H1Sigma_is_b_at_N}
Under the assumption that $(p,i)$ is a regular pair, we know that any nonzero class in $H^{1}_{\Sigma}(\F_{p}(i))$ (for $i \neq 0, 1$) will be a nonzero multiple of $b$ when restricted to $G_{\Q_{N}}$: being in the span of $b$ is the local condition at $N$ for these modules, and since this class is split at $p$ and unramified everywhere else, the regularity assumption on $p$ forces this class to be nonzero locally at $N$.
\end{remark}

\subsection{Cup Products and \texorpdfstring{$\Sigma$}{Sigma}}\label{sec_cup_products}

The purpose of the Selmer condition $\Sigma^{\ast}$ is to detect those classes whose cup product with $b$ is equal to $0$, according to the following propositions.

\begin{proposition}\label{check_cup_product_locally}
Let $p$ be an odd prime and $0 \leq i \leq p-2$. 
Assume either $i = 0$ or $1$, that $(p, i)$ is a regular pair if $i$ is odd, or that $(p, 1-i)$ is a regular pair if $i$ is even.
Let $A$ and $A'$ be $G_{\Q, S}$-modules with a pairing $A \otimes A' \to \F_{p}(i)$ for some $i$.
Given classes $a \in H^{1}_{S}(A)$ and $a' \in H^{1}_{S}(A')$, the global cup product $a \cup a' \in H^{2}_{S}(G_{\Q}, \F_{p}(i))$ induced by this pairing vanishes if and only if the local cup product $\res_{N}(a) \cup \res_{N}(a') \in H^{2}(G_{\Q_{N}}, \F_{p}(i))$ does.
\end{proposition}

\begin{proof}
We first claim that the restriction map $H^2_S(\F_p(i)) \to H^2(G_{\Q_N}, \F_p(i))$ is injective. Under the regularity assumption on $(p, i)$, the Global Euler Characteristic Formula (Theorem 5.1 of \cite{milne_arithmetic_duality_theorems}) combined with Theorems \ref{cohomology_theorem_unconditional} and \ref{cohomology_theorem_conditional} gives us that $H^2_S(\F_p(i))$ is $1$-dimensional. Similarly, $H^2(G_{\Q_N}, \F_p(i))$ is $1$-dimensional by Local Tate Duality (Corollary 2.3 of \cite{milne_arithmetic_duality_theorems}). Thus, to prove injectivity it suffices to prove surjectivity.

The end of the Poitou-Tate exact sequence (Theorem 4.10 of \cite{milne_arithmetic_duality_theorems}) for $\F_{p}(i)$ is
\begin{equation*}
H^{2}_{S}(\F_{p}(i)) \to H^{2}(G_{\Q_{p}}, \F_{p}(i)) \oplus H^{2}(G_{\Q_{N}}, \F_{p}(i)) \to H^{0}(G_{\Q, S}, \F_{p}(1-i))^{\vee} \to 0.
\end{equation*}
If $i \neq 1$ the surjectivity is immediate, as the final term in this sequence is 0.
If $i = 1$, the definitions of the maps involved show that the image of $H^{2}_{S}(\F_{p}(i))$ lands in $H^{2}(G_{\Q_{N}}, \F_{p}(i))$.

Thus, the commutativity of the diagram
\begin{equation*}
\begin{tikzcd}
H^{1}_{S}(A) \otimes H^{1}_{S}(A') \arrow[r, "\cup"] \arrow[d] & H^{2}_{S}(G_{\Q}, \F_{p}(i)) \arrow[d, hook] \\
H^{1}(G_{\Q_{N}}, A) \otimes H^{1}(G_{\Q_{N}}, A') \arrow[r, "\cup"] & H^{2}(G_{\Q_{N}}, \F_{p}(i)) 
\end{tikzcd} 
\end{equation*}
shows that the non-vanishing of $a \cup a'$ can be detected locally, as desired.
\end{proof}

\begin{remark}\label{irregular_H2}
In the notation of the previous proposition, when $(p, i)$ is not a regular pair it is still (clearly) true that $a \cup a' = 0$ implies that $\res_N(a) \cup \res_N(a') = 0$. However, the converse need not hold in this setting as $h^2_S(\F_p(i))$ need not be equal to $1$, and thus the map $H^2_S(\F_p(i)) \to H^2(G_{\Q_N},\F_p(i))$ need not be injective.
\end{remark}

\begin{proposition}\label{classes_cup_iff_H1Sigmastar}
Let $p$ be any odd prime and $0 \leq i \leq p-2$. Let $A, A'$ be as in the previous proposition, but assume now that $A, A' \cong \Sym^nV$ as $G_{\Q_N}$-representations. If $a \in H^1_{\Sigma^*}(A)$ and $\res_N(a) \neq 0$, and if $a' \in H^1_S(A')$, then $\res_N(a) \cup \res_N(a') = 0$ if and only if $a' \in H^1_{\Sigma^*}(A')$. In particular, if $a \cup a' = 0$ then $a' \in H^1_{\Sigma^*}(A')$.

Furthermore, if either $i = 0$ or $1$, or if $(p, i)$ is a regular pair and $i$ is odd, or if $(p, 1-i)$ is a regular pair and $i$ is even, then $a \cup a' = 0$ if and only if $a' \in H^1_{\Sigma^*}(A')$.
\end{proposition}
\begin{proof}
Since $\res_{N}(a)$ is nonzero, Proposition \ref{basis_of_local_H1} gives that $\res_{N}(a) = u\mathbf{b}$ for some nonzero $u \in \F_{p}$.
Furthermore, Proposition \ref{condition_at_N_is_self_dual} shows that the statement $u\mathbf{b} \cup \res_{N}(a') = 0$ implies that $\res_{N}(a')$ is also multiple of $\mathbf{b}$ (possibly $0$), which is the condition for $a'$ to be an element of the Selmer group $H^{1}_{\Sigma^{\ast}}(A')$.
The final statement of the proposition then follows from Proposition \ref{check_cup_product_locally}.
\end{proof}

\section{Selmer groups and \texorpdfstring{$\text{Cl}_K$}{CL\textunderscore K}}\label{sec_big_general}
The goal of this section is to relate the $p$-rank $r_K$ of the class group of $K$ to the rank of a certain Selmer subgroup of the Galois cohomology of a cyclotomic twist of $\Sym^{p-4}V$, which in turn is bounded by dimensions of Selmer subgroups in the Galois cohomology of characters.

The main theorem of this section is:
\begin{theorem}\label{rank_equals_h1Sigma}
Let $p$ be odd. Then
\begin{equation*}
r_{K} = 1 + h^{1}_{\Sigma}(\bigrep{p-4}{2}).
\end{equation*}
Additionally, there is a filtration of $\bigrep{p-4}{2}$ that induces the following lower and upper bounds on $r_{K}$:
\begin{equation*}
1 + h^{1}_{\Sigma}(\F_{p}(-1)) \leq r_{K} \leq 1 + \sum_{i=1}^{p-3} h^{1}_{\Sigma}(\F_{p}(-i)).
\end{equation*}
\end{theorem}

This is essentially Theorem \ref{intro_sigma_bound_theorem}. The lower bound in this theorem was first established by Wake--Wang-Erickson; we recover this as Proposition \ref{wake_wang_erickson_theorem}. Throughout this section, $E$ will be an unramified $\F_p$-extension of $K$ and $M$ will be its Galois closure over $\Q$. The proof begins in Section \ref{pfsetup} with some preliminary lemmas on the structure of $\Gal(M/K(\zeta_p))$ as a $\Gal(K(\zeta_p)/\Q)$-representation.

In Section \ref{pfpart1}, we introduce an auxiliary Selmer condition $\Lambda$, which will encode the local conditions that cut out those Galois cohomology classes corresponding to unramified $\F_p$-extensions of $K$. We will also define a filtration on the $\F_p$-vector space $H^1_\Lambda(\bigrep{p-3}{2})$ related to the filtration defined by Iimura in \cite{iimura} on $\Cl_{K(\zeta_p)}$; see Remark \ref{filtration_on_Lambda}. 
This filtration of Iimura is also used by Lecouturier in \cite{lecouturier}.

The next step in the proof of Theorem \ref{rank_equals_h1Sigma} is to relate the Selmer condition $\Lambda$ to the Selmer condition $\Sigma$ defined in Section \ref{sec_sigma}. This is done in Section \ref{pfpart2}, which also contains some general lemmas that realize $\Sigma^*$ as the ``correct'' Selmer condition for discussing the lifting of representations to higher dimensions.

Finally, we descend the filtration on $H^1_\Lambda(\bigrep{p-3}{2})$ to a filtration on $H^1_\Sigma(\bigrep{p-4}{2})$. In Section \ref{pfpart3}, we use this filtration to bound the rank $h^1_\Sigma(\bigrep{p-4}{2})$ in terms of the ranks $h^1_\Sigma(\F_p(-i))$ of the $\Sigma$-Selmer groups of characters. This will complete the proof of Theorem \ref{rank_equals_h1Sigma}.

\subsection{Indecomposability of some \texorpdfstring{$\text{Gal}(K(\zeta_p)/\Q)$}{Gal(K(zeta\textunderscore p)/Q)}-modules arising from \texorpdfstring{$\text{Cl}_K$}{CL\textunderscore K}}\label{pfsetup}
Let $E/K$ be unramified and Galois of degree $p$ and let $M$ be the Galois closure of $E$ over $\Q$, as in the diagram (\ref{diagram_of_fields}) below.

\begin{figure}[ht]
\begin{equation*}\label{diagram_of_fields}
\begin{tikzcd}\tag{$*$}
& M \arrow[d, dash] \arrow[dd, dash, bend left=40, "A"] \\
& E(\zeta_p) \arrow[dl, dash] \arrow[d, dash] \\
E \arrow[d, dash] & K(\zeta_p) \arrow[dl, dash] \arrow[d, dash] \arrow[dd, dash, bend left = 40, "G"] \\
K \arrow[dr, dash] & \Q(\zeta_p) \arrow[d,dash] \\
& \Q
\end{tikzcd}
\end{equation*}
\end{figure}

$M$ is the compositum of the $G := \Gal(K(\zeta_p)/\Q)$-translates of $E(\zeta_p)/K(\zeta_p)$, which implies that $M$ is an unramified elementary abelian $p$-extension of $K(\zeta_p)$. Thus $A := \Gal(M/K(\zeta_p)) \cong (\Z/p\Z)^m$ for some $m \geq 1$. This prompts the following definition.

\begin{definition}\label{def:type}
In the above notation, we say that the unramified $\F_p$-extension $E/K$ is \emph{type $m$} where $m = \dim_{\F_p}(\Gal(M/K(\zeta_p)))$.
\end{definition}

Our goal in this subsection is to prove the following theorem.

\begin{theorem}\label{galois_closure_indecomposable}
$A = \Gal(M/K(\zeta_p))$ is an $\F_{p}$-vector space, and is isomorphic to $\bigrep{m-1}{1-m}$ as a $G = \Gal(K(\zeta_p)/\Q)$-representation where $m$ is the type of $E/K$. Furthermore, we have $1 \leq m \leq p-2$. In particular, $A$ is indecomposable as a representation of $G$.
\end{theorem}

Note that our fixed primitive $p$th root of unity $\zeta_p$ gives us a canonical generator of $\Gal(K(\zeta_p)/\Q(\zeta_p))$, namely the particular $\sigma$ with $\sigma(N^{1/p}) = \zeta_p N^{1/p}$. We use this to fix an isomorphism $G \cong \Z/p\Z \rtimes (\Z/p\Z)^\times$.

\begin{lemma}\label{galois_group_sequence_splits}
The following short exact sequence splits.
\begin{equation*}
1 \to A \to \Gal(M/\Q) \to G \to 1
\end{equation*}
\end{lemma}
\begin{proof}
We argue by means of group cohomology; consider the Hochschild-Serre spectral sequence. Since $H^j(\Z/p\Z, A)$ is an $\F_p$-vector space, its order is coprime to the order of $(\Z/p\Z)^\times$ and thus
\begin{equation*}
H^i((\Z/p\Z)^\times, H^j(\Z/p\Z, A)) = 0
\end{equation*}
for all $i>0$. Hence the only nonzero column on the $E_2$ page is the $0$th one, which implies that the restriction map
\begin{equation*}
H^2(G,A) \to H^2(\Z/p\Z, A)^{(\Z/p\Z)^\times}
\end{equation*}
is an isomorphism.

We wish to show that the class $[\Gal(M/\Q)] \in H^2(G,A)$ is $0$. Its image in $H^2(\Z/p\Z, A)$ under the restriction map is the class of $[\Gal(M/\Q(\zeta_p))]$ coming from
\begin{equation*}
1 \to A \to \Gal(M/\Q(\zeta_p)) \to \Gal(K(\zeta_p)/\Q(\zeta_p)) \to 1.
\end{equation*}

We can explicitly construct a splitting of this sequence. Let $\mathfrak{N}$ be a prime of $M$ lying above $N$. The total ramification degree of $\mathfrak{N}$ in $M/\Q(\zeta_p)$ is $p$, since $N$ is totally ramified in $K(\zeta_p)/\Q(\zeta_p)$ and unramified in $M/K(\zeta_p)$, so the inertia group at $\mathfrak{N}$ is a copy of $\Z/p\Z$ in $\Gal(M/\Q(\zeta_p))$ that maps isomorphically onto $\Gal(K(\zeta_p)/\Q(\zeta_p))$. This inertia group is the image our desired splitting.
\end{proof}

Before continuing, we record the following general fact that we make use of throughout the section.

\begin{lemma}\label{local_fields_lemma}
Suppose that $F$ and $F'$ are extensions of $\Q_{p}(\zeta_{p})$, each of degree dividing $p$ and Galois over $\Q_{p}$, and that $\Gal(F/\Q_{p}(\zeta_{p}))$ and $\Gal(F'/\Q_{p}(\zeta_{p}))$ are not isomorphic as representations of $\Gal(\Q_{p}(\zeta_{p})/\Q_{p}) = (\Z/p\Z)^{\times}$, or that both extensions are trivial.
If $FF'/F$ is unramified, then $F'/\Q_{p}(\zeta_{p})$ is also unramified.
\end{lemma}

This follows from the fact that any unramified extension of $\Q_p(\zeta_p)$ must be cyclic and Galois over $\Q_p$, and that $\F_p(i) \oplus \F_p(j)$ has exactly two $(\Z/p\Z)^\times$-fixed lines when $i \neq j$.

Our next goal is to show that $1 \leq m \leq p-2$ where $m$, as above, is the type of $E/K$. The lower inequality is immediate. However, we can say slightly more about this edge case.

\begin{proposition}\label{type_1_iff_genus_field}
$E/K$ is of type $1$ (i.e. $m = 1$) if and only if $E = K(\zeta_N^{(p)})$ is the genus field of $K$, where $\zeta_N^{(p)}$ is any generator of the degree-$p$ subfield of $\Q(\zeta_N)/\Q$.
\end{proposition}

\begin{proof}
The backward direction is trivial: It is clearly unramified away from $N$, Lemma \ref{zeta_N_plus_localized} shows that $K(\zeta_N^{(p)})/K$ is unramified at $N$ as well, and the Galois closure of $K(\zeta_N^{(p)})/\Q$ is $K(\zeta_p, \zeta_N^{(p)})$.

If $m = 1$ then $E(\zeta_p) = M$ is Galois over $\Q$ and $A = \Z/p\Z$. Consider the action of $G$ on $A$ by conjugation and recall that $G = \Z/p\Z \rtimes_\cyclo (\Z/p\Z)^\times$. The order-$p$ subgroup of $G$ acts trivially on $A$ as there are no non-trivial $1$-dimensional $\F_p$-representations of $\Z/p\Z$. Referencing (\ref{diagram_of_fields}), we see that $(\Z/p\Z)^\times \subseteq G$ is the image of $\Gal(E(\zeta_p)/E) \subseteq \Gal(E(\zeta_p)/\Q)$ which acts trivially on $A = \Gal(E(\zeta_p)/K(\zeta_p))$, as $E(\zeta_p)$ is the compositum of the Galois extensions $E/K$ and $K(\zeta_p)/K$.

Thus we conclude that $G$ acts trivially on $A$ and hence that 
\begin{equation*}
  \Gal(E(\zeta_p)/\Q) = \Z/p\Z \times G  
\end{equation*}
by Lemma \ref{galois_group_sequence_splits}. Consider $L = E(\zeta_p)^G$, which is $\Z/p\Z$ extension of $\Q$. As $\Gal(E(\zeta_p)/E) = (\Z/p\Z)^\times \subseteq G$ we know that $L \subseteq E$. As $L \neq K$ this tells us that $E = LK$.

We claim that $L = \Q(\zeta_N^{(p)})$. To see this, it suffices to notice that $L$ is unramified away from $N$. By choice of $E$, it is automatically unramified away from $p$ and $N$. At $p$, it suffices to check that $L(\zeta_p)/\Q(\zeta_p)$ is unramified, as $[L:\Q]$ is coprime to $[\Q(\zeta_p):\Q]$. We have the following diagram of fields:
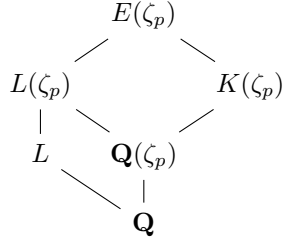
\begin{equation*}
\begin{tikzcd}[column sep = tiny, row sep = small]
& E(\zeta_p) \arrow[dl, dash] \arrow[dr, dash] & \\
L(\zeta_p) \arrow[d, dash] \arrow[dr, dash] & & K(\zeta_p) \arrow[dl, dash] \\
L \arrow[dr, dash] & \Q(\zeta_p) \arrow[d, dash] & \\
& \Q &
\end{tikzcd}
\end{equation*}

Consider the corresponding extensions of fields locally at $p$. Because the groups $\Gal(L(\zeta_p)/\Q(\zeta_p))$ and $\Gal(K(\zeta_p)/\Q(\zeta_p))$ are not isomorphic as modules over the group $\Gal(\Q(\zeta_p)/\Q) = \Gal(\Q_p(\zeta_p)/\Q_p)$, Lemma \ref{local_fields_lemma} gives us the desired conclusion.
\end{proof}

To prove Theorem \ref{galois_closure_indecomposable}, we need to view $A$ as a $G$-representation coming from the conjugation action of $G$ on $A$. Our first goal is to show that $A$ is indecomposable as a $G$-representation. We briefly recall the classification of indecomposable representations of groups of this kind:

\begin{theorem}\label{classification_of_reps}
Let $k \in \Z/(p-1)\Z$ and let $\Gamma_{k}$ be the group $\Z/p\Z \rtimes (\Z/p\Z)^{\times}$, where $u \in (\Z/p\Z)^{\times}$ acts on $\Z/p\Z$ by multiplication by $u^{k}$.

The indecomposable $\F_{p}$-representations of $\Gamma_{k}$ are exactly
\begin{equation*}
\Sym^{j}V_{k} \otimes \F_{p}(i)
\end{equation*}
for $0 \leq i \leq p-2$ and $0 \leq j \leq p-1$, where $\F_p(i)$ is the $1$-dimensional representation where $u \in (\Z/p\Z)^\times$ acts by $u^{i}$ and $V_k$ is the $2$-dimensional representation of $\Gamma_k$ over $\F_p$ given by the map
\begin{align*}
\Gamma_{k}      & \to \GL_{2}(\F_{p}) \\*
(b, u)          & \mapsto \begin{pmatrix} u^{k} & b \\ 0 & 1 \\ \end{pmatrix}.
\end{align*}
\end{theorem}

\begin{proof}
See \cite{alperin} for a proof.
The cyclic case $\Gamma_{0}$ is treated in a discussion following Corollary 7 of Chapter 5, and the general case is treated in discussions following Lemma 8 of Chapter 5 and Corollary 5 of Chapter 6.
The structure of the proof is as follows:
\begin{itemize}
\item The irreducible $\F_{p}$-representations of $\Gamma_{k}$ are all $1$-dimensional, namely they are the $1$-dimensional representations $\F_{p}(i)$ of the quotient $(\Z/p\Z)^{\times}$ of $\Gamma_{k}$.

\item There is a bijection between irreducible $\F_{p}$-representations of $\Gamma_{k}$ and indecomposable projective $\F_{p}[\Gamma_{k}]$-modules, given by associating $P/\rad(P)$ to each indecomposable projective module $P$.
This is Theorem 3 of Chapter 5 of \cite{alperin}.

\item Every $\F_{p}[\Gamma_{k}]$-module $X$ with $X/\rad(X) \cong \F_{p}(i)$ is a homomorphic image of the indecomposable projective module associated to $\F_{p}(i)$.
This is Lemma 5 of Chapter 5 of \cite{alperin}.

\item Each indecomposable projective module has radical length exactly $p$. In particular it is $p$-dimensional as an $\F_{p}$-vector space, as all quotients in its radical series are irreducible.
This is the discussion after Lemma 8 of Chapter 5 of \cite{alperin}.

This is enough to show that the unique indecomposable projective module associated to $\F_{p}(i)$ is $\Sym^{p-1}V_{k} \otimes \F_{p}(i)$, as it is $p$-dimensional, indecomposable, and has $\F_{p}(i)$ as a quotient.

\item Any indecomposable $\F_{p}$-representation of $\Gamma_{k}$ has a unique radical series.
In particular if $X$ is an indecomposable $\F_{p}$-representation of $\Gamma_{k}$, $X/\rad(X)$ is irreducible.
This is the discussion after Corollary 5 of Chapter 6 of \cite{alperin}. 

This allows us to conclude that every such $X$ is surjected to by some $\Sym^{p-1}V_{k} \otimes \F_{p}(i)$; the quotient modules of $\Sym^{p-1}V_{k} \otimes \F_{p}(i)$ are just $\Sym^{j}V_{k} \otimes \F_{p}(i)$ for $0 \leq j \leq p-1$. \qedhere
\end{itemize}
\end{proof}

Writing our $A$ as a sum of indecomposable representations of $G = \Gamma_1$, we know that the number of indecomposable factors is equal to the dimension of $A^{\Z/p\Z}$. Indeed, each indecomposable factor when considered as a representation of $\Z/p\Z$ corresponds to a Jordan block with eigenvalue $1$. Thus we've reduced the indecomposability of $A$ to showing that $A^{\Z/p\Z}$ is $1$-dimensional.

\begin{lemma}\label{fixed_part_of_A_is_1_dimensional}
$A^{\Z/p\Z}$ is $1$-dimensional. Furthermore, it carries the trivial action of $(\Z/p\Z)^\times = \Gal(\Q(\zeta_p)/\Q)$.
\end{lemma}
\begin{proof}
The first part of the claim follows once we have shown that the subgroup $H = A^{\Z/p\Z} \cap \Gal(M/E(\zeta_p))$ is trivial, as $\Gal(M/E(\zeta_p))$ is codimension $1$ in $A$. We will demonstrate this by showing that $H$ is normal in $\Gal(M/\Q)$. Indeed, as $M$ is the Galois closure of $E(\zeta_p)/\Q$, any normal subgroup of $\Gal(M/\Q)$ contained in $\Gal(M/E(\zeta_p))$ is necessarily trivial.

Because $A$ is abelian, to show that $H$ is normal in $\Gal(M/\Q) = A \rtimes G$ it suffices to show that it is fixed by conjugation by $G$. Again applying the classification of indecomposable representations of $G$ we see that $A^{\Z/p\Z}$ is a product of characters and is thus a $G$-subrepresentation of $A$.

Referencing (\ref{diagram_of_fields}), notice now that the action of $(\Z/p\Z)^\times = \Gal(\Q(\zeta_p)/\Q)$ on $A$ is the same as the action of $\Gal(E(\zeta_p)/E)$ on $A$. But the action of $\Gal(E(\zeta_p)/E)$ on $A$ clearly stabilizes $\Gal(M/E(\zeta_p)) \subseteq A$.

This shows that $(\Z/p\Z)^\times \subseteq G$ stabilizes both $A^{\Z/p\Z}$ and $\Gal(M/E(\zeta_p))$ and thus it stabilizes their intersection $H$. As $H \subseteq A^{\Z/p\Z}$ is also fixed pointwise by $\Z/p\Z$, we conclude that $H$ is fixed by the action of $G$ and is thus normal in $\Gal(M/\Q)$.

To see the second part of the lemma, we first notice as above that $(\Z/p\Z)^\times$ acts on $A$ as $\Gal(K(\zeta_p)/K)$ and thus acts trivially on $\Gal(E(\zeta_p)/K(\zeta_p)) = \Gal(E/K)$.

The short exact sequence
\begin{equation*}
1 \to \Gal(M/E(\zeta_p)) \to A \to \Gal(E(\zeta_p)/K(\zeta_p)) \to 1
\end{equation*}
is $\Gal(K(\zeta_p)/K)$-equivariant. As $A^{\Z/p\Z}$ has trivial intersection with the above kernel, it maps isomorphically onto $\Gal(E(\zeta_p)/K(\zeta_p))$, which we just established carries the trivial $(\Z/p\Z)^\times$ action.
\end{proof}

The first part of Lemma \ref{fixed_part_of_A_is_1_dimensional} gives $A\cong \bigrep{j}{i}$ for some $0 \leq i \leq p-2$ and $0 \leq j \leq p-1$, and the second part establishes that $i = -j$. This also implies that $A$ is a faithful representation of $G$ whenever $m \geq 2$, i.e., whenever $j \geq 1$.

We now have that $A \cong \bigrep{m-1}{1-m}$ as $G$-representations, but to complete the proof of Theorem \ref{galois_closure_indecomposable} it remains to show that $m \leq p-2$. In what follows, it will be useful to write $\Gal(M/\Q)$ as an explicit matrix group that we can view as the image of a representation of $G_{\Q,S}$.

Suppose that $A$ is an $\F_p$-vector space and that $G \to \Aut(A) = \GL_m(\F_p)$ is an injective homomorphism. Then $A \rtimes G$ is isomorphic to the $(m+1) \times (m+1)$ block-matrix group
\begin{equation*}
\begin{pmatrix} G & A \\ 0 & 1 \\ \end{pmatrix}
\end{equation*}
where $G$ is identified with its image in $\GL_m(\F_p)$ and elements of $A$ are expressed as column vectors in the corresponding basis.

Assuming that $E$ is not the genus field of $K$, $A$ is a faithful $G$-representation so the previous paragraph establishes that in a suitable basis of $\bigrep{m-1}{1-m}$ (see Remark \ref{nonstandard_basis_for_sym}), $\Gal(M/\Q)$ is isomorphic to the group of matrices
\begin{equation*}\label{big_matrix}
\left(\begin{array}{ccccc|c}\tag{$**$}
    1 & \cyclo^{-1}b & \cyclo^{-2}\frac{b^{2}}{2} & \cdots & \cyclo^{-(m-1)}\frac{b^{m-1}}{(m-1)!} & a_{0} \\
    & \cyclo^{-1} & \cyclo^{-2}b & \cdots & \cyclo^{-(m-1)}\frac{b^{m-2}}{(m-2)!} & a_{1} \\
    & & \cyclo^{-2} & & \vdots & \vdots \\
    & & & \ddots & \cyclo^{-(m-1)} b & a_{m-2} \\
    & & & & \cyclo^{-(m-1)} & a_{m-1} \\
    \hline
    & & & & & 1 \\
\end{array}\right)
\end{equation*}
where the $i,j$-th entry in the top left block is $\cyclo^{-(j-1)} \frac{b^{j-i}}{(j-i)!}$. This also defines a representation 
\begin{equation*}
    G_{\Q,S} \to \Gal(M/\Q) \to \GL_{m+1}(\F_p)
\end{equation*}
of dimension $m+1$ that we will consider more carefully in Section \ref{pfpart1}.

\begin{remark}\label{nonstandard_basis_for_sym}
In light of the discussion after Lemma \ref{fixed_part_of_A_is_1_dimensional}, it will be useful for us to fix bases of the $\bigrep{j}{i}$ for $i,j$ in the range of Theorem \ref{classification_of_reps} so that we can view $\Gal(M/\Q) = A \rtimes G$ as an explicit matrix group. Furthermore, we would like these bases to be compatible with the quotient maps $\Sym^k V \to \Sym^{k-1}V$.

For the $2$-dimensional representation $V$ we are considering, let $\{e, f\}$ be the basis for $V$ as in the discussion at the start of Section \ref{sec_sigma}. The usual basis for $\Sym^kV$ is then $\{e^k, e^{k-1}f, \ldots, f^k\}$. In that basis, the $i,j$-th entry of the top left block is, ignoring powers of the cyclotomic character, $\binom{j-1}{i-1}b^{j-i}$.

We can rescale this basis so that the image of the representation $\Sym^k V$ is the matrix group
\begin{equation*}
\left(\begin{array}{ccccc}
    \cyclo^k & \cyclo^{k-1}b & \cyclo^{k-2}\frac{b^{2}}{2} & \cdots & \frac{b^{k}}{k!} \\
    & \cyclo^{k-1} & \cyclo^{k-2}b & \cdots & \frac{b^{k-1}}{(k-1)!} \\
    & & \cyclo^{k-2} & & \vdots \\
    & & & \ddots & b \\
    & & & & 1
\end{array}\right).
\end{equation*}

This map $G \to \GL_{k+1}(\F_p)$ factors through the group $U_{k+1}$ of upper-triangular matrices. Similarly, $\Sym^{k-1} V$ gives a map $G \to U_k$. There is also a projection $U_{k+1} \to U_k$ given by ``forget the first row and column'', and the bases above are chosen so that the triangle
\begin{equation*}
\begin{tikzcd}[row sep = small]
& U_{k+1} \arrow[dd] \\
G \arrow[ru] \arrow[rd] & \\
& U_k
\end{tikzcd}
\end{equation*}
commutes.
\end{remark}

If $E = \Q(N^{1/p}, \zeta_N^{(p)})$ is the genus field, we instead consider the representation $G_{\Q,S} \to \GL_2(\F_p)$ of the form
\begin{equation*}
\begin{pmatrix}
1 & c \\ 0 & 1
\end{pmatrix}
\end{equation*}
where $c \in \Hom(G_{\Q,S}, \F_p) = H^1_S(\F_p)$ is the class defining the extension $\Q(\zeta_{N}^{(p)})/\Q$.

\begin{remark}
As $M/K(\zeta_p)$ is unramified, we can view its Galois group $A$ as a quotient of the $p$-part of the class group $\Cl_{K(\zeta_p)}$. The results above can then be viewed through the lens of decomposing this class group into a sum of indecomposable $\Gal(K(\zeta_p)/\Q)$-representations, similar to classical results on decomposing the $p$ part of $\Cl_{\Q(\zeta_p)}$ into a sum of $\Gal(\Q(\zeta_p)/\Q)$-representations. In Section \ref{sec_effective_criteria} we will see that the numbers $M_i$ defined in Section \ref{sec_introduction} play a similar role to that of the Bernoulli numbers in the structure of $\Cl_{\Q(\zeta_p)}$.

The structure of $\Cl_{K(\zeta_p)}$ as a Galois module was also studied by Iimura in \cite{iimura}. The connection between Iimura's work and our current approach is discussed in slightly more detail in Remark \ref{filtration_on_Lambda}.
\end{remark}

With the above matrix representation in hand, we can now prove that $m \leq p-2$. Notice that we already have that $m \leq p$ since all indecomposable representations of $G$ have dimension no larger than $p$. We will show directly that $m \neq p, p-1$.

\begin{lemma}\label{bohica_1}
$m \neq p$.
\end{lemma}
\begin{proof}
Suppose that $m = p$.
The lower $2 \times 2$ corner of the matrix (\ref{big_matrix}) will thus be
\begin{equation*}
\begin{pmatrix} 1 & a_{p-1} \\ 0 & 1 \\ \end{pmatrix}
\end{equation*}
which we think of as a quotient of $\Gal(M/\Q)$ (alternatively, as a new $G_{\Q,S}$-representation with $G_{M,S}$ in the kernel). This gives us a class $a_{p-1} \in H^{1}_{S}(\F_{p})$, and it cuts out a $\Z/p\Z$ extension $L$ of $\Q$ contained in $M$ and hence unramified outside of $S$.
We will show that this extension is necessarily unramified at $N$ and $p$ as well, contradicting its existence. We begin by considering the behavior at $p$. Consider the diagram of fields
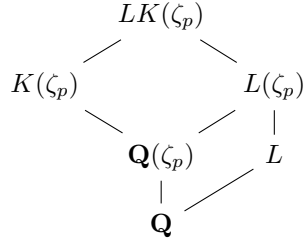
\begin{equation*}
\begin{tikzcd}[column sep = tiny, row sep = small]
& LK(\zeta_{p}) \arrow[dl, dash] \arrow[dr, dash] & \\
K(\zeta_{p}) \arrow[dr, dash] & & L(\zeta_{p}) \arrow[dl, dash] \arrow[d, dash] \\
& \Q(\zeta_{p}) \arrow[d, dash] & L \arrow[dl, dash] \\
& \Q &
\end{tikzcd}
\end{equation*}
locally at $p$. As $L \subseteq M$ we know that $LK(\zeta_p)/K(\zeta_p)$ is unramified at $p$. Applying Lemma \ref{local_fields_lemma}, we conclude that $L(\zeta_p)/\Q(\zeta_p)$, and hence $L/\Q$, is also unramified at $p$.

Suppose independently that $L/\Q$ is (tamely) ramified at $N$. The inertia group(s) above $N$ in $\Gal(M/\Q)$ are cyclic of order $p$ as $M/K(\zeta_p)$ is unramified.
If $\tau$ is a generator of the tame inertia group of $\Q_{N}$ we know by the functoriality of inertia that $b(\tau)$ and $a_{p-1}(\tau)$ are both nonzero, as the extensions $K(\zeta_p)$ and $L$ defined by these classes are ramified at $N$.
Under the quotient map $G_{\Q, S} \to \Gal(M/\Q)$ we have
\begin{equation*}
\tau \mapsto
\begin{pmatrix}
    1 & b(\tau) & \frac{b(\tau)^{2}}{2} & \cdots & \frac{b(\tau)^{p-1}}{(p-1)!} & a_{0}(\tau) \\
    & 1 & b(\tau) & \cdots & \frac{b(\tau)^{p-2}}{(p-2)!} & a_{1}(\tau) \\
    & & 1 & & \vdots & \vdots \\
    & & & \ddots & b(\tau) & a_{p-2}(\tau) \\
    & & & & 1 & a_{p-1}(\tau) \\
    & & & & & 1 \\
\end{pmatrix}.
\end{equation*}

Raising this to the $p$th power, we see that the top-right entry of the image of $\tau^p$ is $b(\tau)^{p-1}a_{p-1}(\tau)$. But this is nonzero, contradicting the fact that the generator of the inertia group at $N$ has order $p$.
\end{proof}

\begin{lemma}\label{bohica_2}
$m \neq p-1$.
\end{lemma}
\begin{proof}
The lower $2 \times 2$ corner of the matrix (\ref{big_matrix}) will thus be
\begin{equation*}
\begin{pmatrix}
    \cyclo & a_{p-2} \\
    0 & 1 \\
    \end{pmatrix}
\end{equation*}
where $a_{p-2} \in H^{1}_{S}(\F_{p}(1))$.
As in the previous lemma, we deduce the existence of an extension $L/\Q(\zeta_{p})$ contained in $M$ which is Galois over $\Q$, with Galois group isomorphic to $G$.
The extension $L$ is not equal to $K(\zeta_{p})$ however; we have that $A$ surjects onto $\Gal(L/\Q(\zeta_{p}))$ but not $\Gal(K(\zeta_{p})/\Q(\zeta_{p}))$ as $A$ is equal to the kernel of $\Gal(M/\Q) \to \Gal(K(\zeta_{p})/\Q)$.
Thus we must have that $a_{p-2}$ is not a nonzero multiple of the class $b \in H^{1}_{S}(\F_{p}(1))$.
However, we know that $a_{p-2} \cup b = 0$ since there is a 3-dimensional representation of $G_{\Q,S}$ coming from the lower $3 \times 3$ quotient
\begin{equation*}
\begin{pmatrix}
    \cyclo^2 & \cyclo b & a_{p-3} \\
    0 & \cyclo & a_{p-2} \\
    0 & 0 & 1 \\
\end{pmatrix}
\end{equation*}
of the matrix (\ref{big_matrix}), so $a_{p-2}$ must be a multiple of $b$ by Proposition \ref{classes_cup_iff_H1Sigmastar}.
This gives that $a_{p-2} = 0$, but then $\dim_{\F_p}(A) \leq p-2$ so $E/K$ is not in fact type $p-1$.
\end{proof}

\subsection{An Auxiliary Selmer Group}\label{pfpart1}

In the previous section, we obtained from an unramified $\F_p$-extension $E/K$ of type $m$ a representation $G_{\Q,S} \to \GL_{m+1}(\F_p)$ of the form (\ref{big_matrix}). As a representation, it is an extension of the trivial representation by $\bigrep{m-1}{1-m}$ considered as a $G_{\Q,S}$-representation via the quotient $G_{\Q,S} \to G$, so it gives a class
\begin{equation*}
a_E = \begin{bmatrix}
a_0 \\ \vdots \\ a_{m-1}
\end{bmatrix} \in H^1_S(\bigrep{m-1}{1-m})
\end{equation*}
as discussed in Section \ref{sec_notation}.

Let $\Lambda$ be the Selmer condition defined by
\begin{itemize}
    \item $L_\ell = H^1_\ur(G_{\Q_\ell}, A)$ for $\ell \neq N,p$
    \item $L_N = H^1(G_{\Q_N}, A)$
    \item $L_p = \res^{-1}(H^1_\ur(G_{K(\zeta_p)_p}, A))$ where $\res$ is the restriction map
    \begin{equation*}
    H^1(G_{\Q_p}, A) \to H^1(G_{K(\zeta_p)_p}, A).
    \end{equation*}
\end{itemize}

\begin{remark}\label{lambda_equals_N_Fp}
In the case $A = \F_p$, the containment $H^1_N(\F_p) \subseteq H^1_\Lambda(\F_p)$ is an equality. This is to say that any $\F_p$-extension $L/\Q$ unramified away from $S$ and unramified at $p$ after base change to $K(\zeta_p)$ was necessarily unramified at $p$ over $\Q$. This follows from Lemma \ref{local_fields_lemma} as in the end of the proof of Proposition \ref{type_1_iff_genus_field}.

Although we don't need the following fact, it is true that for all of the modules $A$ listed in Theorem \ref{galois_closure_indecomposable} which can arise as $\Gal(M/K(\zeta_p))$, one has $H^1_N(A) = H^1_\Lambda(A)$; this follows from Lemma \ref{image_of_H1Lambda_is_H1Sigma}. However, if one wants to use the methods of this section to study $\Cl_{K(\zeta_p)}$ or the case of composite $N$, it is necessary to use modules $A$ for which $H^1_N(A) \neq H^1_\Lambda(A)$.
\end{remark}

In this subsection we prove
\begin{theorem}\label{rank_equals_h1Lambda}
$r_K = h^1_\Lambda(\bigrep{p-3}{2})$.
\end{theorem}

The main step in the proof of this theorem is to show that the class $a_E$ lies in the $\Lambda$-Selmer subgroup $H^1_\Lambda(\bigrep{m-1}{1-m})$ and conversely that any such Selmer class arises from an unramified $\F_{p}$-extension $E/K$. The forward direction is trivial: the only thing to check is that it satisfies the correct condition at $p$, which follows from the fact that $M/K(\zeta_p)$ is unramified above $p$.

Note that there is some ambiguity in the choice of $a_E$ as any constant multiple of it defines the same field extension. The proof of Theorem \ref{rank_equals_h1Lambda} comes down to establishing a bijection between the projectivized space $\Proj H^1_\Lambda(\bigrep{p-3}{2})$ and the set of unramified $\F_{p}$-extensions $E/K$, which can itself be thought of as the projectivization of the $p$-part of $\Cl_K$.

In order to promote $a_E$ to a class in $H^1_\Lambda(\bigrep{p-3}{2})$, consider the natural filtration on the module $\bigrep{p-3}{2} = \bigrep{p-3}{3-p}$ given by
\begin{align*}
0 \subseteq \F_p &= \bigrep{0}{0} \\*
    &\subseteq \bigrep{1}{-1} \\
    &\subseteq \bigrep{2}{-2} \\*
    &\subseteq \cdots \\*
    &\subseteq \bigrep{p-3}{3-p}.
\end{align*}
where the $k$th subspace is the span of the first $k$ basis vectors in the basis used above in the matrix (\ref{big_matrix}). The successive quotients are
\begin{equation*}
\frac{\bigrep{k}{-k}}{\bigrep{k-1}{1-k}} \cong \F_p(-k).
\end{equation*}
Since these have no $G_{\Q,S}$-fixed points, as $1 \leq k \leq p-3$, we get a corresponding filtration in cohomology
\begin{align*}
0 \subseteq H^1_S(\F_p) &\subseteq H^1_S(\bigrep{1}{-1}) \\*
    &\subseteq H^1_S(\bigrep{2}{-2}) \\*
    &\subseteq\cdots \\*
    &\subseteq H^1_S(\bigrep{p-3}{3-p})
\end{align*}
where each inclusion can be realized concretely via
\begin{equation*}
\begin{bmatrix}a_0 \\ \vdots \\ a_{k-1} \end{bmatrix} \mapsto
\begin{bmatrix}a_0 \\ \vdots \\ a_{k-1} \\ 0\end{bmatrix}.
\end{equation*}

This filtration restricts to a filtration on the Selmer subgroups $H^1_\Lambda(-)$. Thus, given our $E/K$ of type $m$, we get an element in $H^1_\Lambda(\bigrep{p-3}{2})$, defined up to a scalar, as desired.

Conversely, given a nonzero class $a \in H^1_\Lambda(\bigrep{p-3}{2})$, we can restrict it to a class in $G_{K,S}$-cohomology to get a representation of $G_{K,S}$ of the form
\begin{equation*}
\begin{pmatrix}
    1 & 0 & 0 & \cdots & 0 & a_0 \\
    & \cyclo^{-1} & 0 & \cdots & 0 & a_1 \\
    & & \cyclo^{-2} & & \vdots & \vdots \\
    & & & \ddots & 0 & a_{p-4} \\
    & & & & \cyclo^{2} & a_{p-3} \\
    & & & & & 1 \\
\end{pmatrix}.
\end{equation*}
From this we see that $a_0|_{G_{K,S}}$ is a homomorphism $G_{K,S} \to \F_p$. Note that some of the $a_i$ might be $0$ if $a$ comes from some smaller piece of the filtration above, but $a_0|_{G_{K,S}} \neq 0$ by the following lemma. Thus, the fixed field of $\ker(a_0|_{G_{K,S}})$, denoted $E_a$, is an $\F_p$-extension of $K$.

\begin{lemma}\label{a_0_determines_E}
If $a \in H^1_\Lambda(\bigrep{p-3}{2})$ is nonzero then $a_0|_{G_{K,S}}: G_{K,S} \to \F_p$ is nonzero as well.
\end{lemma}
\begin{proof}
We will show the equivalent statement that $a_0|_{G_{K(\zeta_p),S}}$ is nonzero. Let $A$ be $\bigrep{p-3}{2}$ and consider the inflation-restriction sequence
\begin{equation*}
0 \to H^1(G, A) \to H^1(G_{\Q,S}, A) \to H^1(G_{K(\zeta_p),S}, A)^G.
\end{equation*}
We claim that $H^1(G,A) = 0$. Using inflation-restriction again, we get that 
\begin{equation*}
H^1(G,A) \cong H^1(\Z/p\Z, A)^{(\Z/p\Z)^\times}.
\end{equation*}
It can be explicitly seen that $H^1(\Z/p\Z, A) = \F_p(2)$ as a $(\Z/p\Z)^\times$-module, implying that \begin{equation*}
H^1(\Z/p\Z, A)^{(\Z/p\Z)^\times} = 0.
\end{equation*}

Therefore, a nonzero $a \in H^1(G_{\Q,S}, A)$ restricts to a nonzero homomorphism $G_{K(\zeta_p),S} \to A = \F_p^{p-2}$ that is invariant under $G$. In particular, its image is fixed by the action of $G$ on $A$ so its image is a nonzero $G$-subrepresentation. However, the only nontrivial $G$-subrepresentations of $A$ are the spans of the first $k \geq 1$ basis vectors, all of which contain some element whose first coordinate is nonzero.
\end{proof}

The Selmer condition $\Lambda$ guarantees that this extension $E_a/K$ is unramified everywhere. This is obvious for all $\ell \neq N, p$.

At $N$, Proposition \ref{basis_of_local_H1} shows that $H^1(G_{\Q_N}, \bigrep{k}{-k})$ is $2$-dimensional, spanned by an unramified class and the class corresponding to $K_N$, so the image of any class here in $H^1(G_{K_N}, \bigrep{k}{-k})$ is unramified.
At $p$, it suffices to remark that $[E_a:K]$ is prime to $[K(\zeta_p):K]$, and thus $E_a/K$ is unramified exactly when $E_a(\zeta_p)/K(\zeta_p)$ is.

Finally, to finish the proof of Theorem \ref{rank_equals_h1Lambda}, we remark that the assignments $E \mapsto a_E$ and $a \mapsto E_a$ are mutually inverse. Indeed, given an unramified $E/K$, Theorem \ref{galois_closure_indecomposable} along with the above discussion implies that $E_{a_{E}}$ is the unique $\F_p$-subextension of $M/K$ such that $\Gal(K(\zeta_p)/K) = (\Z/p\Z)^\times$ acts trivially on $\Gal(E_{a_E}(\zeta_p)/K(\zeta_p))$. But $E$ satisfies this last property as well, and thus $E = E_{a_E}$.

Conversely, take any two cohomology classes $a, a' \in H^1_\Lambda(\bigrep{p-3}{2})$ and assume $E_a = E_{a'}$, which implies that $a_0|_{G_{K,S}}$ is a constant multiple of $a'_0|_{G_{K,S}}$. Scaling $a'$ so that these are equal and applying Lemma \ref{a_0_determines_E} to $a - a'$, we conclude that $a - a' = 0$ and hence $a = a'$.

\begin{remark}\label{filtration_on_Lambda}
We can now think of the filtration on $H^1_\Lambda(\bigrep{p-3}{2})$ from the perspective of the types $m$ of the extensions $E/K$. Under the correspondence used to prove Theorem \ref{rank_equals_h1Lambda}, the subspace $H^1_\Lambda(\bigrep{k}{-k})$ contains the $E$ of type $m \leq k+1$, and the quotient
\begin{equation*}
\frac{H^1_\Lambda(\bigrep{k}{-k})}{H^1_\Lambda(\bigrep{k-1}{1-k})}
\end{equation*}
is nonzero exactly when there is an $E/K$ of type $k+1$.

In \cite{iimura}, Iimura defines a descending filtration on the $p$-part of $A = \Cl_{K(\zeta_p)}$ by considering it as a $\F_p[G]$-module. Let $\sigma \in G$ be order $p$. The $i$th piece $J_{i}$ of the filtration is the image of $(\sigma-1)^i A$.
Comparing his construction with the one given in this section, one sees that quotients of the $(\Z/p\Z)^\times$-coinvariants of $J_0 / J_k$ give extensions $E/K$ of type $m \leq k$, and that quotients of the $(\Z/p\Z)^\times$-coinvariants of $J_{m-1}/J_m$ give extensions $E/K$ of type exactly $m$. This realizes Iimura's filtration as the ``dual'' to our filtration on $H^1_\Lambda(\bigrep{p-3}{2})$.
\end{remark}

\begin{remark}\label{sharifi_connection}
Recall that if $c_j \in H^1(G, \F_p(i_j))$ for $1 \leq j \leq k$, then the $k$-fold Massey product $\langle c_1, \ldots, c_k \rangle$ is a subset of $H^2(G, \F_p(\sum_{j=1}^k i_j))$ that contains $0$ if and only if there is an upper-triangular $\F_p$-representation of $G$ whose image has powers of $\cyclo$ on the diagonal and the cocycles $c_i$ on the upper-diagonal. For example, the matrix (\ref{big_matrix}) witnesses the vanishing of the Massey product $\langle b, \ldots, b, a_{m-1} \rangle$.

In \cite{sharifi_massey_products}, Sharifi works in an Iwasawa-theoretic situation and relates the inverse limit of class groups to the inverse limits of Massey products. In broad terms, his Theorem A estabishes an isomorphism between the $k$th graded piece of an Iimura-like filtration and the quotient of another group by inverse limits of $(k+1)$-fold Massey products of the form $\langle b, \ldots, b, a \rangle$. That is, ``if more Massey products vanish, then the $k$th piece of Iimura's filtration is larger'', which is consistent with the themes of this section.
\end{remark}

\subsection{An Exact Sequence of Selmer Groups}\label{pfpart2}

The goal of this subsection is to provide some motivation for the definitions of the Selmer conditions $\Sigma$ and $\Sigma^*$ and to prove the following proposition:

\begin{proposition}\label{Lambda_to_Sigma_is_exact}
Let $p$ be an odd prime. Let $1 \leq k \leq p-3$. There is an exact sequence
\begin{equation*}
0 \to H^1_N(\F_p) \to H^1_\Lambda(\bigrep{k}{-k}) \to H^1_\Sigma(\bigrep{k-1}{-k}) \to 0.
\end{equation*}
In particular,
\begin{equation*}
h^1_\Lambda(\bigrep{p-3}{2}) = 1 + h^1_\Sigma(\bigrep{p-4}{2}).
\end{equation*}
\end{proposition}

The last equality follows from the $k = p-3$ case of the first part of the proposition combined with part \ref{cohomology_theorem_kw} of Theorem \ref{cohomology_theorem_unconditional}, which gives that $h^1_N(\F_p) = 1$.
 
Let $1 \leq k \leq p-3$ and consider the short exact sequence of $G_{\Q,S}$-representations
\begin{equation*}
0 \to \F_p \to \bigrep{k}{-k} \to \bigrep{k-1}{-k} \to 0.
\end{equation*}

$\bigrep{k-1}{-k}$ has no $G_{\Q,S}$-fixed points, so taking $G_{\Q,S}$-cohomology gives that the top row of the following commutative diagram is exact.
\begin{equation*}
\begin{tikzcd}[column sep = small]
0 \arrow[r] & H^1_S(\F_p) \arrow[r] & H^1_S(\bigrep{k}{-k}) \arrow[r] & H^1_S(\bigrep{k-1}{-k}) \\
0 \arrow[r] & H^1_N(\F_p) \arrow[hookrightarrow]{u} \arrow[r] & H^1_\Lambda(\bigrep{k}{-k}) \arrow[hookrightarrow]{u} \arrow[r, dashed] & H^1_\Sigma(\bigrep{k-1}{-k}) \arrow[hookrightarrow]{u} \arrow[r] & 0
\end{tikzcd}
\end{equation*}

To prove Proposition \ref{Lambda_to_Sigma_is_exact}, we need to show:
\begin{enumerate}
    \item The image of $H^1_\Lambda(\bigrep{k}{-k})$ in $H^1_S(\bigrep{k-1}{-k})$ is contained in the Selmer subgroup $H^1_\Sigma(\bigrep{k-1}{-k})$.
    \item The induced map $H^1_\Lambda(\bigrep{k}{-k}) \to H^1_\Sigma(\bigrep{k-1}{-k})$ is surjective.
    \item The kernel of this induced map is precisely $H^1_N(\F_p) \subseteq H^1_S(\F_p)$.
\end{enumerate}

The third item is the easiest; we just need that the intersection of the images of $H^1_\Lambda(\bigrep{k-1}{-k})$ and $H^1_S(\F_p)$ in $H^1_S(\bigrep{k}{-k})$ is the image of $H^1_N(\F_p)$, which follows from Remark \ref{lambda_equals_N_Fp} that $H^1_N(\F_p) = H^1_\Lambda(\F_p)$.

The proof of the remainder of the proposition is broken up into two parts. Lemma \ref{image_of_H1S_is_H1Sigmastar} establishes Parts $1$ and $2$ above with $\Lambda$ replaced by $S$ and $\Sigma$ replaced by $\Sigma^*$ by considering the local condition at $N$. To get the corresponding statements for $\Lambda$ and $\Sigma$, we need to consider the local conditions at $p$, which is done in Lemmas \ref{H1Sigma_lifts_to_H1Lambda} and \ref{image_of_H1Lambda_is_H1Sigma}.

Lemma \ref{image_of_H1S_is_H1Sigmastar} is stated in slightly more generality than we presently need. To establish Proposition \ref{Lambda_to_Sigma_is_exact}, we only need the case $i = j$. The full strength of this lemma is used in Section \ref{sec_lifting_selmer_classes} when we discuss issues of extending Galois representations of this kind.

\begin{lemma}\label{image_of_H1S_is_H1Sigmastar}
For any $1 \leq i \leq p-3$, and $0 \leq j \leq i$, the image of
\begin{equation*}
    H^1_S(\bigrep{j}{-i}) \to H^1_S(\bigrep{j-1}{-i})
\end{equation*}
is contained in $H^1_{\Sigma^*}(\bigrep{j-1}{-i})$. If in addition we assume that $p$ is regular or that $i = j$, then the image is precisely $H^1_{\Sigma^*}(\bigrep{j-1}{-i})$.
\end{lemma}

\begin{remark}\label{restate_image_of_H1S_is_H1Sigmastar}
The second statement in the proposition is equivalent to the following statement: A $G_{\Q,S}$-representation of dimension $j+1$, coming from an element $a \in H^1_S(\bigrep{j-1}{-i})$, of the form
\begin{equation*}
\begin{pmatrix}
    \cyclo^{j-1-i} & \cyclo^{j-2-i}b & \cdots & \cyclo^{-i}\frac{b^{j-1}}{(j-1)!} & a_{i-(j-1)} \\
    & \cyclo^{j-2-i} & \cdots & \cyclo^{-i}\frac{b^{j-2}}{(j-2)!} & a_{i-(j-2)} \\
    & \hspace{2em}\ddots & & \vdots & \vdots \\
    & & \ddots & \cyclo^{-i} b & a_{i-1} \\
    & & & \cyclo^{-i} & a_{i} \\
    & & & & 1 \\
\end{pmatrix}
\end{equation*}
extends to a $G_{\Q,S}$-representation of dimension $j+2$ of the form
\begin{equation*}
\begin{pmatrix}
    \cyclo^{j-i} & \cyclo^{j-1-i}b & \cdots & \cyclo^{-i}\frac{b^{j}}{j!} & \ast \\
    & \cyclo^{j-1-i} & \cdots & \cyclo^{-i}\frac{b^{j-1}}{(j-1)!} & a_{i-(j-1)} \\
    & \hspace{2em}\ddots & & \vdots & \vdots \\
    & & \ddots & \cyclo^{-i} b & a_{i-1} \\
    & & & \cyclo^{-i} & a_{i} \\
    & & & & 1 \\
\end{pmatrix}
\end{equation*}
if and only if $a \in H^1_{\Sigma^*}(\bigrep{j-1}{-i})$.
\end{remark}

\begin{proof}[Proof of Lemma \ref{image_of_H1S_is_H1Sigmastar}]
The exact sequence
\begin{equation*}
0 \to \F_p(j-i) \to \bigrep{j}{-i} \to \bigrep{j-1}{-i} \to 0
\end{equation*}
induces the commutative diagram
\begin{equation*}
\begin{tikzcd}
H^1_S(\bigrep{j}{-i}) \arrow[r]\arrow[d] & H^1_S(\bigrep{j-1}{-i}) \arrow[r]\arrow[d] & H^2_S(\F_p(j-i))\arrow[d] \\
H^1(G_{\Q_N}, \Sym^jV) \arrow[r] & H^1(G_{\Q_N}, \Sym^{j-1}V) \arrow[r]& H^2(G_{\Q_N}, \F_p).
\end{tikzcd}
\end{equation*}

We are concerned with the image of the first map in the top row, which is the kernel of the second map in that row. This boundary map is given by taking the cup product with the class $\tilde{\mathbf{b}}$ in
\begin{equation*}
H^1_S(\F_p(j-i) \otimes (\bigrep{j-1}{-i})^\vee)
\end{equation*}
that realizes $\bigrep{j}{-i}$ as an extension of $\bigrep{j-1}{-i}$ by $\F_p(j-i)$. Locally at $N$, all of the present modules are self-dual by Lemma \ref{big_rep_dual} and thus we might as well think of $\Sym^jV$ as an extension of $\F_p$ by $\Sym^{j-1}V$. The corresponding class in $H^1(G_{\Q_N}, \Sym^{j-1}V)$ giving this extension is the column vector $\res_N(\tilde{\mathbf{b}}) = \mathbf{b} = [\frac{b^{j}}{j!}, \cdots, \frac{b^{2}}{2}, b]^T$ in the notation of Proposition \ref{basis_of_local_H1}.

That is,
\begin{equation*}
\im(H^1_S(\bigrep{j}{-i}) \to H^1_S(\bigrep{j-1}{-i}))
\end{equation*}
is equal to
\begin{equation*}
    \{a \in H^1_S(\bigrep{j-1}{-i}) ~|~ a \cup \tilde{\mathbf{b}}=0 \}.
\end{equation*}

Noting that $\tilde{\mathbf{b}}$ satisfies the $\Sigma^*$-Selmer condition and that $\res_N(\tilde{\mathbf{b}}) \neq 0$, Proposition \ref{classes_cup_iff_H1Sigmastar} then gives that the latter set is contained in $H^1_{\Sigma^*}(\bigrep{j-1}{-i})$, and that this containment is an equality if $p$ is regular or if $j - i = 0$.
\end{proof}

\begin{lemma}\label{H1Sigma_lifts_to_H1Lambda}
Let $1 \leq k \leq p-3$. Let $a \in H^1_\Sigma(\bigrep{k-1}{-k})$ and assume that $a$ has a lift to $H^1_S(\bigrep{k}{-k})$. Then $a$ has a lift to $H^1_\Lambda(\bigrep{k}{-k})$.
\end{lemma}

\begin{proof}
Write $a = [a_1, \ldots, a_k]^T$. Choose any lift of $a$ to $H^1_S(\bigrep{k}{-k})$, and write it as $[a_0, a_1, \ldots, a_k]^T$. By assumption, $a_i|_{G_{\Q_p}} = 0$ for all $1 \leq i \leq k$. We need to show that $a_0$ can be modified so that it is unramified when restricted to $K(\zeta_p)_p.$

It can in fact be chosen to be unramified over $\Q_p$. The fact that the $a_i$ for $i \geq 1$ vanish when restricted to $G_{\Q_p}$ gives that $a_0|_{G_{\Q_p}}$ is a class in $H^1(G_{\Q_p}, \F_p)$. This is a $2$-dimensional $\F_p$-vector space, spanned by an unramified class and the class corresponding to $\Q_p(\zeta_{p^2}^{(p)})$. But this class is in the image of the global classes, so by adding an appropriate multiple of this class to $a_0$ we get the desired conclusion.
\end{proof}

\begin{lemma}\label{image_of_H1Lambda_is_H1Sigma}
Let $1 \leq k \leq p-3$. Let $a$ be any class
\begin{equation*}
a = \begin{bmatrix}
a_{0} \\
a_{1} \\
\vdots \\
a_{k}
\end{bmatrix} \in H^1_\Lambda(\bigrep{k}{-k}).
\end{equation*}
Then $a_{i}|_{G_{\Q_{p}}} = 0$ for all $i \neq 0$. Furthermore, $a_0$ restricts to an unramified homomorphism $G_{\Q_p} \to \F_p$.
\end{lemma}

\begin{proof}
The proof is by strong induction on $i$, starting with $a_{k}$. Let $M$ be the Galois extension of $\Q$ defined by the kernel of the representation associated to $a$. We begin by examining the $G_{\Q,S}$-representation associated to the image of $a$ in $H^1_S(\bigrep{k-1}{-k})$:
\begin{equation*}
\begin{pmatrix}
    \cyclo^{-1} & \cyclo^{-2}b & \cyclo^{-3}\frac{b^2}{2} & \cdots & \cyclo^{-k}\frac{b^{k-1}}{(k-1)!} & a_{1} \\
    & \cyclo^{-2} & \cyclo^{-3}b & \cdots & \cyclo^{-k}\frac{b^{k-2}}{(k-2)!} & a_2 \\
    & & \ddots & & \vdots & \vdots \\
    & & & \cyclo^{1-k} & \cyclo^{-k}b & a_{k-1} \\
    & & & & \cyclo^{-k} & a_{k} \\
    & & & & & 1 \\
\end{pmatrix}
\end{equation*}
Restrict this representation to $G_{\Q_p}$. Looking at the bottom $2 \times 2$ quotient, we notice that $a_{k}|_{G_{\Q_p}}$ gives an extension $L_{k}/\Q_p(\zeta_p)$ contained in $M_p$. If it is nontrivial, its Galois group is $\F_p(-k)$ as a $\Gal(\Q_p(\zeta_p)/\Q_p)$-module. Since $a$ satisfies the Selmer condition $\Lambda$, $L_{k}K(\zeta_p)_p/K(\zeta_p)_p$ is unramified. As $-k \neq 1 \bmod{p-1}$, Lemma \ref{local_fields_lemma} then applies to conclude that $L_{k}/\Q_p(\zeta_p)$ is unramified. Equivalently, $a_{k}|_{G_{\Q_p}}$ lies in $H^1_\ur(G_{\Q_p}, \F_p(-k))$ which is trivial as $k \neq 0 \bmod{p-1}$. (Indeed, the unique unramified $\F_p$-extension of $\Q_p(\zeta_p)$ is abelian over $\Q_p$, and thus does not correspond to a class in $H^1_\ur(G_{\Q_p}, \F_p(-k))$.)

Still restricting to $G_{\Q_p}$, we now have that the bottom $3 \times 3$ quotient of the representation given by the matrix above is
\begin{equation*}
\begin{pmatrix}
\cyclo^{1-k} & \cyclo^{-k} b & a_{k-1} \\
& \cyclo^{-k} & 0 \\
& & 1
\end{pmatrix}.
\end{equation*}

Thus $a_{k-1} \in H^1(G_{\Q_p}, \F_p(1-k))$ and so defines an extension $L_{k-1}/\Q_p(\zeta_p)$. If it is non trivial, it has an action of $\Gal(\Q_p(\zeta_p)/\Q_p)$ by $\cyclo^{1-k}$. As above, the extension $L_{k-1}K(\zeta_p)_p/K(\zeta_p)_p$ is unramified, so we conclude that
\begin{equation*}
a_{k-1} \in H^1_\ur(G_{\Q_p}, \F_p(1-k)) = 0.
\end{equation*}

We can continue inductively in the same manner to show that $a_{k-i} = 0$ for $0 \leq i \leq k-1$. The two facts we need are that $\cyclo^{i-k} \neq \cyclo$ so that Lemma \ref{local_fields_lemma} applies, and that $\cyclo^{i-k}$ is nontrivial so that $H^1_\ur(G_{\Q_p}, \F_p(i)) = 0$.

To get the final claim about $a_0$, carry out one more step of the induction. Lemma \ref{local_fields_lemma} applies in this case, but the second fact above does not.
\end{proof}

\subsection{\texorpdfstring{$\text{Cl}_K$}{CL\textunderscore K} and Selmer Groups of Characters}\label{pfpart3}

Recall the filtration by type on $H^1_\Lambda(\bigrep{p-3}{2})$ considered in Remark \ref{filtration_on_Lambda}. As a corollary to Proposition \ref{Lambda_to_Sigma_is_exact}, we conclude that this filtration descends to a filtration
\begin{align*}
0 &\subseteq H^1_\Sigma(\F_p(-1)) \\*
    &\subseteq H^1_\Sigma(\bigrep{1}{-2}) \\
    &\subseteq H^1_\Sigma(\bigrep{2}{-3}) \\*
    &\subseteq\cdots \\*
    &\subseteq H^1_\Sigma(\bigrep{p-4}{2}).
\end{align*}

In the spirit of Remark \ref{filtration_on_Lambda}, we think of the $k$th piece $H^1_\Sigma(\bigrep{k-1}{-k})$ in the above filtration as corresponding to those $E/K$ of type $2 \leq m \leq k+1$, and the quotient
\begin{equation*}
\frac{H^1_\Sigma(\bigrep{k-1}{-k})}{H^1_\Sigma(\bigrep{k-2}{1-k})}
\end{equation*}
as corresponding to the extensions of type exactly $k+1$, in the sense that its dimension is the number of inequivalent extensions $E/K$ of type $k+1$, where two such extensions are equivalent if they become the same after taking the compositum with an extension of strictly smaller type.

With this in mind, we offer the following proposition.

\begin{proposition}\label{h1Sigma_filtration_bound}
The following are true:
\begin{enumerate}
\item
$h^1_\Sigma(\F_p(-1)) \leq h^1_\Sigma(\bigrep{p-4}{2})$.

\item
If there is an $E/K$ is of type $m \geq 2$, then $H^1_\Sigma(\F_p(1-m))$ is nontrivial.

\item
$h^1_\Sigma(\bigrep{p-4}{2}) \leq \sum_{i=1}^{p-3} h^1_\Sigma(\F_p(-i))$.
\end{enumerate}
\end{proposition}

\begin{proof}
The first part of the proposition follows from the fact that the smallest piece in the above filtration is
\begin{equation*}
H^1_\Sigma(\F_p(-1)) \subseteq H^1_\Sigma(\bigrep{p-4}{2}).
\end{equation*}

Now, take the exact sequence
\begin{equation*}
0 \to \bigrep{k-2}{1-k} \to \bigrep{k-1}{-k} \to \F_p(-k) \to 0
\end{equation*}
and look at the $\Sigma$-Selmer subgroups of the long exact sequence in $G_{\Q,S}$-cohomology to get the exact sequence
\begin{equation*}
0 \to H^1_\Sigma(\bigrep{k-2}{1-k}) \to H^1_\Sigma(\bigrep{k-1}{-k}) \to H^1_\Sigma(\F_p(-k)).
\end{equation*}

Thus
\begin{equation*}
\frac{H^1_\Sigma(\bigrep{k-1}{-k})}{H^1_\Sigma(\bigrep{k-2}{1-k})} \subseteq H^1_\Sigma(\F_p(-k))
\end{equation*}

which establishes the second part of the proposition: if there is an $E/K$ is of type $m$ then $H^1_\Sigma(\F_p(1-m))$ is nonzero, and furthermore that the size of this group is related to the number of inequivalent extensions of type $m$ as discussed above.

The associated graded space of $H^1_\Sigma(\bigrep{p-4}{2})$ equipped with this filtration is
\begin{align*}
\text{gr}(H^1_\Sigma(\bigrep{p-4}{2})) &= \bigoplus_{k-1}^{p-3}\frac{H^1_\Sigma(\bigrep{k-1}{-k})}{H^1_\Sigma(\bigrep{k-2}{1-k})}\\*
&\subseteq \bigoplus_{k=1}^{p-3} H^1_\Sigma(\F_p(-k))
\end{align*}
which proves the final part of the proposition.
\end{proof}

\section{Lifting Selmer Classes}\label{sec_lifting_selmer_classes}

One might ask if the inequality of Theorem \ref{rank_equals_h1Sigma} is ever an equality:
\begin{equation*}
r_K \overset{?}{=} 1 + \sum_{i=1}^{p-3} h^1_\Sigma(\F_p(-i)).
\end{equation*}

In Section \ref{5_subsection}, we show that this is true when $p=5$. However, it is not true in general. In particular, see Section \ref{7_subsection} for a detailed analysis of the possible cases when $p=7$.

We have seen in Section \ref{sec_big_general} that given an unramified $\F_{p}$-extension $E/K$ of type $m > 1$, we get a $G_{\Q, S}$-representation of dimension $m+1$ whose image is isomorphic to the Galois group $\Gal(M/\Q)$ where $M$ is the Galois closure of $E/\Q$.
This gives a class in $H^{1}_{\Sigma}(\bigrep{m-2}{1-m})$ whose image in the quotient $H^{1}_{\Sigma}(\F_{p}(1-m))$ is nonzero.

This section will tackle the converse to this construction, namely by providing criteria for when a nonzero class $a_{i}$ in $H^{1}_{\Sigma}(\F_{p}(-i))$ may be lifted to an element in $H^{1}_{\Sigma}(\bigrep{i-1}{-i})$, as such a lift gives a representation of the form (\ref{big_matrix}) by Proposition \ref{Lambda_to_Sigma_is_exact}, which corresponds to an extension $E/K$ of type $i+1$. We consider two separate methods, one in each of Sections \ref{sec_climbing_the_ladder} and \ref{sec_matrix_exchange}.

It is worth remarking that in the case $i=1$, there are no obstructions to worry about: The class $a_1 \in H^1_\Sigma(\F_p(-1))$ lifts directly to a class in $H^1_\Lambda(\bigrep{1}{-1})$ which gives an extension $E/K$ of type $2$. (This is the $k=1$ case of Proposition \ref{Lambda_to_Sigma_is_exact}, or equivalently the first part of Proposition \ref{h1Sigma_filtration_bound}.) This is the method that Wake--Wang-Erickson use to prove the lower bound in Theorem \ref{rank_equals_h1Sigma}, which they state as the following proposition.

\begin{proposition}[\cite{wake_wang-erickson_mazur_eisenstein_ideal}, Proposition 11.1.1]\label{wake_wang_erickson_theorem}
If $h^1_\Sigma(\F_p(-1)) \neq 0$ then $r_K \geq 2$.
\end{proposition}

\begin{remark}
The question of lifting representations is related to the vanishing of higher Massey products $\langle b, \ldots, b, a_i \rangle$ in $G_{\Q,S}$-cohomology. In \cite{sharifi_massey_products}, Sharifi has shown that certain higher Massey products of this type vanish in $G_\Q$-cohomology.

For example, one way of interpreting the results of Section \ref{sec_matrix_exchange} is in terms of the vanishing of certain triple Massey products in $G_{\Q,S}$-cohomology.
Theorem \ref{general_middle_lift} could be restated as saying that the triple $G_{\Q, S}$-Massey product $\langle b, b, a \rangle$ vanishes, where $a$ is a class that spans $H^{1}_{\Sigma}(\F_{p}(\frac{p-1}{2}))$.
\end{remark}

\subsection{Climbing the Ladder}\label{sec_climbing_the_ladder}

We approach the problem of lifting the classes in $H^{1}_{\Sigma}(\F_p(-i))$ one dimension at a time.
Namely, we will give a sequence of lemmas which provide criteria for when a class in $H^{1}_{\Sigma}(\bigrep{j-1}{-i})$ may be lifted ``one rung up the ladder'' to a class in $H^{1}_{\Sigma}(\bigrep{j}{-i})$, for $1 \leq 1 \leq p-2$ and $1 \leq j \leq i-1$. Lemma \ref{image_of_H1S_is_H1Sigmastar} shows that one obstruction to this lifting is the irregularity of $p$. Therefore we assume for simplicity for the remainder of this section that $p$ is regular.
Given this assumption, Lemma \ref{image_of_H1S_is_H1Sigmastar} tells us that every class in $H^{1}_{\Sigma}(\bigrep{j-1}{i})$ in the range of $j$ and $i$ we consider has a lift to $H^{1}_{S}(\bigrep{j}{-i})$, so we are tasked with showing that there are lifts which satisfy the local conditions of the Selmer condition $\Sigma$.

Our strategy is as follows. Recall the short exact sequence of $G_{\Q, S}$-modules
\begin{equation*}
0 \to \F_{p}(j-i) \to \bigrep{j}{-i} \to \bigrep{j-1}{-i} \to 0
\end{equation*}
which induces the following piece of the long exact sequence in $G_{\Q, S}$-cohomology
\begin{equation*}
0 \to H^{1}_{S}(\F_{p}(j-i)) \to H^{1}_{S}(\bigrep{j}{-i}) \to H^{1}_{S}(\bigrep{j-1}{-i})
\end{equation*}
as $H^{0}(G_{\Q, S}, \bigrep{j-1}{-i}) = 0$ for the $i$ and $j$ considered.

Thus, if $a$ is a fixed class in $H^1_{\Sigma^*}(\bigrep{j-1}{-i})$ which has a lift $\tilde{a}$ to $H^1_S(\bigrep{j}{-i})$, we may modify $\tilde{a}$ by adding classes in $H^{1}_{S}(\F_{p}(j-i))$ in an attempt to produce others lifts of $a$ which satisfy the local conditions of $\Sigma$. The following lemmas give conditions for when such modification is possible.

\begin{lemma}\label{fix_at_N}
Let $p$ be regular. Suppose that $a \in H^{1}_{\Sigma^{\ast}}(\bigrep{j-1}{-i})$, and that
\begin{equation*}
h^{1}_{\Sigma^{\ast}}(\F_{p}(j-i)) < h^{1}_{S}(\F_{p}(j-i))
\end{equation*}
Then there is a lift of $a$ to $H^{1}_{\Sigma^{\ast}}(\bigrep{j}{-i})$.
\end{lemma}
\begin{proof}
The proof is essentially a diagram chase. Consider the following commutative diagram. For space considerations, we abbreviate $\bigrep{a}{b}$ as $V^a(b)$.

\begin{equation*}
\begin{tikzcd}[column sep = tiny]
& 0 \arrow[d] & 0 \arrow[d] & & \\
& H^1_{\Sigma*}(\F_p(j-i)) \arrow[d] & H^1_{\Sigma^*}(V^j(-i)) \arrow[d] \arrow[dr, dashed] & & \\
0 \arrow[r] & H^1_S(\F_p(j-i)) \arrow[r] \arrow[d] \arrow [dr, dashed] & H^1_S(V^j(-i)) \arrow[r] \arrow[d] & H^1_{\Sigma^*}(V^{j-1}(-i)) \arrow[r] & 0 \\
& H^1(G_{\Q_N}, \F_p(j-i)) / \langle b \rangle \arrow[r, "\sim"] & H^1(G_{\Q_N}, V^j(-i))/\langle \mathbf{b} \rangle & &
\end{tikzcd}
\end{equation*}


The middle row is exact by Lemma \ref{image_of_H1S_is_H1Sigmastar}, and Proposition \ref{basis_of_local_H1} gives that the bottom arrow is an isomorphism and that the two groups are both $1$-dimensional, as well as the fact that the two columns are exact.

The lemma is equivalent to the statement that the top-right diagonal arrow
\begin{equation*}
    H^1_{\Sigma^*}(\bigrep{j}{-i}) \to H^1_{\Sigma^*}(\bigrep{j-1}{-i})
\end{equation*}
is surjective. We first claim that this is implied by the surjectivity of the bottom-left diagonal arrow
\begin{equation*}
    H^1_S(\F_p(j-i)) \to 
    H^1(G_{\Q_N}, \bigrep{j}{-i}) / \langle \mathbf{b} \rangle.
\end{equation*}

Indeed, suppose that diagonal map is surjective and let $\tilde{a}$ be any lift of $a$ to $H^1_S(\bigrep{j}{-i})$. Let $c$ be any class in $H^1_S(\F_p(j-i))$ whose image in $H^1(G_{\Q_N}, \bigrep{j}{-i})/\langle \mathbf{b} \rangle$ is equal to the image of $\tilde{a}$. Then $\tilde{a}-c$ is a lift of $a$ that lies in $H^1_{\Sigma^*}(\bigrep{j}{-i})$.

We are reduced to showing that the bottom-left diagonal arrow is surjective. Because the bottom horizontal arrow is an isomorphism, it suffices to show that the vertical map 
\begin{equation*}
    H^1_S(\F_p(j-i)) \to H^1(G_{\Q_N}, \F_p)/\langle b \rangle
\end{equation*}
is surjective. As the latter group is $1$-dimensional, we just need to show that this map is nonzero, which follows from the assumption
\begin{equation*}
h^{1}_{\Sigma^{\ast}}(\F_{p}(j-i)) < h^{1}_{S}(\F_{p}(j-i)). \qedhere
\end{equation*}
\end{proof}

\begin{lemma}\label{fix_at_p}
Let $p$ be regular. Suppose that $a \in H^{1}_{\Sigma}(\bigrep{j-1}{-i})$ where $j-i \neq 0,1 \bmod{p-1}$, and that
\begin{equation*}
h^{1}_{N}(\F_{p}(j-i)) < h^{1}_{S}(\F_{p}(j-i)).
\end{equation*}
Then there is a lift of $a$ to $H^{1}_{S}(\bigrep{j}{-i})$ which is trivial when restricted to $G_{\Q_p}$.
\end{lemma}
\begin{proof}
The argument is similar to the previous one. Let $H$ be the preimage of $H^1_\Sigma(\bigrep{j-1}{-i})$ under the map
\begin{equation*}
    H^1_S(\bigrep{j}{-i}) \to H^1_{\Sigma^*}(\bigrep{j-1}{-i}).
\end{equation*}

We will reference the following diagram, where the local condition ``split at $p$'' is abbreviated ``spl $p$''. Because $j-i \neq 0 \bmod {p-1}$, we have $H^1_{\ur}(G_{\Q_p}, \F_p(j-i)) = 0$, and thus $H^1_N(\F_p(j-i)) = H^1_{\text{spl }p}(\F_p(j-i))$. As above, we abbreviate $\bigrep{a}{b}$ as $V^a(b)$.

\begin{equation*}
\begin{tikzcd}[column sep = small]
& 0 \arrow[d] & 0 \arrow[d] & & \\
& H^1_{N}(\F_p(j-i)) \arrow[d] & H^1_{\text{spl } p}(V^j(-i)) \arrow[d] \arrow[dr, dashed] & & \\
0 \arrow[r] & H^1_S(\F_p(j-i)) \arrow[r] \arrow[d] \arrow[dr, dashed] & H \arrow[r] \arrow[d] & H^1_{\Sigma}(V^{j-1}(-i)) \arrow[r] \arrow[d, "0"] & 0 \\
& H^1(G_{\Q_p}, \F_p(j-i)) \arrow[r] & H^1(G_{\Q_p}, V^j(-i)) \arrow[r, "\phi"] & H^1(G_{\Q_p}, V^{j-1}(-i)) &
\end{tikzcd}
\end{equation*}


We first remark that this diagram makes sense: Each of $H^1_{\text{spl } p}(\bigrep{j}{-i})$ and $H^1_S(\F_p(j-i))$ lands in $H$ under its respective map to $H^1_S(\bigrep{j}{-i})$. Note that the middle row is exact by Lemma \ref{image_of_H1S_is_H1Sigmastar} and the two columns are exact by definition. Similarly, the vertical map in the final column is $0$.

As in the proof of Lemma \ref{fix_at_N}, we want to show that the top-right diagonal map is surjective. Note that the image of the bottom-left diagonal arrow is contained in the kernel of $\phi$. We first argue that if this map surjects onto $\ker \phi$, then the top-right diagonal map is surjective as well.

Indeed, suppose that the diagonal map
\begin{equation*}
    H^1_S(\F_p(j-i)) \to \ker \phi
\end{equation*}
is surjective and let $a \in H^1_\Sigma(\bigrep{j-1}{-i})$. Choose any lift $\tilde{a}$ of $a$ to $H$ and let $\overline{a}$ be the image of $\tilde{a}$ in $H^1(G_{\Q_p}, \bigrep{j}{-i})$. Since the right-hand vertical map is $0$, we know that $\overline{a} \in \ker \phi$. If $c \in H^1_S(\F_p(j-i))$ is a class whose image in $\ker \phi$ under the diagonal map is $\overline{a}$, then $\tilde{a}-c$ is a lift of $a$ that lies in $H^1_{\text{spl } p}(\bigrep{j}{-i})$.

Now, because 
\begin{equation*}
    \ker \phi = \im(H^1(G_{\Q_p}, \F_p(j-i)) \to H^1(G_{\Q_p}, \bigrep{j}{-i})),
\end{equation*}
we are reduced to showing that the vertical map
\begin{equation*}
    H^1_S(\F_p(j-i)) \to H^1(G_{\Q_p}, \F_p(j-i))
\end{equation*}
is surjective.

Since $j-i \neq 0,1 \bmod{p-1}$, we have that the latter group is $1$-dimensional by the Local Euler Characteristic Formula, so the surjectivity of this map is equivalent to the map being nonzero. As the kernel of this map is $H^1_N(\F_p(j-i))$, this final statement follows from the assumption
\begin{equation*}
    h^{1}_{N}(\F_{p}(j-i)) < h^{1}_{S}(\F_{p}(j-i)). \qedhere
\end{equation*}
\end{proof}

\begin{lemma}\label{fix_at_p_no_change_at_N}
Let $p$ be regular. Suppose that $a \in H^{1}_{\Sigma}(\bigrep{j-1}{-i})$ where $j-i \neq 0,1 \bmod{p-1}$, that there is a lift of $a$ to $\in H^{1}_{\Sigma^{\ast}}(\bigrep{j}{-i})$, and that
\begin{equation*}
h^{1}_{\Sigma}(\F_{p}(j-i)) < h^{1}_{\Sigma^{\ast}}(\F_{p}(j-i)).
\end{equation*}
Then there is a lift of $a$ to $\in H^{1}_{\Sigma}(\bigrep{j}{-i})$.
\end{lemma}
\begin{proof}
The argument is nearly identical to the one above. Let $H'$ be the preimage of $H^1_\Sigma(\bigrep{j-1}{-i})$ under the map
\begin{equation*}
    H^1_{\Sigma^*}(\bigrep{j}{-i}) \to H^1_{\Sigma^*}(\bigrep{j-1}{-i}).
\end{equation*}
Now, repeat the argument given in Lemma \ref{fix_at_p} in reference to the diagram

\begin{equation*}
\begin{tikzcd}[column sep = small]
& 0 \arrow[d] & 0 \arrow[d] & & \\
& H^1_{\Sigma}(\F_p(j-i)) \arrow[d] & H^1_{\Sigma}(V^j(-i)) \arrow[d] \arrow[dr, dashed] & & \\
0 \arrow[r] & H^1_{\Sigma^*}(\F_p(j-i)) \arrow[r] \arrow[d] \arrow[dr, dashed] & H' \arrow[r] \arrow[d] & H^1_{\Sigma}(V^{j-1}(-i)) \arrow[r] \arrow[d, "0"] & 0 \\
& H^1(G_{\Q_p}, \F_p(j-i)) \arrow[r] & H^1(G_{\Q_p}, V^j(-i)) \arrow[r] & H^1(G_{\Q_p}, V^{j-1}(-i)) &
\end{tikzcd}
\end{equation*}


where, as before, $\bigrep{a}{b}$ is abbreviated to $V^a(b)$.
\end{proof}

The final lemma of this section is just a concatenation of Lemmas \ref{fix_at_N} and \ref{fix_at_p_no_change_at_N}. 
We state it as its own lemma for easier reference later.

\begin{lemma}\label{fix_at_both}
Let $p$ be regular. Suppose that $a \in H^{1}_{\Sigma}(\bigrep{j-1}{-i})$, and that
\begin{equation*}
h^{1}_{\Sigma}(\F_{p}(j-i)) < h^{1}_{\Sigma^{\ast}}(\F_{p}(j-i)) < h^{1}_{S}(\F_{p}(j-i)).
\end{equation*}
Then there is a lift of $a$ to $H^{1}_{\Sigma}(\bigrep{j}{-i})$.
\end{lemma}

As $p$ is regular, the condition in Lemma \ref{fix_at_both} can only occur when $j-i$ is odd, $h^1_{\Sigma}(\F_p(j-i)) = 0$, and $h^1_{\Sigma^*}(\F_p(j-i)) = 1$; see Theorem \ref{cohomology_theorem_conditional}.

\begin{remark}
The assumption that $p$ is regular in each of the lemmas in this section could be replaced with the more general assumption that there exists a lift of $a \in H^1_\Sigma(\bigrep{j-1}{-i})$ to $H^1_S(\bigrep{j}{-i})$. We will not need that generality.
\end{remark}

\subsection{Lifting Classes in \texorpdfstring{$H^{1}_{\Sigma}(\mathbf{F}_{p}(\frac{p-1}{2}))$}{H1\textunderscore Sigma(F\textunderscore p((p-1)/2))}}\label{sec_matrix_exchange}

In this section we will prove that in a special case, some classes in a $\Sigma$-Selmer group of a character always lift to the $\Sigma^{\ast}$-Selmer group of a 2-dimensional representation.
In particular we will be able to apply this result in situations where it is not possible to use Lemma \ref{fix_at_N} to show that a class lifts into a $\Sigma^{\ast}$-Selmer group.
Our standing assumptions for this section will be that $p$ is regular and $H^{1}_{\Sigma}(\F_{p}(\frac{p-1}{2})) \neq 0$.
In addition to ensuring that classes in $H^{1}_{\Sigma}(\F_{p}(\frac{p-1}{2}))$ always lift to $H^{1}_{S}(V(\frac{p-1}{2}))$ by Lemma \ref{image_of_H1S_is_H1Sigmastar}, the regularity assumption provides access to the full strength of the results of Sections \ref{sec_cohomology_groups_of_characters} and \ref{sec_cup_products}.
Note that the character $\cyclo^{\frac{p-1}{2}}$ is its own inverse.

\begin{theorem}\label{general_middle_lift}
Assume that $p$ is regular, and that $H^{1}_{\Sigma}(\F_{p}(\frac{p-1}{2})) \neq 0$.
If a class $a \in H^{1}_{\Sigma}(\F_{p}(\frac{p-1}{2}))$ is nonzero, and if $\begin{bmatrix} a' \\ a \\ \end{bmatrix}$ is any lift of $a$ to $H^{1}_{S}(V(\frac{p-1}{2}))$, then $\begin{bmatrix} a' \\ a \\ \end{bmatrix} \in H^{1}_{\Sigma^{\ast}}(V(\frac{p-1}{2}))$.
\end{theorem}

The idea behind the proof of this theorem is to work with a related representation $W$ which allows us to exploit the self-inverse property of $\cyclo^{\frac{p-1}{2}}$ to determine the $\Sigma$-Selmer group of a twist of $W$ explicitly.
Taken together with Theorem \ref{greenberg_wiles}, we will be able to use this explicit determination of a Selmer group to get positive information about the local properties of the class $\begin{bmatrix} a' \\ a \\ \end{bmatrix}$ (namely, that it is always in the $\Sigma^{\ast}$-Selmer group).
We define the representation $W$ that will be used, and then prove the theorem over the course of several lemmas.

We let $W$ be the $2$-dimensional $\F_{p}$-vector space on which $G_{\Q, S}$ acts by
\begin{equation*}
\begin{pmatrix}
\cyclo^{\frac{p-1}{2}} & a \\ 0 & 1 \\
\end{pmatrix}
\end{equation*}
where $a$ is a nonzero class in $H^{1}_{\Sigma}(\F_{p}(\frac{p-1}{2}))$.
By Proposition \ref{classes_cup_iff_H1Sigmastar} we know that $b \cup a = 0$, hence $a$ lifts to a class $\begin{bmatrix} a' \\ a \\ \end{bmatrix} \in H^{1}_{S}(V(\frac{p-1}{2}))$.
In other words there is a $3$-dimensional representation of $G_{\Q, S}$ (which is an extension of $\F_p$ by $V(\frac{p-1}{2})$) defined by
\begin{equation*}\label{3_dimn_rep}
\begin{pmatrix}\tag{$\dagger$} \cyclo^{\frac{p+1}{2}} & \cyclo^{\frac{p-1}{2}} b & a' \\ 0 & \cyclo^{\frac{p-1}{2}} & a \\ 0 & 0 & 1 \\ \end{pmatrix}.
\end{equation*}
Note that the representation (\ref{3_dimn_rep}) is also an extension of $W$ by $\cyclo^{\frac{p+1}{2}}$.
Taking the contragredient of the representation (\ref{3_dimn_rep}) and twisting by $\cyclo^{-\frac{p+1}{2}}$ yields another $3$-dimensional representation of $G_{\Q, S}$ defined by
\begin{equation*}\label{3_dimn_rep_flip}
\begin{pmatrix}\tag{$\ddagger$} \cyclo^{\frac{p+1}{2}} & \cyclo a & ab-a' \\ 0 & \cyclo & -b \\ 0 & 0 & 1 \\ \end{pmatrix}.
\end{equation*}
Note that this representation is an extension of $\F_p$ by $W(1)$, which is to say that $\begin{bmatrix} ab - a' \\ -b \\ \end{bmatrix} \in H^{1}_{S}(W(1))$.

Both $3$-dimensional representations share the same kernel; the operation of taking contragredient and twisting by $\cyclo^{-\frac{p+1}{2}}$ doesn't change the kernel.
Let $L/\Q$ be the fixed field of this kernel.
We have a diagram of fields
\begin{equation*}
\begin{tikzcd}[column sep = tiny, row sep = small]
& L \arrow[dl, dash] \arrow[dr, dash] & \\
K(\zeta_{p}) \arrow[dr, dash] & & L_{a} \arrow[dl, dash] \\
& \Q(\zeta_{p}) \arrow[d, dash] & \\
& \Q &
\end{tikzcd}
\end{equation*}
where $L_{a}$ is the fixed field of the kernel of the representation $W$, which is Galois over $\Q$ with Galois group isomorphic to the semi-direct product $\Z/p\Z \rtimes (\Z/p\Z)^{\times}$ where $(\Z/p\Z)^{\times}$ acts by $\cyclo^{\frac{p-1}{2}}$.
(This is the group $\Gamma_{\frac{p-1}{2}}$ in the notation of Theorem \ref{classification_of_reps}.)
The representation (\ref{3_dimn_rep}) is a realization of $\Gal(K(\zeta_{p})/\Q)$ acting on $\Gal(L/K(\zeta_{p}))$, whereas the representation (\ref{3_dimn_rep_flip}) is a realization of $\Gal(L_{a}/\Q)$ acting on $\Gal(L/L_{a})$.

This commonality between the representations (\ref{3_dimn_rep}) and (\ref{3_dimn_rep_flip}) and their associated cohomology classes $\begin{bmatrix} a' \\ a \\ \end{bmatrix}$ and $\begin{bmatrix} ab - a' \\ -b \\ \end{bmatrix}$ allows us to relate the local behavior of these classes.

\begin{lemma}\label{matrix_flip_equivalence}
The class $\begin{bmatrix} a' \\ a \\ \end{bmatrix}$ is in $H^{1}_{\Sigma^{\ast}}(V(\frac{p-1}{2}))$ if and only if the class $\begin{bmatrix} ab - a' \\ -b \\ \end{bmatrix}$ is in $H^{1}_{\Sigma^{\ast}}(W(1))$.
\end{lemma}
\begin{proof}
Noting that $a$ is necessarily a nonzero multiple of $b$ locally at $N$ by Remark \ref{H1Sigma_is_b_at_N}, we see that $W(1)$ and $V(\frac{p-1}{2})$ are isomorphic representations locally at $N$.
In particular the results of Section \ref{sec_sigma} still apply to the twists of $W$.

In the case of both $V(\frac{p-1}{2})$ and $W(1)$, the $\Sigma^{\ast}$ condition is just that classes vanish when restricted to $K_{N}$.
Interpreting this in terms of the Galois extension $L/\Q$ cut out by both classes, we see that either class satisfies the $\Sigma^{\ast}$ condition if and only if $N$ is split in $L/L_{a}K(\zeta_{p})$, as we know that locally at $N$ the extension $L_{a}K(\zeta_{p})$ is $K_{N}$.
\end{proof}

We will use the fact that $\cyclo^{\frac{p-1}{2}}$ is self-inverse to show that we have an equality $H^{1}_{\Sigma^{\ast}}(W(1)) = H^{1}_{S}(W(1))$, hence the equivalent statements of the previous lemma will always hold.
Since we will be applying Theorem \ref{greenberg_wiles} to compute $h^{1}_{\Sigma^{\ast}}(W(1))$, we will need the fact that
\begin{equation*}
W(1)^{\ast} \cong W(\textstyle{\frac{p-1}{2}}).
\end{equation*}
We start by determining the dimensions of $H^{1}_{S}(W(1))$ and $H^{1}_{S}(W(\frac{p-1}{2}))$.
As this result will depend on whether $p \equiv 1$ or $3 \bmod{4}$ we will use the notation
\begin{equation*}
s_{p} = \begin{cases} 1 & p \equiv 1 \bmod{4} \\ 0 & p \equiv 3 \bmod{4}. \\ \end{cases}
\end{equation*}

\begin{lemma}\label{H1S_flip_dimensions}
The classes generating $H^{1}_{S}(W(1))$ and $H^{1}_{S}(W(\frac{p-1}{2}))$ are as follows.
\begin{enumerate}
\item\label{H1S_flip_dimensions_1} We have that $2 + s_{p} \leq h^{1}_{S}(W(1)) \leq 3 + s_{p}$.
The classes
\begin{equation*}
\begin{bmatrix} x \\ 0 \\ \end{bmatrix}, \begin{bmatrix} \ast \\ b \\ \end{bmatrix}
\end{equation*}
for $x \in H^{1}_{S}(\F_{p}(\frac{p+1}{2}))$ always span a $(2 + s_{p})$-dimensional subspace.
Let $b'$ be the class of $p$ in $H^{1}_{S}(\F_{p}(1))$.
The dimension $h^{1}_{S}(W(1))$ is equal to $3 + s_{p}$ if and only if $p$ is a $p$th power modulo $N$, in which case the final dimension is spanned by some lift of $b'$,
\begin{equation*}
\begin{bmatrix} \ast \\ b' \\ \end{bmatrix}.
\end{equation*}
\item\label{H1S_flip_dimensions_2} We have that $3 \leq h^{1}_{S}(W(\frac{p-1}{2})) \leq 4$.
The classes
\begin{equation*}
\begin{bmatrix} y \\ 0 \\ \end{bmatrix}, \begin{bmatrix} a^{2}/2 \\ a \\ \end{bmatrix}
\end{equation*}
for $y \in H^{1}_{S}(\F_{p})$ span a $3$-dimensional subspace, and $h^{1}_{S}(W(\frac{p-1}{2})) = 4$ if and only if $p \equiv 3 \bmod{4}$ and $H^{1}_{\Sigma^{\ast}}(\F_{p}(\frac{p-1}{2})) = H^{1}_{S}(\F_{p}(\frac{p-1}{2}))$.
In this case, if $z$ is a class spanning $H^{1}_{p}(\F_{p}(\frac{p-1}{2}))$ then the final dimension is spanned by some lift of $z$,
\begin{equation*}
\begin{bmatrix} \ast \\ z \\ \end{bmatrix}.
\end{equation*}

\end{enumerate}
\end{lemma}
\begin{proof}
For the first part of this lemma, consider the following piece of the long exact sequence in $G_{\Q, S}$-cohomology:
\begin{equation*}
0 = H^{0}_{S}(\F_{p}(1)) \to H^{1}_{S}(\F_{p}(\textstyle{\frac{p+1}{2}})) \to H^{1}_{S}(W(1)) \to H^{1}_{S}(\F_{p}(1)) \overset{a \cup -}{\longrightarrow} H^{2}_{S}(\F_{p}(\textstyle{\frac{p+1}{2}})).
\end{equation*}
The $1 + s_{p}$ dimensions of $H^{1}_{S}(\F_{p}(\frac{p+1}{2}))$ give classes in $H^{1}_{S}(W(1))$ immediately.
The classes $b, b'$, which span $H^{1}_{S}(\F_{p}(1))$ lift to $H^{1}_{S}(W(1))$ if and only if their cup product with $a$ vanishes.

For the class $b$, we know that $a \cup b = 0$ by Proposition \ref{classes_cup_iff_H1Sigmastar}, as $a \in H^{1}_{\Sigma}(\F_{p}(\textstyle{\frac{p-1}{2}}))$.
Since $a$ is a nonzero multiple of $b$ when viewed as a class for $G_{\Q_{N}}$, we have that $a \cup b' = 0$ if and only if $b'$ is a multiple of $b$ locally at $N$, again by Proposition \ref{classes_cup_iff_H1Sigmastar}.
As $b'$ is unramified at $N$, the only way for it to be a multiple of $b$ locally at $N$ is if $N$ is split in the extension defined by $b'$, which is $\Q(\zeta_{p}, p^{1/p})$.
$N$ splits in this extension if and only if $p$ is a $p$th power in $\Q_{N}^{\times}$, which happens if and only if $p$ is a $p$th power in $\F_{N}^{\times}$.
Thus the class $b'$ lifts to $H^{1}_{S}(W(1))$ if and only if $p$ is a $p$th power modulo $N$.

The proof for the second part of the lemma is similar, using the long exact sequence for $W(\frac{p-1}{2})$:
\begin{equation*}
0 = H^{0}_{S}(\F_{p}(\textstyle{\frac{p-1}{2}})) \to H^{1}_{S}(\F_{p}) \to H^{1}_{S}(W(\textstyle{\frac{p-1}{2}})) \to H^{1}_{S}(\F_{p}(\textstyle{\frac{p-1}{2}})) \overset{a \cup -}{\longrightarrow} H^{2}_{S}(\F_{p}).
\end{equation*}
The $2$ dimensions of $H^{1}_{S}(\F_{p})$ give classes in $H^{1}_{S}(W(\frac{p-1}{2}))$ immediately.
The class $a$ always lifts to  $H^{1}_{S}(W(\frac{p-1}{2}))$, as we certainly have $a \cup a = 0$ as $a \in H^{1}_{\Sigma}(\F_{p}(\frac{p-1}{2}))$.
If $p \equiv 1 \bmod{4}$, $a$ spans $H^{1}_{S}(\F_{p}(\frac{p-1}{2}))$ and we conclude that $h^{1}_{S}(W(\frac{p-1}{2})) = 3$.
If $p \equiv 3 \bmod{4}$, let $z$ be a class spanning $H^{1}_{p}(\F_{p}(\frac{p-1}{2}))$ (so $a, z$ together necessarily span $H^{1}_{S}(\F_{p}(\frac{p-1}{2}))$ which is $2$-dimensional, see part \ref{cohomology_theorem_odd_regular} of Theorem \ref{cohomology_theorem_conditional}).
We have by Proposition \ref{classes_cup_iff_H1Sigmastar} that $a \cup z = 0$ if and only if $z \in H^{1}_{\Sigma^{\ast}}(\F_{p}(\frac{p-1}{2}))$, hence we conclude that $z$ lifts to $H^{1}_{S}(W(\frac{p-1}{2}))$ if and only if $H^{1}_{\Sigma^{\ast}}(\F_{p}(\frac{p-1}{2})) = H^{1}_{S}(\F_{p}(\frac{p-1}{2}))$.
\end{proof}

\begin{lemma}\label{H1Sigmastar_equals_H1S}
$H^{1}_{\Sigma^{\ast}}(W(1)) = H^{1}_{S}(W(1))$.
\end{lemma}
\begin{proof}
Applying Theorem \ref{greenberg_wiles} to $H^{1}_{\Sigma^{\ast}}(W(1))$ produces the relation:
\begin{equation*}
h^{1}_{\Sigma^{\ast}}(W(1)) = 1 + s_{p} + h^{1}_{\Sigma}(W(\textstyle{\frac{p-1}{2}})),
\end{equation*}
where we have used that $W \cong V$ as $G_{\Q_{N}}$-representations so Proposition \ref{condition_at_N_is_self_dual} applies to the twists of $W$.
We determine $h^{1}_{\Sigma}(W(\frac{p-1}{2}))$ explicitly based on our knowledge of the classes spanning it.
Let $c, c'$ be the classes spanning $H^{1}_{S}(\F_{p})$, which correspond respectively to $\Q(\zeta_{N}^{(p)})$ and $\Q(\zeta_{p^{2}}^{(p)})$.
\begin{itemize}
\item the class $\begin{bmatrix} a^{2}/2 \\ a \\ \end{bmatrix}$ is always in $H^{1}_{\Sigma}(W(\frac{p-1}{2}))$, since $a$ itself is in $H^{1}_{\Sigma}(\F_{p}(\frac{p-1}{2}))$.
\item the class $\begin{bmatrix} c' \\ 0 \\ \end{bmatrix}$ is never in $H^{1}_{\Sigma}(W(\frac{p-1}{2}))$ as it is ramified at $p$.
\item the class $\begin{bmatrix} c \\ 0 \\ \end{bmatrix}$ is in $H^{1}_{\Sigma^{\ast}}(W(\frac{p-1}{2}))$ by Lemma \ref{zeta_N_plus_localized}, and is in $H^{1}_{\Sigma}(W(\frac{p-1}{2}))$ if and only if $p$ is split in $\Q(\zeta_{N}^{(p)})$, which happens if and only if $p$ is a $p$th power mod $N$, since $\Gal(\Q(\zeta_{N}^{(p)})/\Q)$ is canonically $(\Z/N\Z)^{\times}/(\Z/N\Z)^{\times p}$.
\item if $p \equiv 3 \bmod{4}$, there is a class $z \in H^{1}_{p}(\F_{p}(\frac{p-1}{2}))$ which may or may not lift to $H^{1}_{S}(W(\frac{p-1}{2}))$; this class will never lift to $H^{1}_{\Sigma}(W(\frac{p-1}{2}))$ as it is ramified at $p$.
\end{itemize}
Putting this description together with the Lemma \ref{H1S_flip_dimensions} we have that:
\begin{align*}
\text{$p$ is a $p$th power mod $N$ }  & \implies h^{1}_{\Sigma}(W(\textstyle{\frac{p-1}{2}})) = 2 \text{ and } h^{1}_{S}(W(1)) = 3 + s_{p} \\*
& \implies h^{1}_{\Sigma^{\ast}}(W(1)) = 1 + s_{p} + 2 = 3 + s_{p} = h^{1}_{S}(W(1)) \\
\text{ $p$ is not a $p$th power mod $N$ } & \implies h^{1}_{\Sigma}(W(\textstyle{\frac{p-1}{2}})) = 1 \text{ and } h^{1}_{S}(W(1)) = 2 + s_{p}\\*
 & \implies h^{1}_{\Sigma^{\ast}}(W(1)) = 1 + s_{p} + 1 = 2 + s_{p} = h^{1}_{S}(W(1))
\end{align*}
Thus in all cases we have $h^{1}_{\Sigma^{\ast}}(W(1)) = h^{1}_{S}(W(1))$; since $H^{1}_{\Sigma^{\ast}}(W(1)) \subseteq H^{1}_{S}(W(1))$ we conclude that these groups are equal.
\end{proof}

\begin{proof}[Proof of Theorem \ref{general_middle_lift}]
By Lemma \ref{matrix_flip_equivalence}, to show that $\begin{bmatrix} a' \\ a \\ \end{bmatrix} \in H^{1}_{\Sigma^{\ast}}(V(\frac{p-1}{2}))$ it suffices to show that $\begin{bmatrix} ab - a' \\ -b \\ \end{bmatrix}$ (which a priori is just an element of $H^{1}_{S}(W(1))$) is an element of $H^{1}_{\Sigma^{\ast}}(W(1))$.
Lemma \ref{H1Sigmastar_equals_H1S} shows that $H^{1}_{\Sigma^{\ast}}(W(1)) = H^{1}_{S}(W(1))$, so this latter condition is immediate.
\end{proof}

\begin{remark}
The property that $\cyclo^{\frac{p-1}{2}}$ is self-inverse is crucial to this argument, and similar results are not true for other powers of $\cyclo$.
See Section \ref{7_subsection} for examples where the automatic lifting of classes in $H^{1}_{\Sigma}(\F_{p}(i))$ to $H^{1}_{\Sigma^{\ast}}(V(i))$ fails.
\end{remark}

\section{Effective Criteria for \texorpdfstring{$H^{1}_{\Sigma}(\mathbf{F}_{p}(-i)) \neq 0$}{H1\textunderscore Sigma(F\textunderscore p(-i)) /= 0}}\label{sec_effective_criteria}

Our goal in this section is to find an effective method for determining whether the various $H^{1}_{\Sigma}(\F_{p}(-i))$, $1 \leq i \leq p-3$ are zero or not. The cases $i$ even and $i$ odd are treated separately.
For each $i$, under a regularity assumption, we relate the question of whether or not $H^1_{\Sigma}(\F_p(-i)) = 1$ to whether or not a certain quantity in $\F_{N}^{\times}$ is a $p$th power.

\subsection{A Criterion for \texorpdfstring{$H^{1}_{\Sigma}(\mathbf{F}_{p}(-i)) \neq 0$}{H1\textunderscore Sigma(F\textunderscore p(-i)) =/= 0}, \texorpdfstring{$i$}{i} Odd}\label{sec_odd_invariant}

Let $M = \frac{N-1}{p}$, and for any positive integer $i$ define
\begin{equation*}
S_{i} = \prod_{k=1}^{p-1} ((Mk)!)^{k^{i}}.
\end{equation*}

Our goal in this section is to prove the following theorem.

\begin{theorem}\label{odd_invariant_theorem}
Let $p$ be an odd prime, $1 \leq i \leq p-4$ be odd, and assume that $(p,-i)$ is a regular pair. Then $S_{i}$ is a $p$th power in $\F_N^\times$ if and only if $H^1_\Sigma(\F_p(-i)) \neq 0$
\end{theorem}

The general strategy is as follows: Recall from part \ref{cohomology_theorem_odd_regular} of Theorem \ref{cohomology_theorem_conditional} that \begin{equation*}
h^1_\Sigma(\F_p(-i)) \leq h^1_N(\F_p(-i)) = 1.
\end{equation*}
We will show that the vanishing of $S_{i}$ in $\F_N^\times / \F_N^{\times p}$ is equivalent to the statement that the nontrivial class in $H^1_N(\F_p(-i))$ satisfies the Selmer condition $\Sigma$. To do this, we will produce an element $\mathcal{G}_{-i} \in \Q(\zeta_p)^\times$ whose local properties will control the local properties of the nontrivial class in $H^1_N(\F_p(-i))$.

\begin{remark}
The existence of such an element in a slightly different formulation is shown by Lecouturier in \cite{lecouturier}.
Lecouturier computes the image of this element in $\Q_{N}^{\times}/\Q_{N}^{\times p}$ using the Gross-Koblitz formula and $N$-adic Gamma function, and the quantity $M_{i} = \prod_{k=1}^{N-1} \prod_{a=1}^{k-1} k^{a^{i}}$ arises as the image of this element in the factor $\Z_{N}^{\times}/\Z_{N}^{\times p}$ of $\Q_{N}^{\times}/\Q_{N}^{\times p}$.
His results are not stated in terms of the Selmer groups $H^{1}_{\Sigma}(\F_{p}(-i))$; instead he relates the vanishing of $M_{i}$ directly to Iimura's filtration on the class group of $K(\zeta_{p})$ (see Remark \ref{filtration_on_Lambda}) in order to deduce bounds on the rank of the class group of $K$.

We include a proof of Theorem \ref{odd_invariant_theorem} that is better suited to our formulation using Selmer groups.
The quantities $M_{i}$ of Lecouturier play the same role as the $S_{i}$ in our statement of Theorem \ref{odd_invariant_theorem}; we show in Lemma \ref{our_invariant_is_lecouturiers} that $M_{i} = S_{i}^{-1}$ as elements of $\F_{N}^{\times}/\F_{N}^{\times p}$. 
\end{remark}

\begin{remark}\label{bernoulli_and_even_invariants}
One can compare the role of $S_{i}$ in Theorem \ref{odd_invariant_theorem} to the role of classical Bernoulli numbers in the theorems of Herbrand and Ribet on class groups of cyclotomic fields.
The question of Bernoulli numbers being divisible by $p$ is replaced by the question of whether or not the invariants $S_{i}$ are $p$th powers.
Recall that when $i$ is odd, $B_i = 0$. Similarly, the invariant $S_{i}$ for even $i$ is always a $p$th power, as the following computation in $\F_N^\times / \F_N^{\times p}$ shows.
If $i = 2j$ is even, then
\begin{align*}
S_{2j}^{2}   & = \prod_{k=1}^{p-1} ((Mk)!)^{k^{2j}} ((M(p-k))!)^{(p-k)^{2j}} \\*
             & = \prod_{k=1}^{p-1} ((Mk)! (M(p-k))!)^{k^{2j}} \\* 
             & = 1
\end{align*}
where the last step follows from the fact that $a! (N - 1 - a)! \equiv \pm 1 \in \F_{N}^{\times}$ for any $a \not \equiv 0$.
Since $p$ is odd, the fact that $S_{i}^{2}$ is a $p$th power means that $S_{i}$ itself must be a $p$th power.
\end{remark}

While Theorem \ref{odd_invariant_theorem} requires a regularity assumption, the setup does not. Until the beginning of the proof of Theorem \ref{odd_invariant_theorem}, we make no regularity assumptions.

For any prime $\mathfrak{n} | N$ of $\Q(\zeta_{p})$, define
\begin{equation*}
\iota_{\mathfrak{n}}: \Q(\zeta_{p}) \to \Q(\zeta_{p})_{\mathfrak{n}} = \Q_{N}.
\end{equation*}
Note that if $\mathfrak{n}' = [a]\mathfrak{n}$ for $a \in (\Z/p\Z)^{\times}$, then
\begin{equation*}
\iota_{\mathfrak{n}'} = \iota_{[a]\mathfrak{n}} = \iota_{\mathfrak{n}} \circ [a^{-1}].
\end{equation*}
Now, fix a prime $\mathfrak{n}|N$, and set $\iota = \iota_{\mathfrak{n}}$, and $\iota_{a} = \iota_{[a]\mathfrak{n}}$ for $a \in (\Z/p\Z)^{\times}$.

Let $c \neq 0$ be a class in $H^{1}_{N}(\F_{p}(-i))$.
This class $c$ defines an extension $L/\Q(\zeta_{p})$ which is Galois over $\Q$ with Galois group $\Gamma_{-i} = \Z/p\Z \rtimes_{\cyclo^{-i}} (\Z/p\Z)^{\times}$, and $c$ lies in $H^{1}_{\Sigma}(\F_{p}(-i))$ if and only if $L$ localized at a prime above $\mathfrak{n}$ is either trivial or isomorphic to $K_{N}$.

The extension $L/\Q(\zeta_{p})$ corresponds, by global class field theory, to a homomorphism
\begin{equation*}
\psi_c: \A_{\Q(\zeta_{p})}^{\times} \to \F_{p}
\end{equation*}
which factors through the $\cyclo^{-i}$-eigenspace of the $p$-coinvariants of the double quotient
\begin{equation*}
\Q(\zeta_{p})^{\times} \backslash \A_{\Q(\zeta_{p})}^{\times} / U
\end{equation*}
where $U$ is the subgroup
\begin{equation*}
U = \prod_{\mathfrak{n}' | N} (1 + \mathfrak{n}' \mathcal{O}_{\Q(\zeta_{p})_{\mathfrak{n}'}}) \times \prod_{\mathfrak{q} \nmid N}  \mathcal{O}_{\Q(\zeta_{p})_{\mathfrak{q}}}^{\times} \times (\Q(\zeta_{p}) \otimes \R)^{\times}.
\end{equation*}
We call this eigenspace $C_{-i}$.

Identifying $\Q(\zeta_{p})_{\mathfrak{n}}$ with $\Q_{N}$, the extension of $\Q_{N}$ given by localizing $L$ at a prime above $\mathfrak{n}$ is, by local class field theory, determined by a map $\psi_{c,N}: \Q_N^\times \to \F_p$. This map is the composition of the inclusion $j: \Q(\zeta_{p})_{\mathfrak{n}}^\times \to \A_{\Q(\zeta_p)}^\times$ and the map $\psi_c$. This is summarized in the following commutative diagram:
\begin{equation*}
\begin{tikzcd}
\Q_N^\times/\Q_N^{\times p} \arrow[d, "j"] \arrow[dr, "\psi_{c,N}"] & \\
C_{-i} = (\Q(\zeta_{p})^{\times} \backslash \A_{\Q(\zeta_{p})}^{\times} / U)^{\chi^{-i}}_{p} \arrow[r, "\psi_c"] & \F_p
\end{tikzcd}
\end{equation*}

The kernel of $\psi_{c,N}$ is the norm subgroup of the extension of $\Q_N$ coming from $L$. As $\Q_N^\times/\Q_N^{\times p}$ is $2$-dimensional, this extension is either trivial or isomorphic to $K_N$ (i.e., $c \in H^1_\Sigma(\F_p(-i))$) if and only if $N$ is in the kernel of $\psi_{c,N}$.

\begin{remark}\label{idele_dual_to_cohomology}
The above construction realizes the idele group $C_{-i}$ as the dual of the cohomology group $H^1_N(\F_p(-i))$. Indeed, by class field theory as above, every element of the cohomology group corresponds to a map $\psi_c: C_{-i} \to \F_p$, and conversely every such homomorphism gives an $\F_p$-extension of $\Q(\zeta_p)$ that is Galois over $\Q$ with Galois group $\Gamma_{-i}$ and that satisfies the local conditions to lie in $H^1_N(\F_p(-i))$.

Similarly, one identifies $\Q_N^\times/\Q_N^{\times p}$ with the dual of $H^1(G_{\Q_N}, \F_p)$. Then the map $j$ defined above is nothing more than the dual to the restriction map
\begin{equation*}
    \res_N: H^1_N(\F_p(-i)) \to H^1(G_{\Q_N}, \F_p).
\end{equation*}
\end{remark}

We turn now to the map $j$. Under certain conditions, we will prove that the kernel of $j$ is $1$-dimensional, spanned by an element $\mathcal{G}_{-i}$ that will be related to $S_i$.

\begin{lemma}\label{odd_invariant_strategy}
Let $i \not\equiv -1 \bmod{p-1}$ be odd.
Suppose that there exists an element $\mathcal{G}_{-i} \in \Z[\zeta_{p}]$ which satisfies the following properties:
\begin{itemize}
\item[(a)] $\mathcal{G}_{-i}$ lies in the $\cyclo^{-i}$-eigenspace of $\Q(\zeta_{p})^{\times} /\Q(\zeta_{p})^{\times p}$.
\item[(b)] The ideal $(\mathcal{G}_{-i})$ of $\Z[\zeta_{p}]$ is divisible only by prime ideals dividing $N$.
\end{itemize}
Then $\iota(\mathcal{G}_{-i})$ is in the kernel of $j$.
\end{lemma}

\begin{proof}
We will show that $j(\iota(\mathcal{G}_{-i})) = 0$ in the idelic quotient $C_{-i}$ by showing that $j(\iota(\mathcal{G}_{-i}))$ is equal to the diagonal embedding of the global element $\mathcal{G}_{-i}$ in the $\cyclo^{-i}$-eigenspace of the $p$-coinvariants of $\A_{\Q(\zeta_{p})}^{\times}/U$, which we denote by $C'_{-i}$.

Note that since $\mathcal{G}_{-i}$ is a unit at all primes not dividing $N$ by property (b), it will suffice to work only in the coordinates of the ideles above $N$, as the quotient by $U$ kills all units at primes not dividing $N$ and all information at the infinite places.
We index the primes above $N$ relative to our fixed choice $\mathfrak{n}|N$ and the Galois action on primes; namely the set of primes above $N$ is
\begin{equation*}
\{[a]\mathfrak{n} ~|~ a \in (\Z/p\Z)^{\times}\}.
\end{equation*}
Note that an element $a' \in (\Z/p\Z)^{\times}$ permutes the coordinates above $N$, sending the $[a]\mathfrak{n}$-coordinate to the $[a'a]\mathfrak{n}$-coordinate.
The projection operator
\begin{equation*}
P_{\cyclo^{-i}}: \left(\A_{\Q(\zeta_{p})}^{\times}/U\right)_{p} \to C'_{-i}
\end{equation*}
is given by the formula
\begin{equation*}
P_{\cyclo^{-i}} = \sum_{a \in (\Z/p\Z)^{\times}} \cyclo^{-i}(a^{-1}) [a]
\end{equation*}
where we have used additive notation for the group ring $\F_{p}[(\Z/p\Z)^{\times}]$, despite it acting on the multiplicative groups of ideles.
With this notation set up, we are trying to show that
\begin{equation*}
P_{\cyclo^{-i}}(j(\iota(\mathcal{G}_{-i}))) = P_{\cyclo^{-i}}((\iota(\mathcal{G}_{-i}), 1, \ldots, 1))
\end{equation*}

is equal the diagonal embedding of $\mathcal{G}_{-i}$:
\begin{equation*}
(\iota_{a}(\mathcal{G}_{-i}))_{a \in (\Z/p\Z)^{\times}}.
\end{equation*}
We compute directly with the formula for $P_{\cyclo^{-i}}$ that in $C'_{-i}$ we have
\begin{align*}
P_{\cyclo^{-i}}(j(\iota(\mathcal{G}_{-i})))      & = (\cyclo^{-i}(a^{-1}) \iota(\mathcal{G}_{-i}))_{a \in (\Z/p\Z)^{\times}} \\*
                                                        & = (\iota(\mathcal{G}_{-i}^{\cyclo^{-i}(a^{-1})}))_{a \in (\Z/p\Z)^{\times}}.
\end{align*}
Now, by property (a), we know that $\mathcal{G}_{-i}^{\cyclo^{-i}(a^{-1})} = [a^{-1}] \mathcal{G}_{-i}$, hence
\begin{align*}
P_{\cyclo^{-i}}(j(\iota(\mathcal{G}_{-i})))      & = (\iota([a^{-1}] \mathcal{G}_{-i}))_{a \in (\Z/p\Z)^{\times}} \\*
                                                        & = (\iota_{a}(\mathcal{G}_{-i}))_{a \in (\Z/p\Z)^{\times}}
\end{align*}
where we have used that $\iota_{a} = \iota \circ [a^{-1}]$.
\end{proof}

Now we turn our attention to constructing such a $\mathcal{G}_{-i}$ and relating it to the invariant $S_i$.

Let $A = \Q(\zeta_{p}, \zeta_{N}^{(p)})$ and let $B = \Q(\zeta_{p}, \zeta_{N})$.
For any character $\eta$
\begin{equation*}
\eta: \Gal(B/\Q(\zeta_{p})) \cong (\Z/N\Z)^{\times} \to \mu_{p}
\end{equation*}
of order $p$, define the Gauss sum
\begin{equation*}
g_{\eta} = \sum_{k=1}^{N-1} \eta(k) \zeta_{N}^{k}.
\end{equation*}
Let $\mathfrak{N}$ be the prime above $\mathfrak{n}$ in $B$ (so we have $\mathfrak{N}^{N-1} = \mathfrak{n}$).
The Gauss sums $g_{\eta}$ satisfy the following properties
\begin{itemize}
\item $g_{\eta}$ is an element of the ring of integers of $A$, and is divisible only by primes above $N$.
\item Since $\Gal(\Q(\zeta_{p})/\Q) = (\Z/p\Z)^{\times}$ acts on $\mathcal{O}_{A}$, we have that for $a \in (\Z/p\Z)^{\times}$
\begin{equation*}
[a] g_{\eta} = g_{\eta^{a}}. 
\end{equation*}
\item If $[b] \in \Gal(B/\Q(\zeta_{p})) = (\Z/N\Z)^{\times}$, then
\begin{equation*}
[b] g_{\eta} = \eta(b^{-1}) g_{\eta}.
\end{equation*}
\item $g_{\eta}^{p} \in \Q(\zeta_{p})$.
\end{itemize}

Fix the choice of $\eta$ so that the composite map
\begin{equation*}
(\Z/N\Z)^{\times} \overset{\eta}{\to} \mu_{p} \hookrightarrow (\Z[\mu_{p}]/\mathfrak{n})^{\times} = (\Z/N\Z)^{\times}
\end{equation*}
is the map $k \mapsto k^{-\frac{N-1}{p}}$, and let $\tau: (\Z/p\Z)^{\times} \to \Z \setminus \{0\}$ be a set map which satisfies that the composite
\begin{equation*}
(\Z/p\Z)^{\times} \overset{\tau}{\to} \Z \setminus \{0\} \to (\Z/p^{2}\Z)^{\times}
\end{equation*}
is the map $x \mapsto x^{p}$.
In particular, $\tau(xy) \equiv \tau(x)\tau(y) \bmod{p}^{2}$.
Define
\begin{equation*}
\mathcal{G}_{-i} = \prod_{a \in (\Z/p\Z)^{\times}} ([a]g_{\eta})^{\tau(a^{i})} = \prod_{a \in (\Z/p\Z)^{\times}} (g_{\eta^{a}})^{\tau(a^{i})}.
\end{equation*}

To establish the desired properties of the element of $\mathcal{G}_{-i}$, we will need to examine the expansion of $\iota(\mathcal{G}_{-i})$ in terms of the uniformizer of $\Q(\zeta_{p})_{\mathfrak{n}} = \Q_{N}$, and to do this we will need the expansion of a Gauss sum in terms of a uniformizer.
This latter expansion is computed in the following lemma.

\begin{lemma}\label{local_computation_gauss_sum}
Let $1 \leq r < p$, $M = (N-1)/p$, and $m = rM$.
Let
\begin{equation*}
I: B \to B_{\mathfrak{N}} = \Q_{N}(\zeta_{N})
\end{equation*}
be an embedding extending $\iota$.
Note that $\pi = 1 - \zeta_{N}$ is a uniformizer in $\Q_{N}(\zeta_{N})$.
Then we have that
\begin{equation*}
I(g_{\eta^{r}}) = (-1)^{m+1}\frac{\pi^{m}}{m!} + O(\pi^{m+1}).
\end{equation*}
\end{lemma}
\begin{proof}
By definition, we have
\begin{align*}
I(g_{\eta^{r}})     & = \sum_{k=1}^{N-1} \eta(k)^{r} (1 - \pi)^{k} \\*
                    & = \sum_{k=1}^{N-1} \eta(k)^{r} - \pi \sum_{k=1}^{N-1} \binom{k}{1} \eta(k)^{r} + \pi^{2} \sum_{k=2}^{N-1} \binom{k}{2} \eta(k)^{r} - \ldots + \pi^{N-1} \\*
                    & = \sum_{j=0}^{N-1} (-1)^{j} \pi^{j} \sum_{k=1}^{N-1} \binom{k}{j} \eta(k)^{r}
\end{align*}
where we take $\binom{k}{j} = 0$ when $k < j$.
If we expand the binomial coefficients as polynomials in $k$, each term in this last sum will be of the form
\begin{equation*}
(-1)^{j} \pi^{j} \frac{a}{j!}  \sum_{k=1}^{N-1} k^{l} \eta(k)^{r}
\end{equation*}
for some $l < j$ and integer $a$.
Note that
\begin{equation*}
\sum_{k=1}^{N-1} k^{l} \eta(k)^{r} = \begin{cases} O(\pi^{N-1}) & j \neq m \\ -1 + O(\pi^{N-1}) & j = m \\ \end{cases}
\end{equation*}
since $\mathfrak{n} = \mathfrak{N}^{N-1}$ and we have that
\begin{equation*}
\sum_{k=1}^{N-1} k^{l} \eta(k)^{r} \equiv \sum_{k=1}^{N-1} k^{l-m} \bmod{\mathfrak{n}}
\end{equation*}
using that $\eta^{r}$ is the map $k \mapsto k^{-m}$ modulo $\mathfrak{n}$.

Therefore every term in the sum for $I(g_{\eta^{r}})$ will be $O(\pi^{N-1})$ until the first term involving $\sum_{k=1}^{N-1} k^{m} \eta(k)^{r}$.
This term is
\begin{equation*}
(-1)^{m} \pi^{m} \frac{1}{m!} \sum_{k=1}^{N-1} k^{m} \eta(k)^{r}.
\end{equation*}
All other terms in the sum are $O(\pi^{m+1})$, so we conclude that
\begin{equation*}
I(g_{\eta}^{r}) = (-1)^{m+1} \frac{\pi^{m}}{m!} + O(\pi^{m+1}). \qedhere
\end{equation*}
\end{proof}

\begin{lemma}\label{properties_of_invariant}
The element $\mathcal{G}_{-i}$ is in $\Q(\zeta_{p})^{\times}$, and satisfies properties (a) and (b) of Lemma \ref{odd_invariant_strategy}.
Furthermore, as elements of $\Q_{N}^\times / \Q_{N}^{\times p}$, we have
\begin{equation*}
\iota(\mathcal{G}_{-i}) = N^{B_{1, \cyclo^{i}}}S_{i}^{-1}
\end{equation*}
where $B_{1,\cyclo^{i}}$ is the generalized Bernoulli number.
\end{lemma}

\begin{proof}
For $b \in \Gal(B/\Q(\zeta_{p})) = (\Z/N\Z)^{\times}$, we have that
\begin{align*}
[b] \mathcal{G}_{-i}      & = \prod_{a \in (\Z/p\Z)^{\times}}[b] (g_{\eta^{a}})^{\tau(a^{i})} \\*
                & = \prod_{a \in (\Z/p\Z)^{\times}}(\eta^{a}(b^{-1}) g_{\eta^{a}})^{\tau(a^{i})} \\
                & = \mathcal{G}_{-i} \prod_{a \in (\Z/p\Z)^{\times}} \eta^{a\tau(a^i)}(b^{-1}) \\
                & = \mathcal{G}_{-i} \cdot \eta^{(\sum_{a \in (\Z/p\Z)^{\times}} a\tau(a^i))}(b^{-1}) \\*
                & = \mathcal{G}_{-i}.
\end{align*}
The last equality follows from the fact that the character $\eta$ has order $p$: This lets us work mod $p$ in the exponent, so we can use that $\tau(a^i) \equiv a^i \bmod{p}$ and that $\sum_{a \in (\Z/p\Z)^{\times}} a^{i+1} \equiv 0 \bmod{p}$ when $i \not\equiv -1 \bmod{p-1}$. This establishes that $\mathcal{G}_{-i} \in \Z[\zeta_p]$. Along with the properties of the Gauss sums $g_\eta$, we conclude that $\mathcal{G}_{-i}$ is only divisible by the primes above $N$, which is to say it satisfies property (b) of Lemma \ref{odd_invariant_strategy}.

To show that $\mathcal{G}_{-i}$ satisfies property (a) of Lemma \ref{odd_invariant_strategy}, we recall that $\tau$ satisfies $\tau(c^{-i}) \equiv \cyclo^{-i}(c) \bmod p$ and verify that for $c \in (\Z/p\Z)^{\times}$,
\begin{align*}
[c]\mathcal{G}_{-i}
& = \prod_{a \in (\Z/p\Z)^{\times}}[c]([a]g_{\eta})^{\tau(a^{i})} \\*
& = \prod_{a \in (\Z/p\Z)^{\times}}([ac]g_\eta)^{\tau(a^i)} \\
& = \prod_{a' \in (\Z/p\Z)^{\times}}([a']g_\eta)^{\tau(a'^i)\tau(c^{-i})} \\
& = \mathcal{G}_{-i}^{\tau(c^{-i})}\\*
& = \mathcal{G}_{-i}^{\cyclo^{-i}(c)}
\end{align*}
where all equalities are taken to be in $\Q(\zeta_p)^\times/\Q(\zeta_p)^{\times p}$. In the third equality, we have used that $g_{\eta}^{p} \in \Q(\zeta_{p})^{\times}$, so $g_{\eta}^{p^{2}} \in \Q(\zeta_{p})^{\times p}$ which means we can work mod $p^{2}$ in the exponent. For the final equality, we recall from above that $\mathcal{G}_{-i} \in \Q(\zeta_p^\times)$ and thus we can take the exponent mod $p$.

What remains is to show that $\iota(\mathcal{G}_{-i}) = N^{B_{1, \cyclo^{i}}} S_{i}^{-1}$ in $\Q_{N}^\times / \Q_{N}^{\times p}$.

Using Lemma \ref{local_computation_gauss_sum}, we can write
\begin{align*}
\iota(\mathcal{G}_{-i})
& = \prod_{a \in (\Z/p\Z)^\times} I(g_{\eta^a})^{\tau(a^i)} \\*
& = \prod_{r=1}^{p-1} \left((-1)^{rM+1}\frac{\pi^{rM}}{(rM)!} + O(\pi^{rM+1})\right)^{\tau(r^i)} \\
& = \left(\prod_{r=1}^{p-1}\left(\frac{(-1)^{rM+1}}{(rM)!}\right)^{\tau(r^i)} + O(\pi) \right) \pi^{\sum_{r=1}^{p-1}rM\tau(r^i)} \\*
& = \left(\prod_{r=1}^{p-1}\left(\frac{(-1)^{rM+1}}{(rM)!}\right)^{\tau(r^i)}\right)(1 + O(\pi))\pi^{\sum_{r=1}^{p-1}rM\tau(r^i)}
\end{align*}
in $\Q_N(\zeta_N)^\times$. Notice that the first term in this product lies in $\Q_N^\times$ and is equal to $S_i^{-1}$ in $\Q_N^\times / \Q_N^{\times p}$.

To understand the final term, we first write
\begin{align*}
\frac{\pi^{N-1}}{N}
& = \frac{1}{N}(1 - \zeta_N)^{N-1} \\*
& = \frac{1}{N}\Norm_{\Q_N}^{\Q_N(\zeta_N)}(1 - \zeta_N)\prod_{i=1}^{N-1}\frac{1-\zeta_N}{1-\zeta_N^{i}} \\
& = \prod_{i=1}^{N-1}(1 + \zeta_N + \cdots + \zeta_N^{i-1})^{-1} \\
& \equiv \left(\prod_{i=1}^{N-1}i\right)^{-1} \bmod{\pi} \\*
& \equiv -1 \bmod{\pi}
\end{align*}
as $\Z_N[\zeta_N]/(\pi) = \F_N$. Thus $\pi^{N-1} = N(-1 + O(\pi))$ and we can use this to write
\begin{align*}
\pi^{\sum_{r=1}^{p-1}rM\tau(r^i)}
& = \pi^{(N-1)\frac{1}{p}\sum_{r=1}^{p-1}r\tau(r^i)} \\*
& = \pm N^{\frac{1}{p}\sum_{r=1}^{p-1}r\tau(r^i)}(1 + O(\pi)).
\end{align*}

Working modulo $p$ in the exponent, we can substitute $\tau(r^i)$ with $\cyclo(r^i)$. This new exponent $\frac{1}{p}\sum_{r=1}^{p-1}r\cyclo(r^i)$ is exactly the generalized Bernoulli number $B_{1, \cyclo^{i}}$.

Combining the previous calculations, we have now shown that in $\Q_N^\times / \Q_N^{\times p}$,
\begin{equation*}
    \iota(\mathcal{G}_{-i}) = S_{i}^{-1}N^{B_{1,\cyclo^{i}}}w
\end{equation*}
where $w$ is a unit in $\Z_N$ that, considered as an element of $\Z_N[\zeta_N]$, is congruent to $1$ modulo $\pi$. The isomorphism $\Z_N[\zeta_N]/(\pi) = \Z_N/(N)$ tells us that $w \equiv 1 \bmod{N}$ and is thus a $p$th power in $\Q_N^\times$. Thus
\begin{equation*}
    \iota(\mathcal{G}_{-i}) = N^{B_{1,\cyclo^{i}}} S_{i}^{-1}
\end{equation*}
in $\Q_N^\times / \Q_N^{\times p}$, as desired.
\end{proof}

We are now ready to prove Theorem \ref{odd_invariant_theorem}. Up until this point, we have not made any regularity assumptions. From now on, we assume that $(p,-i)$ is a regular pair.

\begin{proof}[Proof of Theorem \ref{odd_invariant_theorem}]
We first check that $\ker j$ is $1$-dimensional and spanned by $\iota(\mathcal{G}_{-i})$. As $(p,-i)$ is a regular pair, we know that the generalized Bernoulli number $B_{1,\cyclo^i}$ is a $p$-adic unit by Remark \ref{regular_pair_bernoulli_number}. Therefore, in $\Q_N^\times/\Q_N^{\times p}$ we have that
\begin{equation*}
\iota(\mathcal{G}_{-i}) = N^{B_{1,\cyclo^{i}}} S_{i}^{-1}
\end{equation*}
is a nonzero element of $\ker j$.

(Equivalently, one could instead notice that $h^1_N(\F_p(-i)) = 1$ from part \ref{cohomology_theorem_odd_regular} of Theorem \ref{cohomology_theorem_conditional}. By Remark \ref{idele_dual_to_cohomology}, this gives us that the codomain of $j$ is $1$-dimensional as well. The domain of $j$ is $\Q_N^\times/\Q_N^{\times p}$ which is $2$-dimensional, which shows that $j$ has a nontrivial kernel.)

Now we need to check that $j$ is nonzero, which by Remark \ref{idele_dual_to_cohomology} is equivalent to showing that the dual map
\begin{equation*}
    \res_N: H^1_N(\F_p(-i)) \to H^1(G_{\Q_N}, \F_p)
\end{equation*}
is nonzero.

This must be the case, as the class in $H^1_N(\F_p(-i))$ is unramified away from $N$, and thus must be ramified at $N$ as $(p,-i)$ is a regular pair. In particular, it is not split at $N$.

To finish, let $c$ be a generator of $H^1_N(\F_p(-i))$. This gives a $\psi_c: C_{-i} \to \F_p$ as in the discussion after the statement of Theorem \ref{odd_invariant_theorem}. Recall also from that discussion that $c \in H^1_\Sigma(\F_p(-i))$ if and only if the kernel of $\psi_{c, N} = \psi_c \circ j$ contains the element $N \in \Q_N^\times / \Q_N^{\times p}$.

Because $\psi_c$ is an isomorphism, we have $\ker \psi_{c,N} = \ker j$ and thus the local behavior of $c$ is completely determined by $\ker j$. By the above, $\ker j$ is spanned by
\begin{equation*}
\iota(\mathcal{G}_{-i}) = N^{B_{1,\cyclo^{i}}} S_{i}^{-1}
\end{equation*}
and thus contains $N$ if and only if $S_i$ is a $p$th power in $\F_N^\times$.
\end{proof}

\subsection{Relationship between \texorpdfstring{$S_i$}{S\textunderscore i}, \texorpdfstring{$M_i$}{M\textunderscore i}, and \texorpdfstring{$C$}{C}}\label{sec_invariant_relationship}

We begin by showing that our $S_i$ is a $p$th power in $\F_N^\times$ if and only if Lecouturier's $M_{i}$ is. Recall from Section \ref{sec_results} that for odd $1 \leq i \leq p-4$, $M_i$ is defined by
\begin{equation*}
M_{i} = \prod_{k=1}^{N-1} \prod_{a=1}^{k-1} k^{a^{i}}.
\end{equation*}

\begin{lemma}\label{our_invariant_is_lecouturiers}
As elements of $\F_{N}^{\times}/\F_{N}^{\times p}$, $S_{i}^{-1} = M_{i}$.
\end{lemma}
\begin{proof}
All equalities in this proof take place in $\F_{N}^{\times}/\F_{N}^{\times p}$.
In Lemma 4.3 of \cite{lecouturier}, Lecouturier proves that
\begin{align*}
M_{i}   & = \prod_{k=1}^{p-1} \Gamma_{N}(k/p)^{k^{i}}
\end{align*}
where $\Gamma_{N}$ denotes the $N$-adic Gamma function (see below for a summary of the properties of this function, and Chapter IV.2 of \cite{koblitz} for the detailed construction).
Using that $\frac{k}{p} \equiv M(p-k) + 1 \bmod{N}$, the Gamma functions can be replaced by factorials
\begin{align*}
M_{i}       & = \prod_{k=1}^{p-1} ((M(p-k))!)^{k^{i}} \\*
            & = \prod_{k=1}^{p-1} ((Mk)!)^{-k^{i}} \\*
            & = S_{i}^{-1}
\end{align*}
where the second step follows by changing variables from $k$ to $p-k$ and discarding $p$-th powers.
\end{proof}

Theorem \ref{odd_invariant_theorem} establishes that under a regularity assumption, $H^{1}_{\Sigma}(\F_{p}(-i))$ is nonzero if and only if $S_{i}$ is a $p$th power for odd $i \not\equiv -1 \bmod{p-1}$.
A similar relationship was known to Wake--Wang-Erickson in the case $i \equiv 1 \bmod{p-1}$; see Theorem 12.5.1 of \cite{wake_wang-erickson_mazur_eisenstein_ideal}.

However, these results are not stated in terms of $S_{1}$, but rather in terms of Merel's number
\begin{equation*}
C = \prod_{k=1}^{(N-1)/2} k^{k}.
\end{equation*}
Theorem 1.3, (ii) of \cite{calegari_emerton_ramification} states that if $r_{K} = 1$ then $C$ is not a $p$th power mod $N$.
Similarly, Proposition \ref{wake_wang_erickson_theorem} and Theorem \ref{odd_invariant_theorem} together imply that if $r_{K} = 1$ then $S_{1}$ is not a $p$th power mod $N$.
Thus one might expect that the quantities $C$ and $S_{1}$ can be related in $\F_{N}^{\times}/\F_{N}^{\times p}$.
The goal of this section is to prove this statement; to do so we will introduce another family of quantities related to both $C$ and the $S_{i}$.

Let
\begin{equation*}
A_{m} = \prod_{k=1}^{N-1} k^{k^{m}}.
\end{equation*}
In Proposition 1.2 of \cite{lecouturier}, Lecouturier proves that
\begin{equation*}
C = A_{2}^{-3/4} \text{ in } \F_{N}^{\times}/\F_{N}^{\times p}.
\end{equation*}

To relate the $A_{m}$ to the $S_{i}$ we will use the $N$-adic Gamma function, the relevant properties of which are:
\begin{itemize}
\item $\Gamma_{N}: \Z_{N} \to \Z_{N}^{\times}$ is a continuous function, constructed by extending the function
\begin{equation*}
\Gamma_{N}(x) = (-1)^{x} \prod_{0 < j < x, N \nmid j} j
\end{equation*}
defined for positive integers $x$ by continuity to all of $\Z_{N}$.
\item For an integer $0 < x < N$, we have $\Gamma_{N}(x) = (-1)^{x} (x-1)!$.
\item If $x \equiv y \bmod{N}$, then $\Gamma_{N}(x) \equiv \Gamma_{N}(y) \bmod{N}$.
\item If $x + r$ is not divisible by $N$ for $0 \leq r \leq M-1$ where $M = \frac{N-1}{p}$, then
\begin{equation*}
\prod_{r = 0}^{M-1} (x + r) = (-1)^{M} \frac{\Gamma_{N}(M+x)}{\Gamma_{N}(x)}.
\end{equation*}
\end{itemize}
See Chapter IV.2 of \cite{koblitz} for the construction of $\Gamma_{N}$.

\begin{proposition}\label{relate_with_higher_merels}
Suppose that $0 < m < p-1$.
Then
\begin{equation*}
A_{m} = \prod_{j=1}^{m-1} S_{j}^{(-1)^{j}\binom{m}{j}} \text{ in } \F_{N}^{\times}/\F_{N}^{\times p}.
\end{equation*}
\end{proposition}
\begin{proof}
All equalities in this proof are in $\F_{N}^{\times}/\F_{N}^{\times p}$.
We start by reindexing the product in the definition of $A_{m}$
\begin{equation*}
A_{m} = \prod_{k=1}^{p-1} \prod_{r=0}^{M-1} (k + pr)^{(k + pr)^{m}}.
\end{equation*}
After removing $p$th powers from the exponent and factoring out a $p$th power of $p$ we have that
\begin{align*}
A_{m}       & = \prod_{k=1}^{p-1}\prod_{r=0}^{M-1} \left(\frac{k}{p} + r\right)^{k^{m}} \\*
            & = \prod_{k=1}^{p-1} \left((-1)^{M} \frac{\Gamma_{N}(M + k/p)}{\Gamma_{N}(k/p)}\right)^{k^{m}}
\end{align*}
where the second step follows from the last listed property of the $N$-adic Gamma function.
Aligning terms using by a ``telescoping series'' argument gives that
\begin{equation*}
A_{m} = \prod_{k=1}^{p-1} \Gamma_{N}(k/p)^{(k+1)^{m} - k^{m}}.
\end{equation*}
Using that $\frac{k}{p} \equiv M(p-k) + 1 \bmod{N}$, the Gamma functions can be replaced by factorials
\begin{align*}
A_{m}       & = \prod_{k=1}^{p-1} ((M(p-k))!)^{(k+1)^{m} - k^{m}} \\*
            & = \prod_{k=1}^{p-1}((Mk)!)^{(p-k+1)^{m} - (p-k)^{m}}
\end{align*}
where the second step follows by changing variables from $k$ to $p-k$.
Simplifying the exponent and combining terms appropriately into the $S_{i}$, this yields that
\begin{equation*}
A_{m} = \prod_{j=0}^{m-1} S_{j}^{(-1)^{j}\binom{m}{j}}. \qedhere
\end{equation*}
\end{proof}

Note that this theorem implies that
\begin{equation*}
A_{2} = S_{1}^{-2} \text{ in } \F_{N}^{\times}/\F_{N}^{\times p}
\end{equation*}
so combining this with the relationship between $C$ and $A_{2}$, we see that one of $C$, $A_{2}$, $S_{1}$, and $M_1$ is a $p$th power mod $N$ if and only if all of them are.

Proposition \ref{relate_with_higher_merels} also shows that the $S_{i}$ can be recovered from the $A_{m}$, at least as elements of $\F_{N}^{\times}/\F_{N}^{\times p}$, using inductively that $S_{1} = A_{2}^{2}$ and that
\begin{equation*}
S_{i} = \left(A_{i+1} \prod_{j=1}^{i-1} S_{j}^{(-1)^{j+1}\binom{i+1}{j}}\right)^{(-1)^{i}(i+1)}
\end{equation*}
for all $i$.

\subsection{A Criterion for \texorpdfstring{$H^{1}_{\Sigma}(\mathbf{F}_{p}(-i)) \neq 0$}{H1\textunderscore Sigma(F\textunderscore p(-i)) /= 0}, \texorpdfstring{$i$}{i} Even}\label{sec_even_invariant}

So far, the focus of this section has been on odd $i$. At this point, we turn to finding invariants that will let us compute whether or not $H^{1}_{\Sigma}(\F_{p}(-i))$ is trivial for even $i \neq 0 \bmod{p-1}$.

\begin{proposition}\label{even_H1Sigma_iff_odd_and_p}
Let $p$ be an odd prime, and let $2 \leq i \leq p-3$ be even. Suppose that $(p,1+i)$ is a regular pair. Then $H^{1}_{\Sigma}(\F_{p}(-i))$ is non-trivial if and only if both of the following are satisfied:
\begin{enumerate}
    \item $H^{1}_{\Sigma}(\F_{p}(1+i)) \neq 0$
    \item $H^{1}_{p}(\F_{p}(1+i)) \subseteq H^{1}_{\Sigma^{\ast}}(\F_{p}(1+i))$
\end{enumerate}
\end{proposition}
\begin{proof}
We see by Theorems \ref{cohomology_theorem_conditional} and \ref{cohomology_theorem_sigma} that $H^{1}_{\Sigma}(\F_{p}(-i))$ is non-trivial if and only if $H^{1}_{\Sigma^{\ast}}(\F_{p}(1+i))$ is $2$-dimensional and thus equal to $H^{1}_{S}(\F_{p}(1+i))$.
Since $H^{1}_{S}(\F_{p}(1+i))$ is spanned by the subspaces $H^{1}_{N}(\F_p(i+1))$ and $H^{1}_{p}(\F_p(i+1))$, this second condition happens if and only if we have both $H^{1}_{N}(\F_{p}(1+i)) = H^{1}_{\Sigma}(\F_{p}(1+i))$ and $H^{1}_{p}(\F_{p}(1+i)) \subseteq H^{1}_{\Sigma^{\ast}}(\F_{p}(1+i))$.
\end{proof}

Since we know how to test for $H^{1}_{\Sigma}(\F_{p}(1 + i))$ being non-trivial, we simply need to find a way of testing whether or not $H^{1}_{p}(\F_{p}(1+i)) \subseteq H^{1}_{\Sigma^{\ast}}(\F_{p}(1+i))$.

The class in $H^{1}_{p}(\F_{p}(1+i))$ is unramified at $N$, so it will land in $H^{1}_{\Sigma^{\ast}}(\F_{p}(1+i))$ if and only if it is split at $N$.
By using the inflation-restriction sequence and Kummer theory, we get that
\begin{align*}
H^{1}_{p}(\F_{p}(1+i))  & \cong H^{1}_{p}(G_{\Q(\zeta_{p})}, \F_{p}(1+i))^{\Gal(\Q(\zeta_{p})/\Q)} \\*
                        & \cong (H^{1}_{p}(G_{\Q(\zeta_{p})}, \F_{p}(1))(i))^{\Gal(\Q(\zeta_{p})/\Q)} \\
                        & \cong \left(\left(\frac{\Z[\zeta_{p}, p^{-1}]^{\times}}{\Z[\zeta_{p}, p^{-1}]^{\times p}}\right)(i)\right)^{\Gal(\Q(\zeta_{p})/\Q)} \\*
                        & \cong \left(\frac{\Z[\zeta_{p}, p^{-1}]^{\times}}{\Z[\zeta_{p}, p^{-1}]^{\times p}}\right)^{\cyclo^{-i}}
\end{align*}
where we have used that the restriction map is an isomorphism as the order of $\Gal(\Q(\zeta_{p})/\Q)$ is prime to $p$.
In other words, the extension of $\Q$ defined by a class in $H^{1}_{p}(\F_{p}(1+i))$ is always of the form $\Q(\zeta_{p}, a^{1/p})$, where $a \in \Z[\zeta_{p}, p^{-1}]^{\times}$ and
\begin{equation*}
\sigma(a) = a^{\cyclo^{-i}(\sigma)} \text{ modulo $p$th powers}
\end{equation*}
for all $\sigma \in \Gal(\Q(\zeta_{p})/\Q)$.
Note that given such an element, all of its Galois conjugates are also Kummer generators of the same extension.
Thus it suffices to find such a Kummer generator $a$ (which is independent of $N$), and then use that the cohomology class spanning $H^{1}_{p}(\F_{p}(1+i))$ is trivial at $N$ if and only if the Kummer generator is a $p$th power in $\Q_{N}^{\times}$, which happens if and only if the Kummer generator is a $p$th power mod $N$.

The minimal polynomials of such elements can be computed using a computer algebra system.
This was done using SageMath \cite{sagemath} for $p=5$ and $p=7$. The SageMath code is available on the second author's personal website.

\begin{theorem}\label{even_H1Sigma_theorem}
We have:
\begin{enumerate}
\item Suppose $p = 5$. Then $H^{1}_{\Sigma}(\F_{p}(-2))$ is nonzero if and only both $S_{1}$ and the roots of $x^{2} + x - 1$ are $5$th powers in $\F_{N}^{\times}$.

\item Suppose $p = 7$. Then
\begin{enumerate}
\item $H^{1}_{\Sigma}(\F_{p}(-2))$ is nonzero if and only both $S_{3}$ and the roots of
\begin{equation*}
    x^{3} + 41x^{2} + 54x + 1
\end{equation*}
are $7$th powers in $\F_{N}^{\times}$.
\item $H^{1}_{\Sigma}(\F_{p}(-4))$ is nonzero if and only if both $S_{1}$ and the roots of
\begin{equation*}
    x^{3} - 25x^{2} + 31x + 1
\end{equation*}
are $7$th powers in $\F_{N}^{\times}$.
\end{enumerate}
\end{enumerate}
\end{theorem}

\begin{remark}
The polynomials in the theorem above are not unique.
One could use any other polynomial whose roots generate the same 1-dimensional subspace of
\begin{equation*}
\left(\frac{\Z[\zeta_{p}, p^{-1}]^{\times}}{\Z[\zeta_{p}, p^{-1}]^{\times p}}\right)^{\cyclo^{-i}}.
\end{equation*}
\end{remark}

\section{Specific Primes}\label{sec_specific_primes}

We now apply the results of the previous sections to the specific cases $p = 3$, $5$, and $7$.
For $p = 3$ the situation is quite straightforward, as the results of Section \ref{sec_big_general} imply that $r_{K} = 1$.
For $p = 5$ we show that the inequality of Theorem \ref{rank_equals_h1Sigma} is always an equality, which then determines $r_{K}$ solely in terms of the dimensions $h^{1}_{\Sigma}(\F_{p}(-1))$ and $h^{1}_{\Sigma}(\F_{p}(-2))$.
A similar argument applied to the case $p = 7$ proves the converse to Theorem \ref{intro_converse_theorem}. 

Throughout this section we will often use without reference the results of Section \ref{sec_cohomology_groups_of_characters} on the dimensions of various Selmer subgroups of $H^{1}_{S}(\F_{p}(-i))$.

\subsection{\texorpdfstring{$p = 3$}{p = 3}}

If $p = 3$, Theorem 4.5 of \cite{gerth_3-class_groups} implies that $r_{K} = 1$.
In other words, if $N \equiv 1 \bmod{3}$, the only degree 3 unramified extension of $K = \Q(N^{1/3})$ is the genus field.

The results of Section \ref{pfsetup} recover this result in the following way.
Lemmas \ref{bohica_1} and \ref{bohica_2} imply that the type $m$ of any unramified extension $E/K$ must satisfy $m \leq p - 2 = 1$.
Lemma \ref{type_1_iff_genus_field} shows that the only extension of type 1 is the genus field $K(\zeta_{N}^{(p)})$.
This proves the following theorem.

\begin{theorem}\label{3_theorem}
Let $p = 3$.
Then $r_{K} = 1$.
\end{theorem}

\subsection{\texorpdfstring{$p = 5$}{p = 5}}\label{5_subsection}

In the case $p = 5$, we prove the following refined version of Theorem \ref{rank_equals_h1Sigma}.

\begin{theorem}\label{5_theorem}
Let $p = 5$.
Then $r_{K} = 1 + h^{1}_{\Sigma}(\F_{p}(-1)) + h^{1}_{\Sigma}(\F_{p}(-2))$.
\end{theorem}
\begin{proof}
We know from Theorem \ref{rank_equals_h1Sigma} that
\begin{equation*}
r_{K} = 1 + h^{1}_{\Sigma}(\bigrep{p-4}{2}) = 1 + h^{1}_{\Sigma}(V(-2)).
\end{equation*}
Thus to prove the theorem it suffices to show that
\begin{equation*}
h^{1}_{\Sigma}(V(-2)) = h^{1}_{\Sigma}(\F_{p}(-1)) + h^{1}_{\Sigma}(\F_{p}(-2)).
\end{equation*}
In light of the short exact sequence of $G_{\Q, S}$-modules
\begin{equation*}
0 \to \F_{p}(-1) \to V(-2) \to \F_{p}(-2) \to 0
\end{equation*}
and the fact that $H^{1}_{\Sigma}(\F_{p}(-1)) \subseteq H^{1}_{\Sigma}(V(-2))$ by the associated long exact sequence in $G_{\Q, S}$-cohomology, it will suffice to prove that any class in $H^{1}_{\Sigma}(\F_{p}(-2))$ lifts to $H^{1}_{\Sigma}(V(-2))$, as in the discussion at the beginning of Section \ref{sec_lifting_selmer_classes}.

Suppose $h^{1}_{\Sigma}(\F_{p}(-2)) \neq 0$, and hence also $h^{1}_{\Sigma}(\F_{p}(-1)) \neq 0$ by Corollary \ref{even_H1Sigma_implies_odd_H1Sigma}.
We satisfy the conditions of Theorem \ref{general_middle_lift}, as $\frac{p-1}{2} = 2 \equiv -2 \bmod{4}$, so we know that the class spanning $H^{1}_{\Sigma}(\F_{p}(-2))$ lifts to a class in $H^{1}_{\Sigma^{\ast}}(V(-2))$.
Since we also have
\begin{equation*}
    h^{1}_{\Sigma^{\ast}}(\F_{p}(-1)) = 2 > 1 = h^{1}_{\Sigma}(\F_{p}(-1))
\end{equation*}
in this situation by Theorem \ref{cohomology_theorem_sigma}, we may apply Lemma \ref{fix_at_p_no_change_at_N} to choose a lift which in fact is in $H^{1}_{\Sigma}(V(-2))$.
\end{proof}

Combining this theorem with the results of Section \ref{sec_effective_criteria} proves Theorem \ref{intro_5_theorem}:

\begin{proof}[Proof of Theorem \ref{intro_5_theorem}]
Since each $h^{1}_{\Sigma}(\F_{p}(-i))$ is at most 1, we obtain the bound $r_{K} \leq 3$.
We know that $r_{K} \geq 2$ if and only if $S_{1} = \prod_{k=1}^{p-1} ((Mk)!)^{k}$ is a $5$th power in $\F_{N}^{\times}$, as Theorem \ref{odd_invariant_theorem} proves that $h^{1}_{\Sigma}(\F_{p}(-1)) = 1$ if and only if $S_{1}$ is a $5$th power, and
further, $r_{K} = 3$ if and only if $h^{1}_{\Sigma}(\F_{p}(-1)) = h^{1}_{\Sigma}(\F_{p}(-2)) = 1$, which by Theorems \ref{odd_invariant_theorem} and \ref{even_H1Sigma_theorem} happens if and only if both $S_{1}$ and $\frac{\sqrt{5} - 1}{2}$ are $5$th powers in $\F_{N}^{\times}$.
\end{proof}

See Appendix \ref{sec_5_data} for data on how often each of the three possible cases $r_{K} = 1$, $2$, or $3$ occurs.

\subsection{\texorpdfstring{$p=7$}{p = 7}}\label{7_subsection}

When $p = 7$ it is not the case that $r_{K}$ can be determined completely by the dimensions $h^{1}_{\Sigma}(\F_{p}(-i))$.
Note that when $p = 7$ the possible groups $H^{1}_{\Sigma}(\F_{p}(-i))$ that may arise are those for $i \in \{1, 2, 3, 4\}$.
When discussing the possible cases we will indicate the dimensions of these $H^{1}_{\Sigma}(\F_{p}(-i))$ by a binary string of length $4$; so for example $1000$ is used to indicate $h^{1}_{\Sigma}(\F_{p}(-1)) = 1$ and $h^{1}_{\Sigma}(\F_{p}(-i)) = 0$ for $i \in \{2, 3, 4\}$.
By Corollary \ref{even_H1Sigma_implies_odd_H1Sigma}, not all binary strings of length $4$ may occur of as the dimensions of the $h^{1}_{\Sigma}(\F_{p}(-i))$; if $h^{1}_{\Sigma}(\F_{p}(-i)) = 1$ for $i = 2$ or $4$, we must have that $h^{1}_{\Sigma}(\F_{p}(-i)) = 1$ for $i = 3$ or $1$, respectively.

\begin{theorem}\label{7_theorem}
Let $p = 7$.
Then $r_{K} \geq 2$ if and only if at least one of $H^{1}_{\Sigma}(\F_{p}(-1))$ or $H^{1}_{\Sigma}(\F_{p}(-3))$ is nonzero.
\end{theorem}
\begin{proof}
By the upper bound given in Proposition \ref{h1Sigma_filtration_bound}, if $r_{K} \geq 2$ we must have at least one of the $h^{1}_{\Sigma}(\F_{p}(-i)) \neq 0$.
Corollary \ref{even_H1Sigma_implies_odd_H1Sigma} shows that if any of the $h^{1}_{\Sigma}(\F_{p}(-i))$ is nonzero we must have that $h^{1}_{\Sigma}(\F_{p}(-i)) = 1$ for $i = 1$ or $3$.
This proves the ``only if'' direction.

We have established in Proposition \ref{wake_wang_erickson_theorem} that $h^{1}_{\Sigma}(\F_{p}(-1)) = 1 \implies r_{K} \geq 2$.
Thus it remains to show that when $h^{1}_{\Sigma}(\F_{p}(-1)) = 0$ and $h^{1}_{\Sigma}(\F_{p}(-3)) = 1$ we have $r_{K} \geq 2$.
There are two possible cases, based on whether or not $h^{1}_{\Sigma}(\F_{p}(-2)) = 0$.

Case 1: The dimensions of the $H^{1}_{\Sigma}(\F_{p}(-i))$ are $0110$.
In this situation, we have by Theorems \ref{cohomology_theorem_conditional} and \ref{cohomology_theorem_sigma} that
\begin{equation*}
2 = h^{1}_{S}(\F_{p}(-1)) > 1 = h^{1}_{\Sigma^{\ast}}(\F_{p}(-1)) > 0 = h^{1}_{\Sigma}(\F_{p}(-1)),
\end{equation*}
hence we may apply Lemma \ref{fix_at_both} to show that the class spanning $H^{1}_{\Sigma}(\F_{p}(-2))$ lifts to $H^{1}_{\Sigma}(V(-2))$.
Since $V(-2)$ is the $2$-dimensional subrepresentation of 
\begin{equation*}
\bigrep{p-4}{2} = \bigrep{3}{-4},
\end{equation*}
we have by Theorem \ref{rank_equals_h1Sigma} and the discussion at the start of Section \ref{pfpart3} that
\begin{align*}
r_{K}   & = 1 + h^{1}_{\Sigma}(\bigrep{3}{-4}) \\*
        & \geq 1 + h^{1}_{\Sigma}(V(-2)) \\*
        & \geq 1 + 1 = 2.
\end{align*}

Case 2: The dimensions of the $H^{1}_{\Sigma}(\F_{p}(-i))$ are $0010$.
The conditions of Theorem \ref{general_middle_lift} are satisfied here, so a class spanning $H^{1}_{\Sigma}(\F_{p}(-3))$ lifts to a class in $H^{1}_{\Sigma^{\ast}}(V(-3))$.
Using that
\begin{equation*}
1 = h^{1}_{\Sigma^{\ast}}(\F_{p}(-2)) > 0 = h^{1}_{\Sigma}(\F_{p}(-2))
\end{equation*}
by Theorem \ref{cohomology_theorem_sigma}, we may apply Lemma \ref{fix_at_p_no_change_at_N} to show that there is in fact a lift to $H^{1}_{\Sigma}(V(-3))$.
Now, using that again that
\begin{equation*}
2 = h^{1}_{S}(\F_{p}(-1)) > 1 = h^{1}_{\Sigma^{\ast}}(\F_{p}(-1)) > 0 = h^{1}_{\Sigma}(\F_{p}(-1)),
\end{equation*}
we apply Lemma \ref{fix_at_both} to show that the class in $H^{1}_{\Sigma}(V(-3))$ lifts to a class in $H^{1}_{\Sigma}(\bigrep{2}{-3})$.
Since $\bigrep{2}{-3}$ is the $3$-dimensional subrepresentation of $\bigrep{3}{-4}$, we have again by Theorem \ref{rank_equals_h1Sigma} and the discussion in Section \ref{pfpart3} that
\begin{align*}
r_{K}   & = 1 + h^{1}_{\Sigma}(\bigrep{3}{-4}) \\*
        & \geq 1 + h^{1}_{\Sigma}(\bigrep{2}{-3}) \\*
        & \geq 1 + 1 = 2. \qedhere
\end{align*}
\end{proof}

Theorem \ref{intro_7_theorem} follows by combining this result and Theorem \ref{odd_invariant_theorem}: the dimensions $h^{1}_{\Sigma}(\F_{p}(-1))$ and $h^{1}_{\Sigma}(\F_{p}(-3))$ are nonzero if and only if, respectively, $S_{1}$ and $S_{3}$ are $7$th powers in $\F_{N}^{\times}$.

We have upper and lower bounds on $r_{K}$ by Theorem \ref{rank_equals_h1Sigma}, and we may interpret Theorem \ref{7_theorem} as improving the lower bound to 
\begin{equation*}
1 + \text{max}\{h^{1}_{\Sigma}(\F_{p}(-1)), h^{1}_{\Sigma}(\F_{p}(-3))\} \leq r_{K} \leq 1 + \sum_{i=1}^{4} h^{1}_{\Sigma}(\F_{p}(-i)).
\end{equation*}
These bounds are optimal, in the sense that for a given binary string of dimensions $h^{1}_{\Sigma}(\F_{p}(-i))$ there exist $N \equiv 1 \bmod{7}$ for which the corresponding $r_{K}$ witness all possible values between the upper and lower bounds.
See Appendix \ref{sec_7_data} for data on the distribution of N among values for the $h^{1}_{\Sigma}(\F_{p}(-i))$ and $r_{K}$.

We turn now to a study of the possibilities that may occur when $r_{K}$ does not achieve the upper bound of Theorem \ref{rank_equals_h1Sigma}.
We say that a class $a_{i} \in H^{1}_{\Sigma}(\F_{p}(-i))$ ``contributes to $r_{K}$'' if $a_{i}$ lifts all the way to $H^{1}_{\Sigma}(\bigrep{i-1}{-i})$, which is a subset of $H^{1}_{\Sigma}(\bigrep{3}{-4})$.

\begin{remark}
If $r_{K} < 1 + \sum_{i=1}^{4} h^{1}_{\Sigma}(\F_{p}(-i))$, it is not always possible to determine using the dimensions $h^{1}_{\Sigma}(\F_{p}(-i))$ which class $a_{i} \in H^{1}_{\Sigma}(\F_{p}(-i))$ is failing to contribute to $r_{K}$.

For example, suppose that $r_{K} = 3$ and the dimensions $h^{1}_{\Sigma}(\F_{p}(-i))$ are $1011$.
It must be the case that one of $a_{3} \in H^{1}_{\Sigma}(\F_{p}(-3))$ and $a_{4} \in H^{1}_{\Sigma}(\F_{p}(-4))$ is contributing to $r_{K}$ and the other is failing to.
However, the conditions of Lemma \ref{fix_at_both} are not satisfied in this situation as $H^{1}_{S}(\F_{p}(-1)) = H^{1}_{\Sigma^{\ast}}(\F_{p}(-1))$, so the results of Section \ref{sec_lifting_selmer_classes} are not strong enough to show that either class always contributes to $r_{K}$.
\end{remark}

When a failure to contribute to $r_{K}$ can be tracked down to a specific class $a_{i} \in H^{1}_{\Sigma}(\F_{p}(-i))$ there are two aspects of its failure to contribute which may be considered.
First, there is the stage of lifting at which the failure occurs: there is a $k \geq 1$ such that $a_{i}$ lifts to $H^{1}_{\Sigma}(\bigrep{k-1}{-i})$ but not one step further to $H^{1}_{\Sigma}(\bigrep{k}{-i})$.
Second, there is the type of failure which occurs at this $k$th stage.
The class $a_{i}$ always lifts to $H^{1}_{S}(\bigrep{k}{-i})$ but it could be the case that:
\begin{enumerate}
\item\label{fail_at_p} No lift to $H^{1}_{S}(\bigrep{k}{-i})$ is split at $p$;
\item\label{fail_at_N} No lift to $H^{1}_{S}(\bigrep{k}{-i})$ vanishes when restricted to $K_{N}$;
\item\label{fail_at_both} There are lifts that satisfy the local condition at $p$ or at $N$, but no lift satisfies both local conditions simultaneously.
\end{enumerate}

In some cases it is possible to determine at which stage and which type of failure to lift is occurring, by an analysis of the dimensions of the subgroups of the $H^{1}_{S}(\F_{p}(-i))$ using the results of Section \ref{sec_cohomology_groups_of_characters}.
Examples of situations witnessing each of the above types of local failure are collected below.
In each example, the class $a_{3} \in H^{1}_{\Sigma}(\F_{p}(-3))$ fails to contribute to $r_{K}$.
Note that by Theorem \ref{general_middle_lift} there is a lift of $a_{3}$ to $H^{1}_{\Sigma^{\ast}}(V(-3))$, and since the set of all lifts is a coset of $H^{1}_{S}(\F_{p}(-2)) = H^{1}_{\Sigma^{\ast}}(\F_{p}(-2))$, we in fact have that every lift of $a_{3}$ is in $H^{1}_{\Sigma^{\ast}}(V(-3))$.

\begin{example}
Suppose that the dimensions $h^{1}_{\Sigma}(\F_{p}(-i))$ are $0110$ and $r_{K} = 2$.
The proof of Theorem \ref{7_theorem} showed that the class in $H^{1}_{\Sigma}(\F_{p}(-2))$ contributes to $r_{K}$, so it must be the case that $a_{3} \in H^{1}_{\Sigma}(\F_{p}(-3))$ does not lift to $H^{1}_{\Sigma}(\bigrep{2}{-3})$.

Suppose that $a_{3}$ lifts to $H^{1}_{\Sigma}(V(-3))$.
Then Lemma \ref{fix_at_both} would apply as 
\begin{align*}
h^{1}_{S}(\F_{p}(-1))   & = 2 \\*
h^{1}_{\Sigma^{\ast}}(\F_{p}(-1)) & = 1 + h^{1}_{\Sigma}(\F_{p}(-4)) = 1 \\*
h^{1}_{\Sigma}(\F_{p}(-1)) & = 0,
\end{align*} so there would exist a lift of $a_{3}$ to $H^{1}_{\Sigma}(\bigrep{2}{-3})$.
Since our assumption that $r_{K} = 2$ implies that $a_{3}$ does not lift to $H^{1}_{\Sigma}(\bigrep{2}{-3})$, it must be the case that $a_{3}$ does not lift to $H^{1}_{\Sigma}(V(-3))$.

We know that every lift of $a_{3}$ to $H^{1}_{S}(V(-3))$ is in $H^{1}_{\Sigma^{\ast}}(V(-3))$, thus it must be the case that no lift is split at $p$.
\end{example}

\begin{example}
Suppose that the dimensions $h^{1}_{\Sigma}(\F_{p}(-i))$ are $1011$ and $r_{K} = 2$.
As in the proof of Theorem \ref{7_theorem}, Theorem \ref{general_middle_lift} shows that $a_{3}$ lifts to $H^{1}_{\Sigma^{\ast}}(V(-3))$, and then Lemma \ref{fix_at_p_no_change_at_N} shows that there is a modification of this lift which is in $H^{1}_{\Sigma}(V(-3))$.

Suppose that there is a lift of this class to $H^{1}_{\Sigma^{\ast}}(\bigrep{2}{-3})$.
Then Lemma \ref{fix_at_p_no_change_at_N} would apply to show that there is a lift to $H^{1}_{\Sigma}(\bigrep{2}{-3})$, as
\begin{align*}
h^{1}_{\Sigma^{\ast}}(\F_{p}(-1))   & = 1 + h^{1}_{\Sigma}(\F_{p}(-4)) = 2 \\*
h^{1}_{\Sigma}(\F_{p}(-1))          & = 1.
\end{align*}
Our assumption that $r_{K} = 2$ means that $a_{3}$ does not contribute to $r_{K}$, hence there cannot be a lift of $a_{3}$ to $H^{1}_{\Sigma^{\ast}}(\bigrep{2}{-3})$.
\end{example}

\begin{example}
Suppose that the dimensions $h^{1}_{\Sigma}(\F_{p}(-i))$ are $1010$ and $r_{K} = 2$.
As in the previous example, $a_{3}$ lifts to $H^{1}_{\Sigma}(V(-3))$.

We have that
\begin{equation*}
2 = h^{1}_{S}(\F_{p}(-1)) > 1 = h^{1}_{\Sigma^{\ast}}(\F_{p}(-1)) = h^{1}_{N}(\F_{p}(-1)) = h^{1}_{\Sigma}(\F_{p}(-1)),
\end{equation*}
hence we may apply Lemmas \ref{fix_at_N} and \ref{fix_at_p} to show that there are lifts of $a_{3}$ to both $H^{1}_{\Sigma^{\ast}}(\bigrep{2}{-3})$ and $H^{1}_{N}(\bigrep{2}{-3})$, respectively.

However, we know that $a_{3}$ fails to contribute to $r_{K}$, so it must be the case that no lift of $a_{3}$ is in 
\begin{equation*}
H^{1}_{\Sigma}(\bigrep{2}{-3}) = H^{1}_{\Sigma^{\ast}}(\bigrep{2}{-3}) \cap H^{1}_{N}(\bigrep{2}{-3}).
\end{equation*}
In other words there is no lift of $a_{3}$ which satisfies the conditions at $p$ and $N$ simultaneously, despite there being lifts which satisfy each condition individually.
\end{example}

\appendix
\renewcommand{\sectionname}{Appendix}
\section{Data for \texorpdfstring{$p = 5, 7$}{p = 5, 7}}\label{sec_data}

All computations in this section were performed using PARI/GP \cite{pari} and SageMath \cite{sagemath}.
The computation of ranks of class groups when $p = 7$ used PARI/GP's built-in algorithms for computing class groups of number fields, which assume GRH to optimize computation.
Thus the ranks computed when $p = 7$ in all cases other than those where the rank is determined by the numbers $h^1_\Sigma(\F_p(-i))$ as in Section \ref{7_subsection} are conditional on GRH.

The SageMath code for computing the numbers $h^1_\Sigma(\F_p(-i))$ for $p=7$ via the methods in Section \ref{sec_effective_criteria} is available on the second author's website. The data in Table \ref{7_dims_table} took approximately 10 hours to gather using a low-range commercial processor.

\subsection{\texorpdfstring{$p = 5$}{p = 5}}\label{sec_5_data}

For primes $N \equiv 1 \bmod{5}$, $N \leq 20{,}000{,}000$ we computed the dimensions $h^{1}_{\Sigma}(\F_{p}(-1))$ and $h^{1}_{\Sigma}(\F_{p}(-2))$ using the results of Section \ref{sec_effective_criteria}.
For each $N$ there are three possible sets of dimensions: both are $0$, $h^{1}_{\Sigma}(\F_{p}(-1)) = 1$ and $h^{1}_{\Sigma}(\F_{p}(-2)) = 0$, and both are $1$; as in Section \ref{7_subsection} these are notated by a binary string of length $2$ ($00$, $10$, and $11$).
Note that by Theorem \ref{5_theorem} the dimensions $h^{1}_{\Sigma}(\F_{p}(-i))$ completely determine the rank $r_{K}$.
There are 317,587 such primes $N$, and their distribution among the three possible cases is given in Table \ref{5_data_table}.

\begin{table}[htb]
\centering
\begin{tabular}{|c|c|d{0}|} \hline
Dimensions & $r_{K}$ & \text{Number of $N$} \\ \hline
00 & 1 & 253,234 \\ \hline
10 & 2 & 51,613 \\ \hline
11 & 3 & 12,740 \\ \hline
Total & & 317,587 \\ \hline
\end{tabular}
\caption{Data for $p = 5$.}\label{5_data_table}
\end{table}

From this we see that $20.26\%$ of $N$ in this range have $r_K \geq 2$, and of those $N$, $19.80\%$ of $N$ have $r_K \geq 3$. We expect that the quantities $M_1$ and $\frac{\sqrt{5}-1}{2}$ should be ``uniformly distributed'' in $\Z/5\Z \cong \F_N^\times/\F_N^{\times 5}$, meaning that they are $5$th powers for a set of primes of density $\frac{1}{5}$ in the primes $N \equiv 1 \bmod{5}$. This would imply that $r_K \geq 2$ for $\frac15$ of those primes and that $r_K = 3$ for $\frac{1}{25}$ of primes $N \equiv 1 \bmod{5}$, which is suggested by the data.

\subsection{\texorpdfstring{$p = 7$}{p = 7}}\label{sec_7_data}

For primes $N \equiv 1 \bmod{7}$, $N \leq 100{,}000{,}000$, we computed the dimensions $h^{1}_{\Sigma}(\F_{p}(-i))$ for $i = 1, 2, 3, 4$ using the results of Section \ref{sec_effective_criteria}.
There are 960,023 such primes $N$, and their distribution among the possible cases is given in Table \ref{7_dims_table}.

\begin{table}[htb]
\centering
\begin{tabular}{|c|d{0}|} \hline
Dimensions & \text{Number of $N$} \\ \hline
0000 & 705,575 \\ \hline
1000 & 99,649 \\ \hline
0010 & 101,126 \\ \hline
1010 & 15,057 \\ \hline
1001 & 16,610 \\ \hline
0110 & 16,580 \\ \hline
1011 & 2,249 \\ \hline
1110 & 2,546 \\ \hline
1111 & 631 \\ \hline
Total & 960,023 \\ \hline
\end{tabular}
\caption{Dimensions of the $H^{1}_{\Sigma}(\F_{p}(-i))$, $p = 7$ and $N \leq 100{,}000{,}000$.}\label{7_dims_table}
\end{table}

For primes $N \equiv 1 \bmod{7}$ and $N \leq 20{,}000{,}000$, we computed the rank $r_{K}$ (which is not determined completely by the $h^{1}_{\Sigma}(\F_{p}(-i))$ in this case).
There are 211,766 such primes $N$, and their distribution between possible ranks $1 \leq r_{K} \leq 5$ and dimensions $h^{1}_{\Sigma}(\F_{p}(-i))$ are given in Table \ref{7_ranks_table}.
The empty cells in Table \ref{7_ranks_table} are cases that are shown to never occur in Section \ref{7_subsection}; in particular every case not ruled out in Section \ref{7_subsection} does occur.

\begin{table}[htb]
\centering
\begin{tabular}{|c|d{0}|d{0}|d{0}|d{0}|d{0}|d{0}|} \hline
Dimensions & r_{K} = 1 & r_{K} = 2 & r_{K} = 3 & r_{K} = 4 & r_{K} = 5 & \text{Total} \\ \hline
0000 & 155,691 & & & & & 155,691 \\ \hline
1000 & & 21,975 & & & & 21,975 \\ \hline
0010 & & 22,201 & & & & 22,201 \\ \hline
1010 & & 2,925 & 478 & & & 3,403 \\ \hline
1001 & & 3,110 & 487 & & & 3,597 \\ \hline
0110 & & 3,133 & 499 & & & 3,632 \\ \hline
1011 & & 444 & 50 & 10 & & 504 \\ \hline
1110 & & 407 & 170 & 2 & & 579 \\ \hline
1111 & & 130 & 46 & 6 & 2 & 184 \\ \hline
Total & 155,691 & 54,325 & 1,730 & 18 & 2 & 211,766 \\ \hline
\end{tabular}
\caption{Ranks $r_{K}$ and dimensions of the $H^{1}_{\Sigma}(\F_{p}(-i))$, $p = 7$ and $N \leq 20{,}000{,}000$.}\label{7_ranks_table}
\end{table}

As in the case $p = 5$, one might expect that $H^{1}_{\Sigma}(\F_{p}(-1))$ and $H^{1}_{\Sigma}(\F_{p}(-3))$ are each nonzero for $\frac{1}{7}$ of primes $N \equiv 1 \bmod{7}$.
Indeed, the data supports this guess, with $14.24 \%$ of the $N$ tested having $H^{1}_{\Sigma}(\F_{p}(-1))$ nonzero, and $14.39 \%$ of the $N$ tested having $H^{1}_{\Sigma}(\F_{p}(-3))$ nonzero.

One might also expect that $\frac{1}{7}$ of primes with  $H^{1}_{\Sigma}(\F_{p}(-1))$ nonzero also have $H^{1}_{\Sigma}(\F_{p}(-4))$ nonzero, as this just rests on whether or not the roots of a fixed polynomial are $7$th powers mod $N$; this holds for $14.25 \%$ of the $N$ tested.
Similarly, $H^{1}_{\Sigma}(\F_{p}(-2))$ is nonzero for $14.30 \%$ of the primes tested for which $H^{1}_{\Sigma}(\F_{p}(-3))$ is nonzero.

\bibliographystyle{amsplain}

\providecommand{\bysame}{\leavevmode\hbox to3em{\hrulefill}\thinspace}
\providecommand{\MR}{\relax\ifhmode\unskip\space\fi MR }
\providecommand{\MRhref}[2]{%
  \href{http://www.ams.org/mathscinet-getitem?mr=#1}{#2}
}
\providecommand{\href}[2]{#2}

\end{document}